\documentclass[a4paper,twoside,11pt]{amsart}

\usepackage{amsfonts,amssymb,amscd,amsmath,enumerate}
\usepackage[greek, french, english]{babel}
\usepackage{graphicx}
\usepackage[all]{xy}
\usepackage{enumerate}
\usepackage{color}
\usepackage[T1]{fontenc}

\def\NZQ{\mathbf}               % the font for N,Z,Q,R,C
\def\NN{{\NZQ N}}
\def\QQ{{\NZQ Q}}
\def\ZZ{{\NZQ Z}}
\def\RR{{\NZQ R}}

\def\AA{{\NZQ A}}

\newtheorem{Theorem}{Theorem}[section]
\newtheorem{Lemma}[Theorem]{Lemma}
\newtheorem{Corollary}[Theorem]{Corollary}
\newtheorem{Proposition}[Theorem]{Proposition}
\newtheorem{Remark}[Theorem]{Remark}

\newtheorem{Example}[Theorem]{Example}

\let\epsilon\varepsilon
\let\phi=\varphi
\let\kappa=\varkappa

\def \a {\alpha}

\begin{document}

\title[Generating sequences]{On the construction of valuations and generating sequences on hypersurface singularities}
\author[S. D. Cutkosky, H. Mourtada, B. Teissier]{Steven Dale Cutkosky, Hussein Mourtada and Bernard Teissier}
\address{SDC: Department of Mathematics,
University of Missouri, Columbia, MO 65211, USA}
\email{cutkoskys@missouri.edu}
\address{HM: IMJ-PRG, CNRS, Universit\'e Paris Diderot, Institut de Math\'ematiques de Jussieu-Paris Rive Gauche, B\^at. Sophie Germain, Place Aur\'elie Nemours, F-75013, Paris, France.}
\email{hussein.mourtada@imj-prg.fr}
\address{BT: IMJ-PRG, CNRS, Universit\'e Paris Diderot, Institut de Math\'ematiques de Jussieu-Paris Rive Gauche, B\^at. Sophie Germain, Place Aur\'elie Nemours, F-75013, Paris, France.}
\email{bernard.teissier@imj-prg.fr}
\thanks{The first author was partially supported by NSF grant DMS-1700046. He also thanks the Universit\'e Paris Diderot and the Fondation Sciences Math\'ematiques de Paris. The second author was partially supported by The Miller scholar in residence program of the University of Missouri at Columbia.}

\pagestyle{myheadings}
\markboth{S.D. Cutkosky, H. Mourtada, B. Teissier]}{\rm Extension of valuations}
\renewcommand\rightmark{S.D. Cutkosky, H. Mourtada, B. Teissier}
\renewcommand\leftmark{Generating sequences on hypersurfaces}

\subjclass[2000]{12J20, 16W60, 14B25}
\begin{abstract}Suppose that $(K,\nu)$ is a valued field, $f(z)\in K[z]$ is a unitary and irreducible polynomial and $(L,\omega)$ is an extension of valued fields, where $L=K[z]/(f(z))$. The description of these extensions is a classical subject. We deal here with the more delicate situation where $A$ is a local domain with quotient field $K$ dominated by the valuation ring of $\nu$ and $f(z)$ is in $A[z]$, and we want to describe the extensions $\omega$ of $\nu$ to $A[z]/(f(z))$. A motivation is the problem of local uniformization in positive characteristic: assuming that the valuation $\nu$ on $A$ can be uniformized, when can $\omega$ on $A[z]/(f(z))$ also be uniformized?\par

In recent years it has appeared that this problem is closely connected to the description of the structure of the associated graded ring ${\rm gr}_\omega A[z]/(f(z))$ of $A[z]/(f(z))$ for the filtration defined by $\omega$ as an extension of the associated graded ring of $A$ for the filtration defined by $\nu$. In important special cases this description reduces the extension of local uniformization to embedded resolution of singularities of toric varieties, which is already known. This paper is devoted to this description. In particular we give an algorithm which in many cases produces a finite set of elements of $A[z]/(f(z))$ whose images in ${\rm gr}_\omega A[z]/(f(z))$ generate it as a ${\rm gr}_\nu A$-algebra as well as the relations between these images. We also work out the interactions of our method with phenomena which complicate the study of ramification and local uniformization in positive characteristic, such as the non tameness and the defect of an extension. For a valuation $\nu$ of rank one and a separable extension of valued fields $(K,\nu)\subset (L,\omega)$ as above our algorithm produces a generating sequence in a local birational extension $A_1$ of $A$ dominated by $\nu$ if and only if there is no defect. In this case, ${\rm gr}_\omega A_1[z]/(f(z))$ is a finitely presented ${\rm gr}_\nu A_1$-module.
\end{abstract}

\maketitle
\section{Introduction}
Given a field $K$ and an extension $L$ of $K$, the study of the problem of extending a valuation from $K$ to $L$ has a long history motivated in part by its close relation with ramification theory, whether in number theory or in algebraic geometry. It has an incarnation in logic, the model theory of valued fields which provides another viewpoint on ramification theory. After fundamental work by E. Artin, H. Hasse, A. Ostrowski and others, S. MacLane created a method for describing all extensions of a discrete rank one valuation on a field $K$ to a primitive extension $K(z)$, be it algebraic or transcendental. The method is based on the existence of \textit{key polynomials} in $K[z]$ which provide successive approximations of a given extension of the valuation and, by the behavior of their degrees, a measure of its complexity.\par
On the side of algebraic geometry, Zariski's approach to resolution of singularities of algebraic varieties using local uniformization of valuations provides a strong motivation for the study of valuations on local domains essentially of finite type over a field, which waned after Hironaka's proof of resolution in characteristic zero but later revived as an approach to resolution in positive characteristic.\par
 In the 1970's and 1980's appeared (see \cite{T1}, \cite{Sp}, \cite{GT}) the idea that the associated graded ring  ${\rm gr}_\nu A$ of a local domain $A$ with respect to the filtration of $A$ associated to a valuation $\nu$ of its field of fractions centered in $A$ (non negative on $A$ and positive on its maximal ideal) encoded in a geometric way essential characters of the valuation and could be used at least in special cases to obtain local uniformization. For example, representatives in $A$ of the generators of the graded algebra associated to the unique valuation of a one dimensional integral complex analytic algebra can be used to embed the corresponding curve in an affine space where a single birational toric modification provides an embedded resolution of singularities (see \cite{GT}). It also became apparent that some of MacLane's essential definitions are better understood using associated graded rings.\par
Somewhat later, MacLane's theory was generalized by Vaqui\'e who extended to all Krull valuations the construction of sequences of key polynomials, now indexed by totally ordered sets \textit(see \cite{V1}, \cite{V2}, \cite{V3} and section \ref{SecV} below). He also described the extension ${\rm gr}_\nu K\subset {\rm gr}_\omega K[z]$ of graded rings corresponding to an extension of valuations from $\nu$ on $K$ to $\omega$ on $K[z]$, for $z$ algebraic or transcendental over $K$. It appeared that the images of MacLane's and Vaqui\'e's key polynomials in the graded algebra ${\rm gr}_\omega K[z]$ were related to its generation as a ${\rm gr}_\nu K$-algebra. \par\noindent In the last three decades or so the problem of describing a \textit{generating sequence} for a valuation, which is a set of elements of a ring $A$ whose images in ${\rm gr}_\nu A$ provide a presentation by generators and relations has become of major interest for the ramification theory of extensions of valued fields as well as for local uniformization in positive characteristic, which is still an open problem.\par\noindent
In fact it has become apparent that given an extension $(A,\nu)\subset (B,\omega)$ of valued \textit{rings} the extension ${\rm gr}_\nu A\subset {\rm gr}_\omega B$ of the associated graded algebras, as well as the similar extensions obtained after birational extensions of $A$ and $B$ encodes in a comparatively simple language, such as the condition of being finitely generated, essential information about the ramification of the original extension. This concerns especially the defect and the possibility to uniformize $\omega$ on $B$ if we can uniformize $\nu$ on $A$. But we can access this information only if we have descriptions by generators and relations of ${\rm gr}_\nu A$ and $ {\rm gr}_\omega B$, or of  ${\rm gr}_\omega B$ as a ${\rm gr}_\nu A$-algebra. This is the main motivation for this work.\footnote{We think that the problem of constructing generating sequences in a Noetherian local domain $A$ which is dominated by a valuation $\nu$ is very difficult, and little is known about it in general. The difficulty reflects the fact that the structure of the semigroup of values $S^A(\nu)=\nu(A\setminus \{0\})$ is closely related to some of the birational maps providing embedded local uniformizations of $\nu$ and can be extremely complicated. It is well understood in the case that $A$ has dimension one (see \cite{T1}, \cite {GT}), and for regular local rings of dimension two (\cite{Sp}, \cite{CV1}, \cite{Mo}). It is known for certain valuations dominating two dimensional quotient singularities \cite{D} and for certain valuations dominating three dimensional regular local rings \cite{Kas}. }\par

Here we consider the case where the essence of the difficulty resides: suppose that $(K,\nu)$ is a valued field, $f(z)\in K[z]$ is a unitary and irreducible polynomial and $(L,\omega)$ is a finite field extension, where $L=K[z]/(f(z))$. Further suppose that $A$ is a local domain with quotient field $K$ such that $\nu$  dominates $A$ and that $f(z)$ is in $A[z]$. We provide an algorithm producing the first significant part of a generating sequence for extensions of a valuation $\nu$ to $A[z]/(f(z))$.\par\noindent
 The valuations $\nu$ and $\omega$ also induce filtrations of $K$ and $K[z]/(f(z))$ respectively and the associated graded ring of $K[z]/(f(z))$ along $\omega$ as an extension of the associated graded ring of $K$ along $\nu$ has been constructed implicitly, in the papers \cite{M1}, \cite{M2} of MacLane 
for discrete rank one valuations, and for general valuations by Vaqui\'e in \cite{V1}, \cite{V2}, \cite{V3}.  Further papers on this topic, and comparison with the method of pseudo convergent sequences (introduced by Ostrowski in \cite[Teil III, \S 11]{O} and developed by Kaplansky in \cite{Ka}) are \cite{RB}, \cite{NS}, \cite{Sa}, \cite{HMOS} and \cite{DMS}. Finding generating sequences for $A[z]/(f(z))$ in the case where $A$ is no longer a field but an arbitrary noetherian subring dominated by $R_\nu$ and with the same field of fractions is much more closely related to resolution of singularities via local uniformization and correspondingly more difficult.\par\noindent
This paper is devoted to this problem. We describe the relationship of our method with the key polynomials of MacLane and Vaqui\'e. We also work out the interactions of our method of computation with phenomena which complicate the study of ramification in positive characteristic, such as the lack of tameness and the defect of an extension.\par\medskip\noindent
We now give more details about the content of this paper:\par\noindent
Let $G_{\nu}$ be the value group of $\nu$ and $R_{\nu}$ be the valuation ring of $\nu$, with maximal ideal $m_{\nu}$. Given a subring $A$ of the field of fractions of $R_\nu$, the associated graded ring of $A$ along $\nu$ is defined as 
$$
{\rm gr}_{\nu}(A)=\bigoplus_{\gamma\in G_{\nu}}\mathcal P_{\gamma}(A)/\mathcal P_{\gamma}^+(A)
$$
where 
$$
\mathcal P_{\gamma}(A)=\{g\in A\setminus \{0\}\mid \nu(g)\ge\gamma\}\mbox{ and }
\mathcal P_{\gamma}^+(A)=\{g\in A\setminus \{0\}\mid \nu(g)>\gamma\}.
$$
The ring ${\rm gr}_{\nu}(A)$ is an algebra over its degree zero subring. It is a domain which is generally not Noetherian. In this text we shall consider subrings of $R_\nu$ so that the semigroup $S^A(\nu)$ of values of elements of $A\setminus\{0\}$ which indexes the homogeneous components of ${\rm gr}_{\nu}(A)$ is contained in the positive part of $G_\nu$. We shall see more about this semigroup below. \par\noindent 
Important invariants of a finite extension $(K,\nu)\subset (L,\omega)$ of valued fields are the reduced ramification index and residue degree of $\omega$ over $\nu$, which are
$$
e(\omega/\nu)=[G_{\omega}:G_{\nu}]\mbox{ and }f(\omega/\nu)=[R_{\omega}/m_{\omega}:R_{\nu}/m_{\nu}].
$$
Another, very subtle invariant is the defect $\delta(\omega/\nu)$ of the extension, which is a power of the characteristic $p$ of the residue field $R_{\nu}/m_{\nu}$.  The defect and its role in local uniformization are explained in \cite{K1}. We give the definition of   the defect  in (\ref{eqefd}) below. In the case where $\omega$ is the unique extension of $\nu$ to $L$ we have that
\begin{equation}\label{eqN300}
[L:K]=e(\omega/\nu)f(\omega/\nu)\delta(\omega/\nu).
\end{equation}
If $A$ and $B$ are local  domains with quotient fields $K$ and $L$ such that $\omega$ dominates $B$ and $B$ dominates $A$, we have a graded inclusion of graded domains
$$
{\rm gr}_{\nu}(A)\rightarrow {\rm gr}_{\omega}(B).
$$
The index of quotient fields is:
$$
[{\rm QF}({\rm gr}_{\omega}(B)):{\rm QF}({\rm gr}_{\nu}(A))]=e(\omega/\nu)f(\omega/\nu)
$$
by Proposition 3.3  of \cite{C4}.  The defect seems to disappear, but it manifests itself in mysterious behavior in the extensions of associated graded rings of injections $A'\rightarrow B'$ of birational extensions of Noetherian local domains $A, B$. For instance, if $\nu$ has rational rank 1 but is not discrete, the defect $\delta(\omega/\nu)$ is larger than 1 and $A$ and $B$ are two dimensional excellent local domains, then ${\rm gr}_{\omega}(B')$ is not a finitely generated ${\rm gr}_{\nu}(A')$-algebra for any regular local rings $A'\rightarrow B'$ which are dominated by $\omega$ and dominate $A$ and $B$ as shown in \cite{C3}.\par
The  construction of  generating sequences is closely related to the problem of local uniformization. In \cite[Theorem 7.1]{CM}, it is shown how reduction of multiplicity along a rank 1 valuation can be achieved in a defectless  extension $A\rightarrow A[z]/(f(z))$. 
A similar statement is proven by San Saturnino in \cite{Sa}.

The statement ``defectless'' means that the rank 1 valuations $\nu$ and $\omega$ satisfy $\delta(\omega/\nu)=1$.
From this assumption, it follows that either
$\omega(z-K)$ has a largest element, or the limsup of this set is $\infty$. If the limsup of this set is $\infty$, then in an appropriate extension, the valuation $\omega$ corresponds to a linear factor of $f(z)$, and it is not difficult to realize a reduction of multiplicity by blowing up. So assume that $\omega(z-K)$ has a largest element $\gamma\in G_{\omega}$. We then have $\gamma\not\in G_{\nu}$. After a birational extension $A_1$ of $A$ and a change of variables of $z$ in $A_1[z]$, we obtain that $\omega(z)=\gamma$ and then after a Cremona transformation involving $z$, we obtain a reduction of the multiplicity of the strict transform of $f$.

In  \cite{T2} and \cite{T3}, it is shown how  associated graded rings along a valuation can be used to prove local uniformization, at least when the associated graded rings are finitely generated algebras over $A/m_A$.  A suitable toric resolution of singularities of the associated graded ring induces a local uniformization of the given valuation.

The subring of degree zero elements of the graded ring ${\rm gr}_{\nu}(A)$ is $({\rm gr}_{\nu}(A))_0=A/Q$ where $Q$ is the prime ideal in $A$ of elements of positive value.
A generating sequence for  $\nu$ on  $A$ is an ordered set of elements of $A$ whose classes in ${\rm gr}_{\nu}(A)$ generate ${\rm gr}_{\nu}(A)$ as a graded $({\rm gr}_{\nu}(A))_0$-algebra.  To be meaningful, a generating sequence should come with a formula for computing the values of elements of $A$, and their relations in ${\rm gr}_{\nu}(A)$. In particular, a generating sequence should give the structure of ${\rm gr}_{\nu}(A)$ as a graded $({\rm gr}_{\nu}(A))_0$-algebra. 

In the case of an inclusion $A\subset B$ of domains, and an extension $\omega$ of $\nu$ to the quotient field of $B$ such that $\omega$ has nonnegative value on $B$, a generating sequence of the extension is an ordered sequence of elements of $B$ whose classes in ${\rm gr}_{\omega}(B)$ generate ${\rm gr}_{\omega}(B)$ as a ${\rm gr}_{\nu}(A)$-algebra. A generating sequence for an extension should come with a formula for computing the values of elements of $B$, relative to the values of elements of $A$, and give their relations in ${\rm gr}_{\nu}(B)$. That is, a generating sequence should give the structure of ${\rm gr}_{\omega}(B)$ as a graded ${\rm gr}_{\nu}(A)$-algebra.

In this paper, we give a very simple algorithm which allows us to compute a generating sequence and the structure of ${\rm gr}_{\omega}(A[z]/(f(z))$ in many situations. Throughout this paper, we have the assumption that $A$ is a local domain which contains an algebraically closed field $\mathbf k$ such that its residue field $A/m_A=\mathbf k$, $\nu$ dominates $A$ and  the residue field of the valuation ring $R_{\nu}$ of $\nu$ is $R_{\nu}/m_{\nu}=\mathbf k$ ($\nu$ is a ``rational valuation''). This algorithm is derived in Section \ref{Sec3}. The algorithm is valid for an arbitrary extension $\omega$ of an arbitrary valuation $\nu$ dominating $A$ ($m_{\nu}\cap A=m_A$).

A  realization of our algorithm  produces a subring of ${\rm gr}_{\omega}(R_{\nu}[z]/(f(z))$  
which is the quotient $C/I$ of a graded polynomial ring $C$ over ${\rm gr}_{\nu}(R_{\nu})$ in either finitely many or countably many variables, and a set of generators of the graded prime ideal $I$ of $C$.  Our algorithm gives an explicit representation of this subring as 
 $$
{\rm gr}_{\nu}(R_{\nu})[\overline \phi_1,\ldots,\overline \phi_k,\ldots]/
I,
$$
where
$$
I=(\overline \phi_1^{n_1}-\overline c_1,\overline \phi_2^{n_2}-\overline c_2\overline\phi_1^{j_1(2)},\ldots,\overline \phi_k^{n_k}-\overline c_k\overline\phi_1^{j_1(k)}\overline\phi_2^{j_2(k)}\cdots\overline \phi_{k-1}^{j_{k-1}(k)},\ldots)
$$
with $\overline c_1,\ldots,\overline c_k,\ldots\in {\rm gr}_{\nu}(R_{\nu})$ homogeneous elements. The  elements $\overline \phi_i$ are homogeneous with  strictly increasing values.
If our algorithm terminates in a finite number of steps $k$, then  elements $\phi_1,\ldots,\phi_k\in R_{\nu}[z]$ whose classes are $\overline\phi_1,\ldots,\overline \phi_k$ form a generating sequence of $R_{\nu}[z]/(f(z))$ over $R_{\nu}$ and we have built up the entire associated graded ring
$$
{\rm gr}_{\nu}(R_{\nu}[z]/(f(z))) = {\rm gr}_{\nu}(R_{\nu})[\overline \phi_1,\ldots,\overline \phi_k]/
I
$$
where
$$
I=(\overline \phi_1^{n_1}-\overline c_1,\overline \phi_2^{n_2}-\overline c_2\overline\phi_1^{j_1(2)},\ldots,\overline \phi_k^{n_k}-\overline c_k\overline\phi_1^{j_1(k)}\overline\phi_2^{j_2(k)}\cdots\overline \phi_{k-1}^{j_{k-1}(k)}).
$$
In this case, we have that ${\rm gr}_{\nu}(R_{\nu}[z]/(f(z))$ is a finitely generated and presented ${\rm gr}_{\nu}(R_{\nu})$-module.

 When we compare our algorithm to the theory of Vaqui\'e (\cite{V1}, \cite{V2}, \cite{V3}) in Subsection \ref{SubSecMV}, we conclude in Proposition \ref{PropVO} that a realization of our algorithm produces the ``first simple admissible family'' $\mathcal S^{(1)}$ of an ``admissible family'' $\mathcal S$ determining the valuation $\omega$. 

In the case of a noetherian local domain $A$ dominated by $R_\nu$ as above, our algorithm produces  in many situations a finite sequence of elements of $A[z]$ whose images generate the ${\rm gr}_\nu A$-algebra ${\rm gr}_\omega A[z]$. It does this even in cases where there are infinitely many key polynomials. Remarks 8.12 in \cite{T3} displays a similar phenomenon of finite generation in the presence of an infinity of key polynomials. \par
More precisely, if the characteristic  $p$ of $k$ does not divide the degree of $f$,  $A$ is a domain as above and $\omega$ is the unique extension of $\nu$ to a valuation of the quotient field $L$ of $A[z]/((f(z))$, then we show in Theorem \ref{Theorem1} that our  algorithm  produces a finite generating sequence in $A[z]/(f(z))$. 
The associated graded ring of $A[z]/(f(z))$ along  $\omega$ is then a finitely generated and presented  module over the associated graded ring of $A$ along  $\nu$.

 Since the defect $\delta(\omega/\nu)$ is always a power of $p$, the assumption that $p$ does not divide the degree of  $f$ in Theorem \ref{Theorem1} and the assumption that $\omega$ is the unique extension of $\nu$  forces the defect $\delta(\omega/\nu)$ to be  1 by (\ref{eqN300}). 
 
We show that if any of the above assumptions are removed, then the conclusions of Theorem \ref{Theorem1} do not hold (Examples of Section \ref{Sec3} and Section \ref{Sec9}).  For instance,  the assumption that $R_{\nu}[z]/(f(z))$ is a ``hypersurface singularity'' is shown to be necessary for finite generation to hold in Example \ref{IEx1}.

To illustrate the power of Theorem \ref{Theorem1}, we compute in Example \ref{Ex40} the associated graded ring when $f(z)$ is a quadratic polynomial, $k$ has characteristic not equal to 2 and $\omega$ is the unique extension of $\nu$.  It has the simple form
$$
{\rm gr}_{\omega}(A[z]/(f(z))\cong {\rm gr}_{\nu}(A)[\overline\phi]/(\overline \phi^2-\overline c)
$$
for some homogeneous $\overline c\in{\rm gr}_{\nu}(A)$.
From the classification of associated graded rings of valuations dominating a two dimensional regular local ring $A$ (\cite{Sp} and \cite{CV1})) we see that we are able to completely calculate the associated graded ring along an extended valuation in the local rings of two dimensional rational double points, when the extension $\omega$ is unique. In constrast, if $\omega$ is not the unique extension of $\nu$, then ${\rm gr}_{\omega}(A[z]/(f(z))$ might not be a finitely generated ${\rm gr}_{\nu}(A)$-module, as shown in Examples \ref{Ex40} and \ref{IEx2}.

 In Theorem \ref{Theorem4*}, we consider an arbitrary separable extension (with no assumption on the degree) and
 assume that $A$ is a Nagata local domain. We show that  an extension of a rank one valuation $\nu$ is without  defect if and only if 
 there exists a realization of our algorithm with coefficients in a birational extension $A_1$ of $A$   which constructs $\omega$, either as a valuation or a limit valuation.  A birational extension $A_1$ of $A$ is  a localization of a finitely generated $A$-algebra whose quotient field is $K$ and which is dominated by $\nu$.  
 
 An example showing that the conclusions of Theorem \ref{Theorem4*} may not hold if $\nu$ has rank larger than one is given in Section \ref{SecNEX}. In Example \ref{Ex201}, it is shown that the conclusions of Theorem \ref{Theorem4*} may not hold if $f(z)$ is not separable over $K$.

 In Section \ref{Sec8} we analyze our algorithm in a rank 1 example with defect from \cite{CP} to motivate the necessary condition of Theorem \ref{Theorem4*}. 
 We explicitly show that a generating sequence does not exist in $A_1[z]$ for any birational extension $A_1$ of $A$ which is dominated by $\nu$, and the  valuation $\omega$ is not realizable as a limit valuation; that is, $\omega$ is not realizable as a sequence of approximants, only of a collection of approximants indexed by a more general well ordered set.

 In the final section, Section \ref{Sec9}, we give examples showing that the finite generation of extensions of associated graded rings and valuation semigroups ensured by Theorem \ref{Theorem1}  may fail if any of the assumptions of the theorem are removed. 
 The semigroup $S^{A}(\nu)$ of values of $\nu$ on $A$ is
 $$
 S^{A}(\nu)=\{\nu(g)\mid g\in A\setminus \{0\}\}.
 $$

 In Example \ref{IEx1}, it is shown that there exists an extension $L$ of the quotient field $K$ of $A$  of degree prime to $p$, a valuation $\nu$ of $K$ which dominates $A$ and has a  unique extension to $L$ such that if $B$ is the integral closure of $A$ in $L$, then ${\rm gr}_{\omega}(B)$ is not a finitely generated ${\rm gr}_{\nu}(A)$-module   and the semigroup $S^{B}(\omega)$ is not a finitely generated $S^{A}(\nu)$-module.    In particular, the conclusions of Theorem \ref{Theorem1} do not hold for this extension.
  This example shows that we must have the condition that $B=A[z]/(f(z))$ is a ``hypersurface singularity'' for the conclusions of Theorem \ref{Theorem1} to be true. 
 
We make use of the theory of MacLane, \cite{M1}, \cite{M2}, which he developed to construct the extensions of a (rank 1) discrete valuation $\nu$ of $K$ to a discrete valuation  $\omega$ of $K[z]$ or of $K[z]/(f(z))$ for some irreducible unitary polynomial $f(z)\in K[z]$. Our algorithm can be viewed as a realization of MacLane's method in the context of a general valuation, in a specific, nice form. MacLane's theory is surveyed in Section \ref{Sec2}.

We also make use of Vaqui\'e's generalization of MacLane's method in \cite{V1}, \cite{V2}, \cite{V3} to construct extensions of general valuations in $K[z]$ and $K[z]/(f(z))$ in our proof of Theorem \ref{Theorem4*}. The essential new concept in Vaqui\'e's work is that of a ``limit key polynomial''. He gave in \cite[Exemple 4.1]{V3} an example of infinite sequences of key polynomials due to the non uniqueness of valuation extension. Vaqui\'e's method is surveyed in Section \ref{SecV}, as well as a study of its relationship to our algorithm. In the situation of this paper we shall meet only finite sequences of limit key polynomials since the number of limit key polynomials is bounded by the degree of $f(z)$.  In Section \ref{Sec5} we collect and derive some results about Henselizations of rings and valued fields which we need for the proof of Theorem \ref{Theorem4*}.

 In this paper, a local ring is a commutative ring with a unique maximal ideal. In particular, we do not require a local ring to be Noetherian. We will denote the maximal ideal of a local ring $A$ by $m_A$.  The quotient field of a domain $A$ will be denoted by ${\rm QF}(A)$. We will say that a local ring $B$ dominates a local ring $A$ if $A\subset B$ and $m_B\cap A=m_A$.
 
 We will denote the natural numbers by $\NN$ and the positive integers by $\ZZ_+$.

\section{Valuations and pseudo valuations}
We shall in the sequel consider sequences of valuations which approximate $\omega$. For that reason we change notations and denote these sequences by $V_0,V_1, \ldots$ as in \cite{M1} and \cite{M2}. A general valuation will be denoted by $V$ and the reader may think of $\nu$ as $V_0$.

Suppose that $V$ is a valuation on a field $K$. We will denote the valuation ring of $V$ by $R_V$ and its maximal ideal by $m_V$. The value group of $V$ will be denoted by $G_V$. 

Suppose that $A$ is a Noetherian local domain with quotient field $K$ and $A\rightarrow A_1$ is an extension of local domains such that  $A_1$ is a domain whose quotient field is $K$ and $A_1$ is essentially of finite type over $A$ ($A_1$ is a localization of a finitely generated $A$-algebra). Then we will say that $A\rightarrow A_1$ is a birational extension.

 If $A$ is a domain which is contained in $R_V$, then the associated graded ring of $A$ along $V$ is
 $
 {\rm gr}_V(A)$ as defined in the introduction, The initial form ${\rm In}_V(g)$ of $g\in A$ is the class of $g$ in $\mathcal P_{V(g)}(A)/\mathcal P^+_{V(g)}(A)$. The semigroup of $V$ on $A$ has also been defined in the introduction.\par\noindent A pseudo valuation (or semivaluation) $V$ on a domain $A$ is a surjective map $V:A\rightarrow G_V\cup\{\infty\}$ where $G_V$ is a totally ordered Abelian group and a prime ideal
 $$
 I(V)_{\infty}=I^A(V)_{\infty}=\{g\in A\mid V(g)=\infty\}
 $$
 of $A_1$ such that $V:{\rm QF}(A/I(V)_{\infty})\setminus \{0\}\rightarrow G_V$ is a valuation.
 
% If $I(V)_{\infty}=(0)$, then $V$ is called a finite value, and if $I(V)_{\infty}\ne (0)$ then $V$ is called an infinite value.
 
\section{The MacLane  theory of key polynomials}\label{Sec2}
 
 Suppose that  $V$ is a valuation or a pseudo valuation on a domain $A$. Following MacLane in \cite{M1} in the case $A=K[z]$, we can define an equivalence $\sim$ on $A$ defined for $g,h\in A$ by $g\sim g$ in $V$ if $V(g-h)>\min\{V(g),V(h)\}$ or $V(g)=V(h)=\infty$.
We say that $g\in A$ is equivalence divisible by $h$ in $V$, written $h|g$ in $V$, if there exists $a\in A$ such that $g\sim ah$ in $V$. An element $g$ is said to be equivalence irreducible in $V$ if $g|ab$ in $V$ implies $g|a$ or $g|b$ in $V$.\par\noindent These conditions can be expressed respectively as the statement that ${\rm In}_V(h))={\rm In}_V(g)$ in ${\rm gr}_V(A)$, that ${\rm In}_V(h)$ divides ${\rm In}_V(g)$ in ${\rm gr}_V(A)$ and that the ideal generated by ${\rm In}_V(g)$ in ${\rm gr}_V(A)$ is prime.

\subsection{MacLane's algorithm}
We review MacLane's algorithm \cite{M1}  to construct the extensions of a valuation $V_0$ of a field $K$ to a valuation or pseudo-valuation of the polynomial ring $K[z]$. MacLane applied his method to construct extensions of rank 1 discrete valuations of  $K$ to $K[z]$. This algorithm  has been extended to general valuations by Vaqui\'e \cite{V1}. MacLane constructs ``augmented sequences of inductive valuations''
\begin{equation}\label{eqM1}
V_1,\ldots,V_k,\ldots
\end{equation}
which extend $V_0$ to $K[z]$. An augmented sequence (\ref{eqM1}) is constructed from successive inductive valuations
\begin{equation}\label{eqM2}
V_k=[V_{k-1};V_k(\phi_k)=\mu_k]\mbox{ for $1\le k$}
\end{equation}
of $K[z]$, where $\phi_k$ is a ``key polynomial'' over $V_{k-1}$  and $\mu_k$ is a ``key value'' of $\phi_k$ over $V_{k-1}$. We always take $\phi_1=z$.

We say that $\phi(z)\in K[z]$ is a key polynomial with key value $\mu$ over $V_{k-1}$ if
\begin{enumerate}
\item[1)] $\phi(z)$ is equivalence irreducible in $V_{k-1}$.
\item[2)] $\phi(z)$ is minimal in $V_{k-1}$; that is, if $\phi(z)$ equivalence divides $g(z)$ in $V_{k-1}$, then $\deg_z\phi(z)\le\deg_zg(z)$.
\item[3)] $\phi(z)$ is unitary and $\deg_z\phi(z)>0$.
\item[4)] $\mu>V_{k-1}(\phi(z))$.
\end{enumerate}
Following MacLane (\cite[Definition 6.1]{M1}) we also assume
\begin{enumerate}
\item[5)] $\deg_z\phi_i(z)\ge\deg_z\phi_{i-1}(z)$ for $i\ge 2$.
\item[6)] $\phi_i(z)\sim \phi_{i-1}(z)$ in $V_{i-1}$ is false. Here the equivalence is to be understood for polynomials in $K[z]$.
\end{enumerate}

It follows from \cite[Theorem 9.3]{M1} that 
\begin{equation}\label{eqM16}
\mbox{if $\phi(z)$ is a key polynomial over $V_{k-1}$ then $\deg_z\phi_{k-1}(z)$ divides $\deg_z\phi(z)$.}
\end{equation} 

The key polynomials $\phi_k(z)$ can further be assumed to be homogeneous in $V_{k-1}$, which will be defined after (\ref{eqM5}).

MacLane shows that if $V_0$ is discrete of rank 1, then the extensions of $V_0$ to a valuation or pseudo valuation of  $K[z]$ are the $V_k$ arising from  augmented sequences of finite length
 (\ref{eqM1}) and the limit sequences of augmented sequences of infinite length (\ref{eqM1}) which determine a limit value
 $V_{\infty} $ on $K[z]$ defined by 
 $$
 V_{\infty}(g(z))=\lim_{k\rightarrow\infty} V_k(g(z))\mbox{ for }g(z)\in K[z].
 $$
 We have that $V_{\infty}(g(z))$ is well defined whenever  $V_0$ has rank 1, and is a valuation or pseudo-valuation by the argument of  \cite[page 10]{M1}.
 
 MacLane's method has been extended by Vaqui\'e \cite{V1},  to eventually construct all extensions of an arbitrary valuation $V_0$ of $K$ to a valuation or pseudo valuation of $K[z]$. We will discuss Vaqui\'e's method in Section \ref{SecV}.
 
 To compute the ``$k$-th stage''  value $V_k(g(z))$  for $g(z)\in K[z]$ by MacLane's method, we consider the unique expansion
 \begin{equation}\label{eqM3}
 g(z)=g_m(z)\phi_k^m(z)+g_{m-1}\phi_k^{m-1}(z)+\cdots+g_0
 \end{equation}
 with $g_i(z)\in K[z]$, $\deg_zg_i(z)<\deg_z\phi_k(z)$ for all $i$ and $g_m(z)\ne 0$. Then
 $$
 V_k(g(z))=\min\{V_{k-1}(g_m(z))+m\mu_k,V_{k-1}(g_{m-1}(z))+(m-1)\mu_k,\ldots,V_{k-1}(g_0)(z)\}.
 $$
 This expression suffices to prove by induction, assuming the existence of a unique expansion of the coefficients $g_i(z)$ in terms of the polynomials $\phi_j(z)$ with $j<k$, that every $g(z)\in K[z]$ has a unique expansion
 \begin{equation}\label{eqM4}
 g(z)=\sum_ja_j(z)\phi_1^{m_{1,j}}(z)\phi_2^{m_{2,j}}(z)\cdots \phi_k^{m_{k,j}}(z)
\end{equation}
with $a_j\in K$ and
  $0\le m_{i,j}<\deg_z\phi_{i+1}/\deg_z\phi_i\mbox{ for }i=1,\ldots,k-1$. Recall that $\deg_z\phi_{i+1}/\deg_z\phi_i$ is a positive integer by (\ref{eqM16}).  Then
\begin{equation}\label{eqM5}
V_k(g)=\min_jV_k(a_j\phi_1^{m_{1,j}}\phi_2^{m_{2,j}}\cdots \phi_k^{m_{k,j}}).
\end{equation}
  If all terms in (\ref{eqM4}) have the same values in $V_k$ then $g$ is said to be homogeneous in $V_k$.\par\noindent
  \textit{We shall often, as we just did, simplify notations by writing $g$ for $g(z)$, etc. when there is no fear of confusion.}

  \begin{Remark}\label{RemarkM20}If 
   $A$ is a subring of $K$ such that $\phi_i\in A[z]$ for $1\le i\le k$ and $g\in A[z]$, then the coefficients $a_j$ in (\ref{eqM4}) are all in $A$.
   \end{Remark}
  
  The polynomial $g$, with expansion (\ref{eqM3}), is minimal in $V_k$ if and only if $g_m\in K$ and \begin{equation}\label{eqM13}
  V_k(g)=V_k(g_m\phi_k^m)
  \end{equation}
  by 2.3 \cite{M2} or Theorem 9.3 \cite{M1}.

  By 3.13 of \cite{M2} or \cite[Theorem 6.5]{M1}, for $k>i$,
  \begin{equation}\label{eqM8}
  V_k(\phi_i)=V_i(\phi_i)\mbox{ and }V_k(g)=V_{i}(g)\mbox{ whenever }\deg_zg<\deg_z\phi_{i+1}.
  \end{equation}
  
 Further, by \cite[Theorems 5.1 and 6.4]{M1}, or \cite[3.11 and 3.12]{M2},
 \begin{equation}\label{eqM15}
 \mbox{For all $g\in K[z]$, $V_k(g)\ge V_{k-1}(g)$ with equality if and only if $\phi_k\not\,\mid g$ in $V_{k-1}$}.
 \end{equation}

\subsection{MacLane's algorithm  in a finite primitive extension}\label{SubSecMA}
Suppose $f(z)\in K[z]$ is unitary and irreducible. The extensions of $V_0$ to valuations of $K[z]/(f(z))$ are the extensions of $V_0$ to 
pseudo valuations $V$ of $K[z]$ such that $I(V)_{\infty}=(f(z))$. MacLane \cite{M2} gives an explicit explanation of how his algorithm can be applied to construct the pseudo valuations $V$ of $K[z]$ which satisfy $I(V)_{\infty}=(f(z))$ in Section 5 of \cite{M2} (when $V_0$ is discrete of rank 1). Vaqui\'e shows in \cite{V2} and \cite{V3} how this algorithm can be extended to arbitrary valuations $V_0$ of $K$.

Suppose $V_1,\ldots, V_k$ is an augmented  sequence of inductive valuations in $K[z]$. Expand 
$$
f=f_m\phi_k^m+\cdots+f_0
$$
as in (\ref{eqM3}). Define the projection of $V_k$ by ${\rm proj}(V_k)=\alpha-\beta$ where $\alpha$ is the largest and $\beta$ is the smallest amongst the exponents $j$ for which $V_k(f(z))=V_k(f_j\phi_k^j)$. 
A $k$-th approximant $V_k$ to $f(z)$ over $V_0$ is a $k$-th stage homogeneous (meaning that the key polynomial $\phi_i$ is homogeneous in $V_{i-1}$ for $i\le k$) inductive valuation which is an extension of $V_0$ and which has a positive projection (\cite[Definition 3.3]{M2}).

First approximants $V_1$ to $f$ are defined as $V_1=[V_0;V_1(\phi_1)=\mu_1]$, where $\phi_1=z$ and $\mu_1$ is chosen so that ${\rm proj}(V_1)>0$. MacLane shows in \cite[Lemma 3.4]{M2} that if $V_k$ is a $k$-th approximant to $f(z)$, then so is $V_i$ for $i=1,\ldots,k-1$. Further, $\phi_k|f$ in $V_{k-1}$ and $V_k(f(z))>V_{k-1}(f(z))>\cdots>V_1(f(z))$. In \cite[Theorem 10.1]{M2}, MacLane shows that if $V_0$ is a discrete valuation of rank 1 then
every extension of $V_0$ to a valuation of $K[z]/(f(z))$ is an augmented sequence of finite length of approximants
$V_1,\ldots,V_k$ such that $V_k(f(z))=\infty$  or a limit of an augmented sequence of approximants of infinite length such that $V_{\infty}(f(z))=\infty$. If $V_0$ is not discrete of rank 1, then there is the possibility that the algorithm will have to be continued to construct a pseudo valuation $W$ of $K[z]$ with $W(f(z))=\infty$. If this last case occurs, then the situation becomes quite complicated, as we must then extend the family $\{V_k\mid k\in \ZZ_+\}$ to a ``simple admissible family'' and possibly make some jumps. 
This is shown by Vaqui\'e in \cite[Theorem 2.5]{V1} and  is explained in Section \ref{SecV}. An essential point is that for every construction
$V_1,\ldots,V_k$ of approximants to $f$ over $V_0$ by MacLane's algorithm, there exists an extension $W$ of $V_0$ to a pseudo valuation of $K[x]$ such that $I(W)_{\infty}=(f(z))$ and $W(\phi_k)=V_k(\phi_k)$ for all $k$ (This will be deduced from \cite[Theorem 1]{V3}  in Theorem \ref{Theorem4}).

We will assume now that $V_0$ has rank 1, so we may assume that $G_{V_0}$ is an ordered subgroup of $\RR$. We will now look a little more at the case where  we have an infinite sequence of approximants, leading to a limit valuation $V_{\infty}$. In this case, there exists $k_0$ such that $\phi_k=\phi_{k_0}+h_k$ with $\deg_zh_k<\deg_z\phi_{k_0}$ for $k\ge k_0$. Thus for $k>k_0$, 
$$
V_k(\phi_k)>V_{k-1}(\phi_k)\ge V_{k-1}(\phi_{k-1}).
$$
Thus $\lim_{k\rightarrow\infty}V_k(\phi_k)$ exists, and is either equal to $\infty$ or an element of $\RR$. 

\begin{Lemma}\label{Lemma4} Suppose that $V_0$ has rank 1 and $V_1,\ldots,V_k,\ldots$ is an infinite sequence of approximants to $f$ over $V_0$. Then the following are equivalent:
\begin{enumerate}
\item[1)] $V_{\infty}=\lim_{k\rightarrow\infty}V_k$ is a pseudo valuation on $K[z]$ (but not a valuation).
\item[2)] $I^{K[z]}(V_{\infty})_{\infty}=(f(z))$.
\item[3)] $\lim_{k\rightarrow \infty} V_k(\phi_k)=\infty$.
\end{enumerate}
\end{Lemma}

\begin{proof} We first prove 1) implies 3). By assumption, there exists $0\ne h\in I(V_{\infty})_{\infty}$. 
There exists $k_0$ such that 
for $k\ge k_0$, $\deg_z\phi_k=\deg_z\phi_{k_0}$. 
Expand 
$$
h=h_m\phi_{k_0}^m+h_{m-1}\phi_{k_0}^{m-1}+\cdots+h_0
$$
with $\deg_zh_i<\deg_z\phi_{k_0}$ for all $i$ and $h_m\ne 0$. There exists $\lambda\in\ZZ_+, \ 1\leq\lambda \leq\deg\phi_{k_0}$ such that
$\deg_zz^{\lambda}h_m=\deg_z\phi_{k_0}$ and so there exists $0\ne\alpha\in K$ such that $\alpha z^\lambda h_m=\phi_{k_0}+\eta_m$ with $\deg \eta_m<\deg\phi_{k_0}$. This implies that  $\alpha z^{\lambda}h$ has an expansion
$$
\alpha z^{\lambda}h=\phi_{k_0}^{m+1}+\eta_m\phi_{k_0}^m+\alpha z^{\lambda}h_{m-1}\phi_{k_0}^{m-1}+\cdots+\alpha z^\lambda h_{m-j}\phi_{k_0}^{m-j}+\cdots +\alpha z^\lambda h_0
$$
with $\deg_zz^\lambda h_{m-j}<2\deg_z\phi_{k_0}$ for all $j$. Now we can expand each $\alpha z^\lambda h_{m-j} =\eta_{m-j}\phi_{k_0}+\theta_{m-j}$, with $\deg_z \eta_{m-j}$ and $\deg_z \theta_{m-j}$ less than $\deg_z\phi_{k_0}$, so that finally we can expand
$$
\alpha z^\lambda h=\phi_{k_0}^{m+1}+h'_m\phi_{k_0}^m+ \cdots +h'_{m+1-j}\phi_{k_0}^{m+1-j}+\cdots +h'_0
$$
with $\deg_zh'_{m+1-j}<\deg_z\phi_{k_0}$ for all $j$.
Thus, substituting $\alpha z^{\lambda}h\in I(V_{\infty})_{\infty}$ for $h$ and continuing to denote by $m$ the degree of its expansion in $\phi_{k_0}$, we may assume that $h_m =1$.
The same argument shows that for $k\geq k_0$ there exist  $h_i(k)\in K[z]$ for $i<m$ such that 
$$
h=\phi_k^m+h_{m-1}(k)\phi_k^{m-1}+\cdots +h_0(k)
$$
with $\deg_zh_j(k)<\deg_z\phi_k$. Now by definition of $V_k$ we have
$$
V_k(h)\le mV_k(\phi_k)
$$
for $k\ge k_0$, so $\lim_{k\rightarrow\infty}V_k(\phi_k)=\infty$.

We now prove that  3) implies  2). In the expansion 
$$
f=f_m\phi_k^m+\cdots+ f_0
$$
with $\deg_zf_i<\deg_z\phi_k$, we have that at least two distinct terms have the same value
$$
V_k(f(z))=\min_i\{V_{k-1}(f_i)+iV_k(\phi_k)\}.
$$
Thus $V_k(f(z))\ge V_k(\phi_k)$ for all $k$, which implies
$$
\lim_{k\rightarrow \infty}V_k(f(z))=\infty
$$
so that $f\in I(V_{\infty})_{\infty}$.  Now $I(V_{\infty})_{\infty}$ is a proper principal ideal in $K[z]$ and $f$ is ireducible in $K[z]$ so $I(V_{\infty})_{\infty}=(f(z))$.

Finally, 2) implies 1) follows since $I(V_{\infty})_{\infty}\ne (0)$. 
\end{proof}

We observe that if the equivalent conditions of Lemma \ref{Lemma4} hold and $g\in K[z]$ is such that $f\not | g$, then there exists $k$ such that $V_k(g)=V_{\infty}(g)$. This follows since we can find a $\phi_k$ such that $V_k(\phi_k)=V_{\infty}(\phi_k)>V_k(g)$. Then, expanding
$g=g_m\phi_k^m+\cdots+g_0$ with $\deg_zg_i<\deg_z\phi_k$, we have that $V_{\infty}(g)=V_k(g)=V_k(g_0)$.

For the rest of this section, we will assume that $V_0$ has arbitrary rank.
MacLane gives the following explanation of how to find all of the extensions of a $(k-1)$-st stage  approximant $V_{k-1}$ to $f$ over $V_0$ to a $k$-th stage approximant $V_k$ to $f$ over $V_0$. 

We say that $e\in K[z]$ is an ``equivalence unit'' for $V_k$ if there exists an ``equivalence-reciprocal''  $h\in K[z]$ such that $eh\sim 1$ in $V_k$. It is shown in Section 4 of \cite{M2} that $e$ is an equivalence unit if and only if $e$ is equivalent in   $V_k$ to a polynomial $g$ such that $\deg_zg<\deg_z\phi_k$.

 By \cite[Theorem 4.2 ]{M2}, $f$ has an essentially unique (unique up to equivalence in $V_{k-1}$) expression

\begin{equation}\label{eqM6}
f\sim e\phi_{k-1}^{m_0}\psi_1^{m_1}\cdots\psi_t^{m_t}
\end{equation}
in $V_{k-1}$, with $m_0\in \NN$ and $m_1,\ldots,m_t>0$.  Here $e$ is an equivalence unit for $V_{k-1}$  and $\psi_1,\ldots,\psi_t$ are homogeneous key polynomials  over $V_{k-1}$ all not equivalent to  $\phi_{k-1}$ in $V_{k-1}$ and not equivalent in $V_{k-1}$ to  each other. We have that $t>0$ since ${\rm proj}(V_{k-1})>0$. We have that $\phi_{k-1}$ is a homogeneous  key polynomial in $V_{k-1}$ by \cite[Lemma 4.3]{M2}.

If $f$ is a homogeneous key polynomial for $V_{k-1}$, then $V_k=[V_{k-1};V_k(f(z))=\infty]$ is a 
pseudo valuation of $K[z]$ with $I(V)_{\infty}=(f(z))$.

If $f$ is not a homogeneous key polynomial for $V_{k-1}$, then none of the $\psi_i$ are equal to $f$, and we may define a $k$-th stage approximant to $f$ over $V_0$ which is an inductive valuation of $V_{k-1}$ by $V_k=[V_{k-1};V_k(\phi_k)=\mu_k]$ where $\phi_k$ is one of the $\psi_i$. In the expansion (\ref{eqM3}) of $f$,
$$
f=f_m\phi_k^m+\cdots+f_0
$$
$\mu_k$ must be chosen so that ${\rm proj}(V_k)>0$. All $k$-th stage approximants $V_k$ to $f$ extending $V_{k-1}$ are found by the above procedure.

Let $T=\RR\times G_{V_0}$. Given $\alpha,\beta\in G_{V_0}$ and $q\in \RR$, we have the line
$$D=\{(x,\gamma)\in T\mid q\gamma+\alpha x+\beta=0\}
$$
in $T$. When $q\ne 0$, we define the slope of $D$ to be $-\frac{\alpha}{q}\in G_{V_0}\otimes_{\ZZ}\RR$.
Associated to $D$ are two half spaces of $T$,
$$
H^D_{\ge}=\{(x,\gamma)\in T\mid q\gamma+\alpha x+\beta\ge 0\}
$$
and
$$
H^D_{\le}=\{(x,\gamma)\in T\mid q\gamma+\alpha x+\beta\le0\}.
$$
 Given a subset $A$ of $T$, the convex closure of $A$ is ${\rm Conv}(A)=\cap H$ where $H$ runs over the half spaces of $T$ which contain $A$.

The Newton polygon is constructed as on page 500 of \cite{M2} and page 2510 of  \cite{V3}. These constructions are equivalent but slightly different. We use the convention of \cite{M2}. 
 The possible values $\mu_k$ can be conveniently found from the Newton polygon $N(V_{k-1},\phi_k)$. This is constructed by taking the convex closure in  $T$ of
 $$
 A=\{(m-i,\delta)\mid \delta\ge V_{k-1}(f_i), 0\le i\le m\},
 $$
 where the union is over $i$ such that $f_i\ne 0$. A segment $F$ of the boundary of ${\rm Conv}(A)$ is a subset $F$ of ${\rm Conv}(A)$ which is defined by $F={\rm Conv}(A)\cap D$ where $D$ is a line of $T$ such that ${\rm Conv}(A)$ is contained in one of the half spaces $H^D_{\ge}$ or $H^D_{\le}$ defined by $D$ and $F={\rm Conv}(A)\cap D$ contains at least two distinct points.

The  slopes $\mu$ of the segments of $N(V_{k-1},\phi_k)$ satisfying $\mu>V_{k-1}(\phi_k)$ are the possible values of $\phi_k$. The polygon composed of those segments of slope $\mu$ with $\mu>V_{k-1}(\phi_k)$ is called the principal part of the Newton polygon $N(V_{k-1},\phi_k)$.

In the proof of Theorem 5.1 of \cite{M2}, it is shown that for $1\le i\le t$, the principal polygon of $N(V_{k-1},\psi_i)$  (from (\ref{eqM6})) is 
\begin{equation}\label{eqM9}
\{(x,y)\in N(V_{k-1},\psi_i)\mid x\ge m-m_i\}.
\end{equation}
Further, $m_0$ is the smallest exponent $i$ such that in the expansion $f=\sum f_i\phi_{k-1}^i$ with $\deg_z f_i<\deg_z\phi_{k-1}$, we have that $V_{k-1}(f_i\phi_{k-1}^i)=V_{k-1}(f(z))$.

\begin{Remark}\label{RemarkM17}
 If the coefficients of $f(z)$ are all in the valuation ring $R_{V_0}$ of $V_0$, then the coefficients of all key polynomials $\phi_k$ are also in $R_{V_0}$, as is established in \cite[Theorem 7.1]{M2}.
\end{Remark}

%\subsection{Extension of an infinite sequence of approximants to a pseudo valuation}
The following theorem follows from a criterion of \cite{V3}.

\begin{Theorem}\label{Theorem4} Suppose that $V_k$ is a $k$-th approximant to $f$ over $V_0$. Then there exists a pseudo valuation $W$ of $K[z]$ such that $W| K=V_0$, $I(W)_{\infty}=(f(z))$, $W(g)\ge V_k(g)$ for all $g\in K[z]$ and $W(\phi_i)=V_i(\phi_i)$ for $1\le i\le k$.
\end{Theorem}

\begin{proof} As explained in the construction of $V_k$ above, we have that $\phi_k|f$ in $V_{k-1}$, and there exists a key polynomial $\psi$ for $V_k$ with $\psi$ not equivalent to $\phi_k$ in $V_k$ and such that $\psi|f$ in $V_k$. The theorem now follows from \cite[Theorem 1]{V3}.
\end{proof}

\section{An algorithm to construct generating sequences}\label{Sec3}
%\section{The case  of a unique extension}\label{Sec3} 

Let  $V_0$ be a   valuation of a field $K$. Suppose that there exists an algebraically closed field $\mathbf k$ such that $\mathbf k\subset R_{V_0}$ and $R_{V_0}/m_{V_0}\cong \mathbf k$. Let $f(z)\in R_{V_0}[z]$ be an irreducible unitary polynomial.

In this section we give an inductive construction of a sequence of approximants to $f$ over $V_0$, so that the key polynomials constructed have a particularly nice form.  We will call the sequence of approximants ``a realization of the algorithm of Section \ref{Sec3}''. We will prove the following theorem by induction on the index $k$.

\begin{Theorem}\label{TheoremRMA}
Suppose that $W$ is a pseudo valuation of $K[z]$ extending  $V_0$ such that $I(W)_{\infty}=(f(z))$. Then we can construct a sequence of approximants to $f$ over $V_0$
\begin{equation}\label{eqN200}
V_1,\ldots,V_k,\ldots,
\end{equation}
where 
\begin{equation}\label{eq1}
V_i=[V_{i-1};V_i(\phi_i)=W(\phi_i)]
\end{equation}
for all $i$  such that the key polynomials $\phi_i$ satisfy $\phi_1=z$ in $V_0$ and 
\begin{equation}\label{eq21}
\phi_i=\phi_{i-1}^{n_{i-1}}-c_{i-1}\phi_1^{j_1(i-1)}\cdots\phi_{i-2}^{j_{i-2}(i-1)}
\end{equation}
 in $V_{i-1}$ for $2\le i\le k$ with $c_{i-1}\in R_{V_0}$, $n_{i-1}=[G_{V_{i-1}}:G_{V_{i-2}}]$ and $0\le j_l(m)<n_l$ for all $l$ and $m$.
The sequence (\ref{eqN200}) is either of finite length $k$ with $\phi_k=f$ and $V_k(f(z))=\infty$ or the sequence is infinite.
 \end{Theorem}

 Observe that we have that 
 $$
 \phi_i^{n_i}\sim c_i\phi_1^{j_1(i)}\cdots\phi_{i-1}^{j_{i-1}(i)}
 $$
 in $V_i$ for $1\le i\le k-1$, since $\phi_{i+1}$ is a key polynomial over $V_i$.

 The proof of the theorem will be given after we have established Lemmas   \ref{Lemma80} and \ref{Lemma1} and Theorem \ref{Prop1}.

%%%%%%%%%%%%%%%%%%%%%%%%%%

%%%%%%%%%%%%%%%%%%%%%%

%Let $A$ be a  local domain, containing an algebraically closed  field $k$ such that $A/m_A\cong k$ and let  $V_0$ be a   valuation of the quotient field $K$ of $A$ which dominates $A$.
%Assume that $V_0$ is rational (the residue field of the valuation ring of $V_0$ is $k$).
 %Let $G_0$ be the value group of $V_0$ which we can consider to be an ordered subgroup of $\RR$. Let  $S_0$ be the semigroup of $V_0$ in $A$. Let $f(z)\in A[z]$ be an irreducible and monic polynomial. In this section, we will assume that the above assumptions hold.

%Suppose we have constructed approximants  to $f$ over $V_0$ 
%\begin{equation}\label{eq1}
%V_i=[V_{i-1};V_i(\phi_i)=\mu_i]
%\end{equation}
 %for $1\le i\le k$ with key polynomials $\phi_1=z$ in $V_0$ and 
%\begin{equation}\label{eq21}
%\phi_i=\phi_{i-1}^{n_{i-1}}-c_{i-1}\phi_1^{j_1(i-1)}\cdots\phi_{i-2}^{j_{i-2}(i-1)}
%\end{equation}
 %in $V_{i-1}$ for $2\le i\le k$ with $c_{i-1}\in A$ and $0\le j_l(m)<n_l$ for all $l$ and $m$ and $n_i$ is the smallest positive integer such that $n_iV_i(\phi_i)\in G_{i-1}$, for $i\le k-1$ where 
 %$G_i$ is the value group of $V_i$.  
 
 %%%%%%%%%%%%%%%

% We observe that $n_i=[G_i:G_{i-1}]$ for $1\le i\le k-1$ and defining $n_k=[G_k:G_{k-1}]$, 

 \begin{Lemma}\label{Lemma80}  Suppose that  $V_1,\ldots,V_{k}$ satisfy the conclusions (\ref{eq1}) and (\ref{eq21}) of Theorem \ref{TheoremRMA} and we have an  equality 
  $$
 n_{k}V_{k}(\phi_{k})=V_{k}(c_{k}\phi_1^{j_1(k)}\cdots\phi_{k-1}^{j_{k-1}(k)})
$$
in $V_{k}$ with $c_{k}\in K$, $n_k=[G_{V_{k}}:G_{V_{k-1}}]$  and $0\le j_l(k)<n_l$ for all $l$.  Then $c_{k}\in R_{V_0}$.
 \end{Lemma}
 
 \begin{proof} 
  In the case that $k=1$, we have that $W(z)\ge 0$ since $f$ is unitary and the coefficients of $f$ are in $R_{V_0}$. Thus $V_0(c_1)\ge 0$.
 
 Now suppose that $k\ge 2$. 
 Since $n_i$ is the smallest positive integer $m$ such that $mV_i(\phi_i)\in G_{V_{i-1}}$, we  have by repeated Euclidean division that every element $\gamma\in G_{V_k}$ has a unique decomposition as 
 \begin{equation}\label{eq80}
 \gamma=\gamma_0+j_1\mu_1+\cdots+j_k\mu_k
 \end{equation}
 where $\gamma_0\in G_{V_0}$, $\mu_i=W(\phi_i)$ for $1\le i\le k$ and $0\le j_i<n_i$ for $1\le i\le k$.  
 We have from (\ref{eq21}) that
 \begin{equation}\label{eq81}
 n_i\mu_i<\mu_{i+1}\mbox{ for all }1\le i<k.
 \end{equation}
 There is a unique representation
 $$
 n_l\mu_l=\gamma_0+j_1\mu_1+\cdots+j_{l-1}\mu_{l-1}
 $$
 of the form of (\ref{eq80}). It follows from (\ref{eq81}) that
 $$
 j_1\mu_1+\cdots+j_{l-1}\mu_{l-1}<n_l\mu_l.
 $$
 Thus 
 $V_0(c_k)=\gamma_0>0$.
 \end{proof}

\begin{Lemma}\label{Lemma1}  Suppose that 
$V_1,\ldots,V_k$ satisfy the conclusions  (\ref{eq1}) and (\ref{eq21}) of Theorem \ref{TheoremRMA} with
 $I(V_i)_{\infty}=(0)$ for all $i\le k$. Let $A$ be a  local domain whose quotient field is $K$ and suppose that $f(z)\in A[z]$. Further suppose that $A$  is dominated by $V_0$ and that $A$ contains $\mathbf k$ (so  that $A/m_A\cong \mathbf k$). 
Suppose that $c_i\in A$ for $i\le k-1$. 
Then 
we have a graded $\mathbf k$-algebra  isomorphism of  ${\rm gr}_{V_{k}}(A[z])$  with the quotient 
${\rm gr}_{V_0}(A)[\overline \phi_1,\ldots,\overline \phi_{k}]/I$
of the graded polynomial ring ${\rm gr}_{V_0}(A)[\overline \phi_1,\ldots,\overline \phi_{k}]$ over 
${\rm gr}_{V_0}(A)$,
where 
$$
I=(\overline\phi_1^{n_1}-\overline c_1,\overline\phi_2^{n_2}-\overline c_2\overline\phi_1^{j_1(2)},\ldots,
\overline\phi_{k-1}^{n_{k-1}}-\overline c_{k-1}\overline\phi_1^{j_1(k-1)}\cdots\overline\phi_{k-2}^{j_{k-2}(k-1)}).
$$
Here $\overline c_1,\ldots,\overline c_{k-1}$ are the initial forms of $c_1,\ldots,c_{k-1}$ in ${\rm gr}_{V_0}(A)$ and $\overline\phi_i$ has the weight $V_{k}(\phi_i)$ for all $i$.

 Suppose there exists $c\in A$ and $j_i\in \NN$ for $1\le i\le k-1$ with $0\le j_i<n_i$ such that $V_k(\phi_k^{n_k})=V_k(c\phi_1^{j_1}\cdots\phi_{k-1}^{j_{k-1}})$. Then $(\overline\phi_k^{n_k}-\overline c\overline\phi_1^{j_1}\cdots\overline\phi_{k-1}^{j_{k-1}})$ is a prime ideal in ${\rm gr}_{V_k}(A[z])$.
\end{Lemma}

\begin{proof}  Every $g\in A[z]$ has the unique decomposition of (\ref{eqM4}) and Remark \ref{RemarkM20},
$$
g=\sum_ja_j\phi_1^{m_{1,j}}\phi_2^{m_{2,j}}\cdots\phi_k^{m_{k,j}}
$$
with $a_j\in A$, $m_{1,j},\ldots,m_{k,j}\in \NN$ and $0\le m_{i,j}<n_i$ for $i<k$ and 
$$
\begin{array}{lll}
V_k(g)&=&\min_j\{V_0(a_j)+m_{1,j}V_1(\phi_1)+\cdots+m_{k,j}V_k(\phi_k)\}\\
&=&\min_j\{V_k(a_j)+m_{1,j}V_k(\phi_1)+\cdots+m_{k,j}V_k(\phi_k)\}
\end{array}
$$
by (\ref{eqM8}). 

Since ${\rm gr}_{V_k}(A[z])$ is generated by the initial forms of elements of $A[z]$, the natural graded ${\rm gr}_{V_0}(A)$-algebra map
$$
\Psi:{\rm gr}_{V_0}(A)[\overline \phi_1,\ldots,\overline \phi_k]\rightarrow {\rm gr}_{V_k}(A[z])
$$ is a surjection and $I$ is contained in the kernel. A homogeneous element $G$ of ${\rm gr}_{V_0}[\overline\phi_1,\ldots,\overline\phi_k]$ has a unique representation 
$$
G\equiv \overline c\overline\phi_1^{j_1}\cdots\overline \phi_{k-1}^{j_{k-1}}\overline \phi_k^{j_k}\mbox{ mod }I
$$
with $c\in A$, $j_1,\ldots,j_k\in \NN$ and $0\le j_i<n_i$ for $i<k$. Now $\Psi(G)=0$ implies that $c=0$ which implies that $G\equiv 0\mbox{ mod }I$. Thus $\Psi$ is an isomorphism, and the first statement of the lemma follows.

We now prove the second statement. Let 
$$
\psi=\overline\phi_k^{n_k}-\overline c\overline \phi_1^{j_1}\cdots\overline\phi_{k-1}^{j_{k-1}}.
$$
We have that ${\rm gr}_{V_k}(A[z])\cong B[\overline\phi_k]$ is a graded polynomial ring over the domain
$$
B={\rm gr}_{V_0}(A)[\overline\phi_1,\ldots,\overline\phi_{k-1}]/(\overline \phi_1^{n_1}-\overline c_1,\ldots,
\overline\phi_{k-1}^{n_{k-1}}-\overline c_{k-1}\overline\phi_1^{j_1(k-1)}\cdots\overline\phi_{k-2}^{j_{k-2}(k-1)}).
$$
Let $L$ be an algebraic closure of the quotient field of $B$. Choose $t\in L$ such that
$t^{n_k}=\overline c\overline\phi_1^{j_1}\cdots\overline\phi_{k-1}^{j_{k-1}}$. Then giving $t$ the weight $V_k(\phi_k)$, we have that $B[t]$ is a graded domain which is a free $B$-module of rank $n_k$, since $V_k(\phi_k)$ has order $n_k$ in $G_k/G_{k-1}$, and so $1,t,\ldots t^{n_k-1}$ is a $B$-basis of $B[t]$. 
We have a natural surjection of graded $B$-modules
\begin{equation}\label{eqM33}
B[\overline\phi_k]/(\psi)\rightarrow B[t].
\end{equation}
Now $B[\overline\phi_k]/(\psi)$ is a also a free $B$-module of rank $n_k$, as $1,\overline\phi_k,\ldots,\overline\phi_k^{n_k-1}$ is a $B$-basis. Thus (\ref{eqM33}) is an isomorphism, and so $B[\overline\phi_k]/(\psi)$ is a domain.
\end{proof}

Suppose that $G$ is a totally ordered Abelian group. Let $U=G\otimes_{\ZZ}\RR$, $d\in \ZZ_+$ and $\gamma\in G$. Since $\ZZ$ is a principal ideal domain, we have that 
\begin{equation}\label{eqG4}
\left(\frac{1}{d}\ZZ\gamma\right)\cap G=\frac{1}{\overline m}\ZZ\gamma\ \ {\rm for\  some}\ \overline m\in \ZZ_+.
\end{equation}
Indeed, we must have 
\begin{equation}\label{eq*N1}
\left(\frac{1}{d}\ZZ\gamma\right)\cap G=\frac{a}{d}\ZZ\gamma
\end{equation}
 for some $a\in \ZZ_+$. Now $\gamma\in \frac{a}{d}\ZZ\gamma$ implies $a|d$, and so there exists $\overline m\in \ZZ_+$ such that $\frac{1}{\overline m}=\frac{a}{d}$. \par\noindent This implies:
\begin{equation}\label{eqG3}
\frac{1}{\overline m}(d,\gamma)\in \ZZ\oplus G.
\end{equation}
\par\noindent We shall need the following fact:
\begin{equation}\label{eqG1}
\mbox{For $n, q\in \ZZ_+$, $\frac{q}{n}(d,\gamma)\in \ZZ\oplus G$ if and only if $n$ divides $qd$ and $\frac{q}{n}=\frac{e}{\overline m}$ for some $e\in \ZZ_+$.}
\end{equation}
For the reader's convenience, we give a proof of (\ref{eqG1}). Suppose that $\frac{q}{n}(d,\gamma)\in \ZZ\oplus G$. Then $n$ divides $qd$ and writing $qd=rn$ we see that $\frac{r}{d}\ZZ\gamma\subset G$ so that it follows from (\ref{eqG4}) that $\frac{r}{d}=\frac{q}{n}$ is an integral multiple of $\frac{1}{\overline m}$. The converse follows from (\ref{eqG3}).

\begin{Theorem}\label{Prop1}  Suppose that we have constructed approximants $V_i=[V_{i-1},V_i(\phi_i)=W(\phi_i)]$ for $1\le i\le k-1$ to $f$ over $V_0$ satisfying the conclusions of Theorem \ref{TheoremRMA}, $V_{k-1}(\phi_{k-1})<\infty$  and  we have an equivalence in $V_{k-1}$
\begin{equation}\label{eqN10}
f\sim e\phi_{k-1}^{m_0}\overline \psi_1^{m_1}\cdots\overline \psi_t^{m_t}
\end{equation}
of the form of (\ref{eqM6})
with $m_0\in\NN$ and $m_1,\ldots,m_t\in \ZZ_+$ such that $e$ is an equivalence unit for $V_{k-1}$, $\overline \psi_1\ldots,\overline\psi_t$ are homogeneous key polynomials  over $V_{k-1}$ such that there are expressions
$$
\overline\psi_i=\phi_{k-1}^{n_{k-1}}-\epsilon_{k-1,i}c_{k-1}\phi_1^{j_1(k-1)}\cdots\phi_{k-2}^{j_{k-2}(k-1)}
$$
with $c_{k-1}\in R_{V_0}$ non zero, $\epsilon_{k-1,i}\in \mathbf k$ distinct and nonzero, and $0\le j_i(k-1)<n_i$ for all $i$. The $\overline\psi_i$ define approximants to $f$ as explained after (\ref{eqM6}).

Then there exists a unique $\overline\psi_i$ such that $W(\overline\psi_i)>V_{k-1}(\overline\psi_i)$ and setting $\phi_k=\overline\psi_i$, there exists a unique segment $S$ of the principal part of the Newton polygon $N(V_{k-1},\phi_k)$ which has slope $s=W(\phi_k)$.

 Defining $V_k=[V_{k-1},V_k(\phi_k)=W(\phi_k)]$, we have that  $V_k$ is an approximate to $f$ over $V_0$, 
 such that the approximants $V_1,\ldots, V_k$ satisfy the conclusions of Theorem \ref{TheoremRMA}.
 
 Now suppose that $V_k(\phi_k)<\infty$.
The Newton polygon $N(V_{k-1},\phi_k)$ is computed from the expansion 
\begin{equation}\label{eq3}
f=\sum f_i\phi_k^i
\end{equation}
 with 
$\deg_zf_i<\deg_z\phi_k$.
Let $(m-i_1,\beta_1)$ be the lowest point on the segment $S$ and let $(m-i_0, \beta_0)$ be the highest point. 
Let
\begin{equation}\label{eq4}
F_{k,s}(\phi_k)=\sum f_i\phi_k^i,
\end{equation}
where the sum is restricted to $i$ such that $(m-i,V_{k-1}(f_i))$ is on $S$.
Then there exists a polynomial  in $\phi_k$
\begin{equation}\label{eq22}
G_{k,s}(\phi_k)=\sum g_i\phi_k^i
\end{equation}
with $g_i\in K[z]$ such that  the $i$ such that $g_i$ is not zero are exactly the $i$ such that $f_i$ is a coefficient of $F_{k,s}$  and 
 $g_i\sim f_i$ in $V_{k-1}$ for all such $i$. 
Further,  factoring the right side of (\ref{eq22}) as a polynomial in $\phi_k$, 
\begin{equation}\label{eq2}
 G_{k,s}(\phi_k)=f_{i_1}^{\overline m+1}\phi_{k}^{i_0}\psi_1^{a_1}\cdots\psi_t^{a_t}
\end{equation}
 where 
  \begin{equation}\label{eq63}
 \psi_i=\phi_{k}^{n_k}-\epsilon_{k,i}c_{k}\phi_1^{j_1(k)}\cdots\phi_{k-1}^{j_{k-1}(k)}
 \end{equation}
 with $c_{k}\in R_{V_0}$  nonzero, $\epsilon_{k,i}\in \mathbf k$ are distinct and nonzero, $0\le j_i(k)<n_i$ for all $i$. 
  Further, we have that 
 $f_{i_1}^{\overline m+1}$ is an equivalence unit in $V_k$, 
 $$
 n_k=[G_k:G_{k-1}]
 $$
 and  the $\psi_i$ are homogeneous key polynomials in $V_k$.   Also, 
 there is a $V_{k}$ equivalence  
\begin{equation}\label{eq20}
f\sim G_{k,s}(\phi_k)
\end{equation}
 in $V_{k}$.
  \end{Theorem}

\begin{proof}  The fact that there exists a $\overline\psi_i$ such that $W(\overline\psi_i)>V_{k-1}(\overline\psi_i)$ follows  from the equivalence relation  (\ref{eqN10}), since  $W(f(z))=\infty$ and $W(e\phi_{k-1}^{m_0})=V_{k-1}(e\phi_{k-1}^{m_0})$.
 Uniqueness of $\overline\psi_i$ follows since the $\epsilon_{k-1,i}$ are distinct.  The existence of a segment $S$ of the principal part of the Newton polygon $N(V_{k-1},\phi_k)$ with slope $s=W(\phi_k)$, follows from Theorem \ref{Theorem4} and the discussion of Subsection \ref{SubSecMA}.
 The fact that  upon setting $\phi_k=\overline\psi_i$, we have that $V_k=[V_{k-1},V_k(\phi_k)=W(\phi_k)]$ is an approximate to $f$ over $V_0$ then follows since ${\rm proj}(V_k)$ is positive, as $W(f(z))=\infty$, and the fact that the
  approximants $V_1,\ldots, V_k$ satisfy the conclusions of Theorem \ref{TheoremRMA} follows from our assumptions on the $\phi_i$ for $i\le k$.

Let $y=sx+r$ be the equation of the line containing the segment $S$, so that 
$$
s=\frac{\beta_0-\beta_1}{i_1-i_0}.
$$
   Let $\overline m$ be the largest positive integer such that
\begin{equation}\label{eq23}
\frac{1}{\overline m}(i_1-i_0, \beta_0-\beta_1)\in \ZZ\oplus G_{k-1}.
\end{equation}
Here $\overline m$ is as defined in (\ref{eqG4}), with $d=i_1-i_0$, $\gamma=\beta_0-\beta_1$ and $G=G_{k-1}$.
 Let $(\overline b,\overline c)=\frac{1}{\overline m}(i_1-i_0, \beta_0-\beta_1)$. If $V_{k-1}(f_i)-s(m-i)=r$, then
 $$
 (m-i,V_{k-1}(f_i))=(m-i_1,\beta_1)+\lambda(\overline b,\overline c)=\left(m-i_1+\lambda\overline b,\beta_1+\lambda\left(\frac{\beta_0-\beta_1}{\overline m}\right)\right)
 $$
 for some $\lambda\in \NN$ with $0\le \lambda\le \overline m$ (this follows from (\ref{eqG1})). 
 Using the relations (\ref{eq21}) for $2\le i\le k$, 
 there exists
 $$
 h=c_k\phi_1^{j_1(k)}\cdots \phi_{k-1}^{j_{k-1}(k)}\in K[z]
 $$
 with $c_k\in K$ and $0\le j_l(k)<n_l$ for $1\le l<k$ such that $V_{k-1}(h)=\frac{\beta_0-\beta_1}{\overline m}$. We  have that 
 $$
 \begin{array}{lll}
 F_{k,s} &=& \sum_{\tau=0}^{\overline m}f_{i_0+\tau\overline b}\phi_k^{i_0+\tau\overline b}\\
 &=&\phi_k^{i_0}(\sum_{\tau=0}^{\overline m}f_{i_0+\tau\overline b}\phi_k^{\tau\overline b})
 \end{array}
  $$
 where 
 \begin{equation}\label{eq5}
 \begin{array}{lll}
 V_{k-1}(f_{i_0+\tau \overline b})&=&s(m-(i_0+\tau\overline b))+r=-\tau\left(\frac{\beta_0-\beta_1}{\overline m}\right)+\beta_0\\
 &=&(\overline m-\tau)\left(\frac{\beta_0-\beta_1}{\overline m}\right)+\beta_1=V_{k-1}(h^{\overline m-\tau})+V_{k-1}(f_{i_1}).
 \end{array}
 \end{equation}
%We have that
% $$
% ({\rm In}(\phi_k))=(\overline\phi_{k-1}^{n_{k-1}}-\overline c_{k-1}\overline\phi_1^{j_1(k-1)}\cdots \overline \phi_{k-2}^{j_{k-2}(k-1)})
% $$
 %is a graded prime ideal in ${\rm gr}_{V_{k-1}}(K[z])$, since $\phi_k$ is a key polynomial in $V_{k-1}$, and thus is equivalence irreducible in $V_{k-1}$.
%Let
%$$
%B={\rm gr}_{V_{k-1}}(K[z])/({\rm In}(\phi_k)).
%$$
By (\ref{eq5}),  and since $V_0$ is rational ($R_{V_0}/m_{V_0}=\mathbf k$), there exist $\gamma_{\tau}\in \mathbf k$ such that 
\begin{equation}\label{eq7}
\gamma_{\tau}{\rm In}(h^{\overline m-\tau}){\rm In}(f_{i_1})={\rm In}(f_{i_0+\tau \overline b})
\end{equation}
 in ${\rm gr}_{V_{k-1}}(K[z])$. 
Define $G_{k,s}(\phi_k)$ by
$$
\begin{array}{lll}
G_{k,s}(\phi_k)&=&\phi_k^{i_0}(\sum_{\tau=0}^{\overline m}\gamma_{\tau}h^{\overline m-\tau}f_{i_1}\phi_k^{\tau \overline b})\\
&=& f_{i_1}^{\overline m+1}\phi_k^{i_0}h^{\overline m}(\sum_{\tau=0}^{\overline m} \gamma_{\tau}(h^{-1}\phi_k^{\overline b})^{\tau})\\
&=& f_{i_1}^{\overline m+1}\phi_k^{i_0}h^{\overline m}\prod_{j=1}^{\overline m}((h^{-1}\phi_k^{\overline b})-\alpha_j)\\
&=&f_{i_1}^{\overline m+1}\phi_k^{i_0}\prod_{j=1}^{\overline m}(\phi_k^{\overline b}-\alpha_jh)
\end{array}
$$
for suitable nonzero $\alpha_j\in \mathbf k$. 

We will compute the order
$$
[G_k:G_{k-1}]=[(G_{k-1}+s\ZZ):G_{k-1}].
$$
We will show that the order $n_k$ is  $n_k=\overline b=\frac{i_1-i_0}{\overline m}$. Since $s=\frac{\beta_0-\beta_1}{i_1-i_0}$,
$$
\overline b s=\frac{\beta_0-\beta_1}{\overline m}\in G_{k-1}.
$$
Now with $a$ as defined in (\ref{eq*N1}), with $d=i_1-i_0$, $\gamma=\beta_0-\beta_1$ and $G=G_{k-1}$, we have that $a=\frac{d}{\overline m}=\overline b$. 

Suppose $n\in\ZZ_+$  and $ns\in G_{k-1}$.
Now 
$$
ns=n\left(\frac{\beta_0-\beta_1}{i_1-i_0}\right)=\frac{n}{d}\gamma\in G_{k-1}.
$$
which implies that $a=\overline b\mid n$.

Thus we have that
$$
n_k=\overline b=[G_k:G_{k-1}].
$$
We now have that $c_k\in R_{V_0}$ by Lemma \ref{Lemma80}.

The fact  that $f\sim G_{k,s}$ in $V_k$ follows since $\gamma_{\tau}h^{\overline m-\tau}f_{i_1}\sim f_{i_0+\tau\overline b}$ in $V_k$, which follows from (\ref{eq7}), the facts that by (\ref{eqM8}),
$$
V_k(h)=V_k(c_k)+
\sum_{i=1}^{k-1}j_i(k)V_k(\phi_i)=V_0(c_k)+\sum_{i=1}^{k-1}j_i(k)V_i(\phi_i)=V_{k-1}(h)
$$
and 
$$
V_k(f_i)=V_{k-1}(f_i)
$$
for all $i$ since $\deg_z f_i<\deg_z\phi_k$.

We know that $\phi_{k}$ is a key polynomial in $V_k$ as discussed after (\ref{eqM6}). Finally, we  verify that each $\psi_i=\phi_k^{n_k}-\epsilon_{k,i}c_k\phi_1^{j_1(k)}\cdots \phi_{k-1}^{j_{k-1}(k)}$ is a key polynomial in $V_k$. 
By Lemma \ref{Lemma1}, the ideal
$$
({\rm In}(\psi_i))=(\overline \phi_k^{n_k}-\epsilon_{k,i}\overline c_k\overline \phi_1^{j_1(k)}\cdots \overline \phi_{k-1}^{j_{k-1}(k)})
$$
 is a prime ideal in ${\rm gr}_{V_k}(R_{V_0}[z])$, where $\overline c_k={\rm In}(c_k)$, and $\overline\phi_i={\rm In}(\phi_i)$. Thus $\psi_i$ is equivalence irreducible in $V_k$ as a polynomial in $R_{V_0}[z]$. Since every non zero element of $R_{V_0}[z]$ is a unit in $K$ this implies that it is equivalence irreducible in $V_k$ as a polynomial in $K[z]$.
 We have that $\psi_i$ is minimal in $V_k$ by (\ref{eqM13}). Since $\psi_i$ has the leading coefficient 1 and $\deg_z\psi_i>0$, we have that $\psi_i$ is a key polynomial over $V_k$.
  \end{proof}

Constructions similar to those used in the proof of theorem 4.4 can be found in \cite{HMOS}, pp. 17-18.

  We now give the proof of Theorem \ref{TheoremRMA}. Set $\phi_1=z$ and $V_1=[V_0;V_1(\phi_1)=W(\phi_1)]$,
  which is an approximant to $f$ over $V_0$ since $W(f(z))=\infty$. By a simplification of the proof of Theorem \ref{Prop1}, we have that $f\sim ez^{m_0}\overline\psi_1^{m_1}\cdots \overline\psi_t^{m_t}$ in $V_1$,
  where $e$ is an equivalence unit in $V_1$ and $\overline \psi_i=z^{n_1}-\epsilon_{1,i}c_1$ with $c_1\in R_{V_0}$ and $\epsilon_{1,i}\in k$ are nonzero and distinct. 
  
  Now the conclusions of the theorem follow from induction using Theorem \ref{Prop1}.
  
  As pointed by the referee, another point of view on theorem 4.1 can be  obtained from  \cite[Formula (3.8)]{HMOS} applied to our situation ; note here that one should prove that $c_i \in R_{V_0}.$

\begin{Proposition}\label{Prop5} Suppose that there is a unique extension of $V_0$ to a pseudo valuation $W$  of $K[z]$ with $I(W)_{\infty}=(f(z))$ and we have constructed a finite or infinite sequence of approximants $V_1,\ldots,V_k,\ldots$ to $f$ over $V_0$ satisfying the conclusions of Theorem \ref{TheoremRMA}.
 Then we have that for $k\ge 2$, with  notation as in (\ref{eq2}), setting $e_k=i_0$,
\begin{equation}\label{eq81*}
f\sim \phi_k^{e_k}\mbox{ in }V_{k-1}
\end{equation}
where 
  \begin{equation}\label{eqT1}
 \phi_{k+1}=\phi_{k}^{n_k}-c_{k}\phi_1^{j_1(k)}\cdots\phi_{k-1}^{j_{k-1}(k)}
 \end{equation}
 with $c_{k}\in R_{V_0}$  nonzero, $0\le j_i(k)<n_i$ for all $i$ 
and
\begin{equation}\label{eq82}
\deg_zf=e_k\deg_z\phi_k.
\end{equation}
\end{Proposition}

\begin{proof} We use the notation of the statement and proof of Theorem \ref{Prop1}.
By Theorem \ref{Theorem4}, every realization of the algorithm to construct a $k$-th stage approximant $V_k$ to $f$ over $V_0$  extends to the construction of a pseudo valuation $U$ extending $V_0$ with $I(U)_{\infty}=(f(z))$.  Since $W$ is unique, every realization of the algorithm must extend to the construction of $U=W$.
%Since $W$ is the unique such extension, we have that 
%\begin{equation}\label{eq80*}
%V_k(\phi_k)=W(\phi_k)\mbox{ for all }k.
%\end{equation}

We will prove  the following equations,
 \begin{equation}\label{eq72}
 f\sim \phi_k^{e_k}\mbox{ in }V_{k-1}\mbox{ with }\deg_zf=e_k\deg_z\phi_k
 \end{equation}
 and for all $k\ge 2$, (\ref{eq22}) of Theorem \ref{Prop1} satisfies
\begin{equation}\label{eq73}
f\sim \psi_1^{a_1}\cdots\psi_t^{a_t}\mbox{ in }V_k\mbox{ with }\deg_zf=a_1\deg_z\psi_1+\cdots+a_t\deg_z\psi_t.
\end{equation}

We will establish (\ref{eq72}) and (\ref{eq73}) for $k=2$. Since the extension is unique,  every realization of the algorithm must extend to the construction of $W$, so $N(V_0;\phi_1=z)$ has a unique segment. Let $\mu_1=s=W(z)$ be the slope of this segment, so that 
$$
V_1=[V_0;V_1(\phi_1)=\mu_1].
$$
 Expand
$$
f=z^d+f_{d-1}z^{d-1}+\cdots+f_0
$$
with $f_i\in K$. Since $N(V_0,\phi_1)$ has a unique segment, $i_0=0$, $i_1=d$ and $f_{i_1}=1$ in (\ref{eq2}) for $k=1$, so by (\ref{eq20}) and (\ref{eq2}) for $k=1$,
\begin{equation}\label{eq74}
f\sim G_{1,s}(\phi_1)=\psi_1^{a_1}\cdots\psi_t^{a_t}
\end{equation}
in $V_1$, where
\begin{equation}
\psi_i=\phi_1^{n_1}-\epsilon_{1,i}c_1
\end{equation}
from (\ref{eq63}).  Suppose that $t>1$. Any choice of $\psi_i$ is a key polynomial for $V_1$, and if $W_2=[V_1;V_2(\psi_i)=\mu_2]$ is an approximant extending $V_1$, then since  every realization of the algorithm must extend to the construction of $W$, as observed in the first part of the proof, we have that 

$$
W(\psi_i)=W_2(\psi_i)=\mu_2>n_1V_1(\phi_1)=V_0(c_1).
$$
For $j\ne i$,
$$
\psi_j=\psi_i+(\epsilon_{1,i}-\epsilon_{1,j})c_1
$$
so for $j\ne i$,
$$
W(\psi_j)=W(\psi_i+(\epsilon_{1,i}-\epsilon_{1,j})c_1)=V_0(c_1).
$$
This contradiction shows that  $t=1$ in (\ref{eq74}) and so $f\sim \phi_2^{e_2}$ in $V_1$ with $\deg_zf=e_2\deg_z\phi_2$, establishing (\ref{eq72}) for $k=2$.

From (\ref{eq72}) for $k=2$, we have that there is an expression
$$
f=\phi_2^{e_2}+f_{e_2-1}\phi_2^{e_2-1}+\cdots+f_0
$$
with $\deg_zf_i<\deg_z\phi_2$ for all $i$. From (\ref{eqM9}), we then have that the principal part of the Newton polygon $N(V_1,\phi_2)$ is the entirety of $N(V_1,\phi_2)$. Further, by uniqueness of the extension of $V_0$, we have that $N(V_1,\phi_2)$ has a unique segment, so $i_0=0$, $i_1=e_2$ and $f_{i_1}=1$ in (\ref{eq2}) for $k=2$, so 
$$
f\sim G_{2,s}(\phi_2)=\psi_1^{a_1}\cdots\psi_t^{a_t}
$$
in $V_2$ with the $\psi_i$ given by (\ref{eq63}) for $k=2$, establishing (\ref{eq73}) in $V_2$ for $k=2$, with $\deg_zf=a_1\deg_z\psi_1+\cdots+a_t\deg_z\psi_t$.

Now by induction on $k$, repeating the argument for the case $k=2$ with the application of Theorem \ref{Prop1}, we obtain the conclusions of Proposition \ref{Prop5}.
\end{proof}
Formulas (\ref{eq81*}) and (\ref{eq82})  also follow from \cite[Theorem 3.1]{V3}, and then formula (\ref{eqT1}) follows from Theorem \ref{Prop1}.
 
 \section{When the degree is prime to $p$ and the extension is unique}

\begin{Theorem}\label{Theorem1}   Suppose that $A$ is a  local domain which contains an algebraically closed field $\mathbf k$ such that $A/m_A\cong \mathbf k$. Let $K$ be the quotient field of $A$ and suppose that $V_0$ is a  valuation of $K$ which dominates $A$, such that the residue field of the valuation ring of $V_0$ is $\mathbf k$. Suppose that 
$f(z)\in A[z]$ is unitary and irreducible, there is a unique extension of $V_0$ to a valuation $\omega$ of $K[z]/(f(z))$ and  the characteristic $p$ of $\mathbf k$ does not divide $\deg_zf$. Let $W$ be the associated 
pseudo valuation of $K[z]$ such that $I(W)_{\infty}=(f(z))$ in $K[z]$. Then there exists a realization of the algorithm of Section \ref{Sec3} constructing approximants $V_1,\ldots,V_k$  to $f$ over $V_0$ 
satisfying equations (\ref{eq1}) and (\ref{eq21}) for all $i\le k$
such that  $W=V_k$. We have that
$$
\deg_zf=[G_{\omega}:G_{V_0}]=[G_{V_k}:G_{V_0}].
$$
Further, with the notation of (\ref{eq21}),  $c_i\in A$ for all $1\le i\le k$, and 
$$
{\rm gr}_{\omega}(A[z]/(f(z)))\cong {\rm gr}_{V_0}(A)[\overline \phi_1,\ldots,\overline \phi_{k-1}]/I
$$
where
$$
I=(\overline \phi_1^{n_1}-\overline c_1,\overline \phi_2^{n_2}-\overline c_2\overline\phi_1^{j_1(2)},\ldots,\overline \phi_{k-1}^{n_{k-1}}-\overline c_{k-1}\overline\phi_1^{j_1(k-1)}\overline\phi_2^{j_2(k-1)}\cdots\overline \phi_{k-2}^{j_{k-2}(k-1)})
$$
is a finitely generated and presented ${\rm gr}_{V_0}(A)$-module.
\end{Theorem}

\begin{proof}  Suppose by induction on $i$ that we have constructed approximants $V_1,\ldots,V_i$ to $f$ over $V_0$ satisfying equations (\ref{eq1}) and (\ref{eq21}) with
$c_1,\ldots,c_{i-1}\in A$ and  that $\phi_i$ is not equal to $f$. 
By Theorem \ref{Prop1} and Proposition \ref{Prop5},  $f\sim G_i= \phi_{i+1}^{e_{i+1}}$ in $V_i$, with $\phi_{i+1}$ a key polynomial over $V_i$ such that 
\begin{equation}\label{eq8}
\phi_{i+1}=\phi_{i}^{n_{i}}-c_i\phi_1^{j_1(i)}\cdots \phi_{i-1}^{j_{i-1}(i)}\mbox{ and }\deg_zf=e_{i+1}\deg_z\phi_{i+1}
\end{equation}
for some nonzero $c_i\in R_{V_0}$.
 Expanding  
\begin{equation}\label{eqU1}
f=\sum f_j\phi_i^j
\end{equation}
 in $K[z]$, with $\deg_z f_j<\deg_z\phi_i$, let $F=\sum f_j\phi_i^j$ where the sum is restricted to $f_j$ such that $V_{i-1}(f_j)+j\mu_i$ (with $\mu_i=W(\phi_i)$) is minimal, and expanding 
$G_i$ as a polynomial in $\phi_{i}$, we see that 
the coefficients of $G_i=\phi_{i}^{n_{i}e_{i+1}}-e_{i+1}c_i\phi_1^{j_1(i)}\cdots \phi_{i-1}^{j_{i-1}(i)}\phi_{i}^{n_i(e_{i+1}-1)}+\cdots$ as a polynomial in $\phi_i$ and of the coefficients $f_j$ in the expansion  
 $F=\sum f_j\phi_i^j$   must be equivalent in $V_{i-1}$ by Theorem \ref{Prop1}. 
 
 Now $e_{i+1}n_i\deg_z\phi_i=\deg_zf$, so since we assume that  $p$ does not divide $\deg_zf$, we have that $p$ does not divide $e_{i+1}$. Comparing the expansions of $F$ and $G_i$, we see that 
 $$
 0\ne f_{n_i(e_{i+1}-1)}\sim g_{n_i(e_{i+1}-1)}=-e_{i+1}c_i\phi_1^{j_1(i)}\cdots \phi_{i-1}^{j_{i-1}(i)}
 $$
 in $V_{i-1}$. Since $\deg_zf_{n_i(e_{i+1}-1)}<\deg_z\phi_i$ and   $c_1,\ldots,c_{i-1}\in A$ by induction, by (\ref{eqM4}) and Remark \ref{RemarkM20}, 
  $f_{n_i(e_{i+1}-1)}$ has a unique expansion (with only finitely many terms)
 \begin{equation}\label{eqU2}
 f_{n_i(e_{i+1}-1)}=\sum_{\alpha\ge 1} a_{\sigma_1(\alpha),\ldots,\sigma_{i-1}(\alpha)}\phi_1^{\sigma_1(\alpha)}\cdots\phi_{i-1}^{\sigma_{i-1}(\alpha)}
\end{equation}
 with
 $$
 W(a_{\sigma_1(\alpha),\ldots,\sigma_{i-1}(\alpha)}\phi_1^{\sigma_1(\alpha)}\cdots\phi_{i-1}^{\sigma_{i-1}(\alpha)})
 <W( a_{\sigma_1(\alpha+1),\ldots,\sigma_{i-1}(\alpha+1)}\phi_1^{\sigma_1(\alpha+1)}\cdots\phi_{i-1}^{\sigma_{i-1}(\alpha+1)})
 $$ 
 for all $\alpha$,
 $0\le \sigma_l(\alpha)<n_l$ for $1\le l\le i-1$ and $a_{\sigma_1(\alpha),\ldots,\sigma_{i-1}(\alpha)}\in A$. Thus the minimum value term in $V_{i-1}$ in this expansion is
 $$
 a_{\sigma_1(1),\ldots,\sigma_{i-1}(1)}\phi_1^{\sigma_1(1)}\cdots\phi_{i-1}^{\sigma_{i-1}(1)}
 $$
 and so 
 $$
 j_l(i)=\sigma_l(1)\mbox{ for }1\le l\le i-1
 $$
 and
 $$
 -e_{i+1}c_i\sim a_{\sigma_1(1),\ldots,\sigma_{i-1}(1)}
 $$
 in $V_0$. Replacing $c_i$ with $-\frac{1}{e_{i+1}} a_{\sigma_1(1),\ldots,\sigma_{i-1}(1)}$ in (\ref{eq8}), we have that $c_i\in A$.
 
 Suppose $n_i=1$, so that $e_{i+1}=e_i$. Then substituting (\ref{eq8}) and (\ref{eqU2}) into (\ref{eqU1}), we obtain
 $$
 f=\phi_{i+1}^{e_i}+(\sum_{k\ge 2} a_{\sigma_1(k),\ldots,\sigma_{i-1}(k)}\phi_1^{\sigma_1(k)}\cdots\phi_{i-1}^{\sigma_{i-1}(k)})\phi_{i+1}^{e_i-1} +\sum_{j=2}^{e_i-2}f_j'\phi_{i+1}^j
 $$
 where $\deg_zf_j'<\deg_z\phi_{i+1}=\deg_z\phi_i$ for all $j$. Since (\ref{eqU2}) is a finite sum, we can only have $n_i=1$ for finitely many consecutive $i$.
 
 Since $\deg_zf=e_in_1\cdots n_{i-1}$ for all $i$, we must have that the algorithm terminates in a finite number of iterations $k$. We then have that $\phi_k=f$ and $W=V_k$.
 
 The final statement on the structure of ${\rm gr}_{\omega}(A[z]/(f(z)))$ now follows from Lemma \ref{Lemma1}.

\end{proof}

As an immediate consequence of Theorem \ref{Theorem1}, we have the following example, which allows us to easily compute the associated graded rings and valuation semigroups of many examples, including the rational double point singularities in dimension two,
since the semigroups of valuations dominating two dimensional regular local rings are completely known (\cite{Sp}. \cite{CV1}).

\begin{Example}\label{Ex40} Let $\mathbf k$ be an algebraically closed field of characteristic $p\ne 2$, and $A=\mathbf k[[x_1,\ldots,x_n]]$ be a power series ring over $\mathbf k$. Let $f(z)=z^2+az+b$ with $a,b\in m_A$ be irreducible and let $B=A[z]/(f(z))$. Suppose that $\nu$ is a valuation of the quotient field of $A$ which dominates $A$ and such that $R_{\nu}/m_{\nu}=\mathbf k$. 

Suppose that $\nu$  has a unique extension $\omega$  to the quotient field of $B$ which dominates $B$.  Then there exists $g\in m_A$ such that setting $\overline z=z-g$, we have that
\begin{enumerate}
\item[1)] $\omega(\overline z)$ is a generator of $G_{\omega}/G_{\nu}\cong \ZZ/2 \ZZ$ and
\item[2)] ${\rm gr}_{\omega}(B)={\rm gr}_{\nu}(A)[{\rm in} (\overline z)]\cong {\rm gr}_{\nu}(A)[\overline\phi]/(\overline\phi^2-\overline c)$ for some $\overline c\in {\rm gr}_{\nu}(A)$.
\end{enumerate}

In constrast, if $\nu$ does not have a unique extension to the quotient field of $B$ which dominates $B$, then 
it can happen that ${\rm gr}_{\omega}(B)$ is not a finitely generated ${\rm gr}_{\nu}(A)$-module (as will follow from Example \ref{IEx2}).
\end{Example}

The good conclusions of Theorem \ref{Theorem1} may fail if either the extension is not unique or $p$ divides $\deg_zf$.
In \cite[Example 8.1]{T3}, an example of Guillaume Rond is presented which shows that the conclusions of Theorem \ref{Theorem1} may fail if the extension of $V_0$ to a valuation of $K[z]/(f(z))$ is not unique and $p\not|\deg_zf$.

\begin{Example}\label{Ex1}
The conclusions of Theorem \ref{Theorem1} may fail if the characteristic $p$ of the field $\mathbf k$ divides the degree of $f(z)$. In our example, $f(z)$ is separable and $V_0$  has a unique extension to $K[z]/(f(z))$.
\end{Example}

%In the example, their is no defect (defined in  (\ref{eqefd})), as is seen from our construction, or since $V_0$ is an Abhyankar valuation, from Theorem 1 of \cite{K2}.

We now give the construction of the example. Let $\mathbf k$ be an algebraically closed field of characteristic 2 and let $A=\mathbf k[x_1,x_2]_{(x_1,x_2)}$ be a localization of a two dimensional  polynomial ring over $\mathbf k$.
Let $K$ be the quotient field of $A$.   Let $V_0$ be the rank 1  valuation on $K$ defined by $V_0(x_1)=1$ and $V_0(x_2)=\sqrt{37}$, so that $G_{V_0}=\ZZ+\sqrt{37}\ZZ$. Let
$$
f(z)=z^4+x_1^{317}z+x_1^4x_2^2+x_2^{31}.
$$
We have that $f(z)$ is an irreducible, separable polynomial in $K[z]$.

Setting $\phi_1=z$, we have that the Newton polygon $N(V_0,\phi_1)$ has only one segment, from $(0,0)$ to $(4,4+2\sqrt{37})$. The slope of this segment is $1+\frac{1}{2}\sqrt{37}$, giving the first step approximant to $f$ over $V_0$, $V_1=[V_0;V_1(\phi_1)=1+\frac{1}{2}\sqrt{37}]$. We have that $G_{V_1}=\ZZ+\frac{\sqrt{37}}{2}\ZZ$.

Now $f\sim (z^2+x_1^2x_2)^2$ in $V_1$ and $V_1(z)\not\in G_{V_0}$ so $\phi_2=z^2+x_1^2x_2$ is a key polynomial over $V_1$. We have that
$$
f=\phi_2^2+x_1^{317}z+x_2^{31}
$$
so the principal part of $N(V_1,\phi_2)$ is equal to $N(V_1,\phi_2)$, which has only one segment, from $(0,0)$ to $(2,31\sqrt{37})$. The slope is $\frac{31}{2}\sqrt{37}$, giving the 2-nd step approximant to $f$ over $V_0$, $V_2=[V_1;V_2(\phi_2)=\frac{31}{2}\sqrt{37}]$, with $G_{V_2}=G_{V_1}$. We have that
$$
f=(\phi_2+zx_1^{-1}x_2^{15})^2+\phi_2x_1^{-2}x_2^{30}+x_1^{317}z
$$
so that $f\sim (\phi_2+zx_1^{-1}x_2^{15})^2$ in $V_2$. Thus
\begin{equation}\label{eqE1}
\phi_3=\phi_2+zx_1^{-1}x_2^{15}
\end{equation}
is a key polynomial for $V_2$. We have that
$$
f=\phi_3^2+x_1^{-2}x_2^{30}\phi_3+x_1^{-3}x_2^{45}z+x_1^{317}z
$$ 
so the principal part of $N(V_2,\phi_3)$ is equal to $N(V_2,\phi_3)$ which has only one segment, from $(0,0)$ to $(2,\frac{91}{2}\sqrt{37}-2)$. The slope is $\frac{91}{4}\sqrt{37}-1$, giving the 3-rd stage approximant to $f$ over $V_0$, $V_3=[V_2;V_3(\phi_3)=\frac{91}{4}\sqrt{37}-1]$, with 
$$
G_{V_3}=G_{V_1}+\left(\frac{91}{4}\sqrt{37}-1\right)\ZZ=\ZZ+\frac{\sqrt{37}}{4}\ZZ.
$$
Now $f\sim \phi_3^2+x_1^{-3}x_2^{45}z$ in $V_3$ and $V_3(\phi_3)\not\in G_{V_1}$, so
$\phi_4=\phi_3^2+x_1^{-3}x_2^{45}z$ is a key polynomial over $V_3$. We have that 
$$
f=\phi_4+x_1^{-2}x_2^{30}\phi_3+x_1^{317}z
$$
so the principal part of $N(V_3,\phi_4)$ is $N(V_3,\phi_4)$, which has only one segment, from $(0,0)$ to $(1,-3+\frac{131}{4}\sqrt{37})$. The slope is $-3+\frac{131}{4}\sqrt{37}$, giving the 4-th stage approximant to $f$ over $V_0$, $V_4=[V_3;V_4(\phi_4)=-3+\frac{131}{4}\sqrt{37}]$. We have that $G_{V_4}=G_{V_3}$. Now
$f\sim \phi_4+x_1^{-2}x_2^{30}\phi_3$ in $V_4$ so
$$
\phi_5=\phi_4+x_1^{-2}x_2^{30}\phi_3
$$
 is a key polynomial over $V_4$. We have that $f=\phi_5+x_1^{317}z$ so the principal part of $N(V_4,\phi_5)$ is $N(V_4,\phi_5)$, which has only one segment, from $(0,0)$ to $(1,318+\frac{1}{2}\sqrt{37})$. The slope is $318+\frac{1}{2}\sqrt{37}$, giving the 5-th stage approximant to $f$ over $V_4$, $V_5=[V_4;V_5(\phi_5)=318+\frac{1}{2}\sqrt{317}]$. We have that $G_{V_5}=G_{V_3}$.

Now $f=\phi_5+x_1^{317}z$ is a key polynomial for $V_5$, so $V_6=[V_5;V_6(f(z))=\infty]$ is a pseudo valuation with $I(V_6)_{\infty}=(f(z))$. 

Let $\omega$ be the induced extension of $V_0$ to $K[z]/(f(z))$. We have that $G_{\omega}=G_{V_3}$ and thus
$$
[G_{\omega}:G_{V_0}]=4=\deg_zf=[L:K]
$$
showing that $\omega$ is the unique extension of $V_0$ to a valuation of $L$, and that $\delta(\omega/V_0)=1$, so that the extension is defectless (Section \ref{Sec6}).
Observe that we cannot avoid substitutions like (\ref{eqE1}), leaving the ring $A$ in any realization of the algorithm. Notice that the conclusions of Theorem \ref{Theorem1} are verified, if we take $A_1$ to be a birational extension of $A$ containing $x_1^{-1}x_2^{15}$.

\begin{Remark} In the example, the valuation $V_0$ is an Abhyankar valuation, which means that there is equality in the fundamental inequality of Abhyankar (\cite[Theorem 1]{Ab1}), 
$$
\dim_{\QQ}G_{V_0}\otimes_{\ZZ}\QQ+{\rm trdeg}_{A/m_A}R_{V_0}/m_{V_0}=\dim A.
$$
It is known  (\cite[Theorem 1]{K2})    that  Abhyankar valuations have ``no defect'', a fact which plays a role in this example. We will come back to the study of the effect of defect  in Sections \ref{Sec5}, \ref{SecV}, \ref{Sec6}  and \ref{Sec8} below.
\end{Remark}

\section{Henselization and completion}\label{Sec5}

A valued field $(K,\nu)$ is Henselian if for all algebraic extensions $L$ of $K$, there exists a unique valuation $\omega$ of $L$ which extends $\nu$. Some references on the theory of Henselian fields are  \cite{K1}, \cite{E}, \cite{R} and \cite{V2}.

An extension $(K^h,\nu^h)$ of a valued field $(K,\nu)$ is called a Henselization of $(K,\nu)$ if $(K^h,\nu^h)$ is Henselian and for all Henselian valued fields $(L,\omega)$ and all embeddings $\lambda:(K,\nu)\rightarrow (L,\omega)$, there exists a unique embedding $\tilde\lambda:(K^h,\nu^h)\rightarrow (L,\omega)$ which extends $\lambda$.

A Henselization $(K^h,\nu^h)$ of $(K,\nu)$ can be constructed by choosing an extension $\nu^s$ of $\nu$ to a separable closure $K^{\rm sep}$ of $K$ and letting $K^h$ be the fixed field of the decomposition group  
$$
\{\sigma\in G(K^{\rm sep}/K)\mid \nu^s\circ\sigma=\nu^s\}
$$
of $\nu^s$, and defining $\nu^h$ to be the restriction of $\nu^s$ to $K^h$ (\cite[Theorem 17.11]{E}).

\begin{Lemma}\label{Lemma2} Suppose that $(K,\nu)$ is a valued field  and 
let $(K^h,\nu^h)$ be a Henselization of $(K,\nu)$. 
Suppose that $f(z)\in K[z]$ is unitary, irreducible and separable. Then $f(z)$ is reduced in $K^h[z]$. Let $f(z)=f_1(z)f_2(z)\cdots f_r(z)$ be the factorization of $f(z)$ into irreducible unitary factors in  $K^h[z]$. If the coefficients of $f(z)$ are in $R_{\nu}$ then the coefficients of the $f_i(z)$ are in $R_{\nu^h}$. 

Let $\nu_i^h$ be the (unique) extension of $\nu^h$ to $K^h[z]/(f_i)$. Then the distinct extensions of $\nu$ to $K[z]/(f(z))$ are the $r$ restrictions $\nu_i$ of $\nu^h_i$ to $K[z]/(f(z))$, under the natural inclusions $K[z]/(f(z))\rightarrow K^h[z]/(f_i(z))$. 
\end{Lemma}

\begin{proof} The polynomial $f(z)$ is reduced in $K^h[z]$ since the separable polynomial $f(z)$ is reduced in $K^{\rm sep}[z]$ where  $K^{\rm sep}$ is a separable closure  of $K$.

Let $\overline z$ be a root of $f_i(z)$ in $K^{\rm sep}$. Then $f(z)$ is the minimal polynomial of $\overline z$ in $K[z]$, and $K[z]/(f(z))\cong K[\overline z]$.  If $\overline z$ is integral over $R_{\nu}$, then $\overline z$ is integral over $R_{\nu^h}$. Thus the coefficients of $f_i$ are in $R_{\nu^h}$ since $R_{\nu^h}$ is normal (\cite[Theorem 5, page 260]{ZS1}).  

If $L$ is a finite separable extension of $K$, then we have two associated sets,
$$
{\rm Mon}(L,K)=\mbox{$K$-embeddings of $L$ in $K^{\rm sep}$}
$$
and
$$
\mathcal E(L,\nu)=\mbox{Extensions of $\nu$ to a valuation of $L$}.
$$
By \cite[Lemma 1.4 ]{V2} or \cite[Section 17]{E}, the map $\Phi:{\rm Mon}(L,K)\rightarrow \mathcal E(L,\nu)$, defined by $\Phi(\lambda)=\nu^s\circ\lambda$ is surjective, with $\Phi(\lambda)=\Phi(\lambda')$ if and only if $\lambda\sim_{K^h}\lambda'$. The equivalence $\sim_{K^h}$ is defined by 
$\lambda\sim_{K^h}\lambda'$ if and only if there exists a $K^h$-isomorphism $\sigma:K^{\rm sep}\rightarrow K^{\rm sep}$ such that $\lambda'=\sigma\circ\lambda$. 

The valuation $\nu^s\circ\lambda$ is obtained from the embedding
$$
L\cong \lambda(L)\rightarrow \lambda(L)\cdot K^h
$$
into the join of $\lambda(L)$ and $K^h$ in $K^{\rm sep}$, and the restriction of the valuation $\nu^s|\lambda(L)\cdot K^h$ to $L$.

Let $L=K[z]/(f(z))$. The elements $\lambda\in {\rm Mon}(L,K)$ are in one to one correspondence with the distinct roots $\alpha_{\lambda}$ of $f(z)$ in $K^{\rm sep}$. We have $\lambda(L)\cdot K^h=K^h[\alpha_{\lambda}]$. Thus $\lambda(L)\cdot K^h\cong K^h[z]/(f_i)$ for some $i$. Further, $\lambda\sim_{K^h}\lambda'$ if and only if $\alpha_{\lambda}$ and $\alpha_{\lambda'}$ have the same minimal polynomial $f_i$ in $K^h[z]$.

Since $K^h$ is Henselian, for each $i$ there is a unique extension of $\nu^h$ to $K^h[z]/(f_i)$, and so the last assertion of the lemma follows.
\end{proof}

Suppose that $A$ is a local ring and $g(z)\in A[z]$ is a polynomial. Let $\overline g(z)\in A/m_A[z]$ be the polynomial obtained by reducing the coefficients of $g(z)$ mod $m_A$. 

A local ring $A$ is a Henselian local ring if it has the following property: Let $f(z)\in A[z]$ be a unitary polynomial of degree $n$. If $\alpha(z)$ and $\alpha'(z)$ are relatively prime unitary polynomials in $A/m_A[z]$ of degrees $r$ and $n-r$ respectively such that $\overline f(z)=\alpha(z)\alpha'(z)$, then there exist unitary polynomials $g(z)$ and $g'(z)$ in $A[z]$ of degrees $r$ and $n-r$ respectively such that $\overline g(z)=\alpha(z)$, $\overline g'(z)=\alpha'(z)$ and $f(z)=g(z)g'(z)$.

If $A$ is a local ring, a local ring $A^h$ which dominates $A$ is called a Henselization of $A$ if any local homomorphism from $A$ to a Henselian local ring can be uniquely extended to $A^h$. A Henselization always exists (\cite[Theorem 43.5]{N}). The construction is particularly nice when $A$ is a normal local ring. Let $K$ be the quotient field of $A$ and Let $K^{\rm sep}$ be a separable closure of $A$. Let $\overline A$ be the integral closure of $A$ in $K^{\rm sep}$ and let $\overline m$ be a maximal ideal of $\overline A$.

Let $H$ be the decomposition group
$$
H=G^s(\overline A_{\overline m}/A)=\{\sigma\in G(K^{\rm sep}/K)\mid \sigma(\overline A_{\overline m})=\overline A_{\overline m}\}.
$$
Then $A^h=(\overline A_{\overline m})^H$ is the fixed ring of the action of $H$ on $\overline A_{\overline m}$.
We have
$$
A^h=(\overline A\cap K^H)_{\overline m\cap(\overline A\cap K^H)}=\overline A_{\overline m}\cap K^H=(\tilde A)_{\overline m\cap\tilde A}
$$
where $\tilde A$ is the integral closure of $A$ in $K^H$.

Nagata rings are defined and their basic properties are developed in \cite[Chapter 12]{Mat}. Nagata rings are called Universally Japanese in \cite{EGA}. Their basic properties are established in \cite[IV.7.2.2]{EGA}.

We remark that if $A$ is a Nagata local domain with quotient field $K$ and $\nu$ is a valuation of $K$ which dominates $A$, then there exists a directed system of normal birational extensions $A_i$ of $A$ such that $\bigcup_i A_i=R_{\nu}$.

\begin{Lemma}\label{Lemma3} Continuing the assumptions of Lemma \ref{Lemma2},
suppose that $A$ is a Nagata local domain with quotient field $K$ such that $\nu$ dominates $A$, and that $A_i$ is a directed system of birational extensions of $A$ such that the $A_i$ are normal local domains which are dominated by $\nu$ and $\bigcup_i A_i=R_{\nu}$. Then there are natural equalities
$$
R_{\nu^h}=(R_{\nu})^h=\bigcup_i A_i^h.
$$
\end{Lemma}

\begin{proof} Let $\nu^s$ be an extension of $\nu$ to $K^{\rm sep}$ and 
$$
H=\{\sigma\in {\rm Gal}(K^{\rm sep}/K)\mid \nu^s\circ\sigma=\nu^s\},
$$
so that $K^h=(K^{\rm sep})^H$. Let $\overline V$ be the integral closure of $R_{\nu}$ in $K^{\rm sep}$, and let $m=\overline V\cap m_{\nu^s}$, a maximal ideal in $\overline V$. Since $K^{\rm sep}$  is algebraic over $K$, we have that $R_{\nu^s}=\overline V_m$ by \cite[Theorem 12, page 27]{ZS2}. Now, as is shown on the bottom of page 68 of \cite{ZS2}, $H$ is the decomposition group
$$
H=G^s(R_{\nu^s}/R_{\nu})=\{\sigma\in G(K^{\rm sep}/K)\mid \sigma(R_{\nu^s})=R_{\nu^s}\},
$$
so that 
$$
(R_{\nu})^h=\overline V_{m}\cap K^h=R_{\nu^s}\cap K^h=R_{\nu^h},
$$
establishing the first assertion of the lemma. 

Suppose that $A$ is a normal local ring with quotient field  $K$. Let $\tilde A$ be the integral closure of $A$ in $K^h$. if $A$ is dominated by $V=R_{\nu}$, then $\tilde A_{m_{\nu^s}\cap \tilde A}$ is dominated by $\tilde V_{m_{\nu^s}\cap \tilde V}$ (where $\tilde V$ is the integral closure of $V$ in $K^h$). Suppose $g,h\in \tilde V$ with $h\not\in m_{\nu^s}\cap \tilde V$. Since $\tilde A_i$ is a directed system, there exists $i$ such that $g,h\in \tilde A_i$, so $h\not\in m_{\nu^s}\cap \tilde A_i$ and $\frac{g}{h}\in (\tilde A_i)_{m_{\nu^s}\cap \tilde A_i}$. Thus
$$
\bigcup_i(\tilde A_i)_{m_{\nu^s}\cap \tilde A_i}=R_{\nu}^h.
$$

Let $\overline A_i$ be the integral closure of $A_i$ in $K^{\rm sep}$. By \cite[Lemma 3.3]{C1}, we have inclusions of decomposition groups 
$$
G^s(R_{\nu^s}/R_{\nu})\subset G^s((\overline A_i)_{m_{\nu^s}\cap \overline A_i}/A_i)
$$
for all $i$, and by \cite[Lemma 3.4]{C1}, there exists $i_0$ such that 
$$
G^s(R_{\nu^s}/R_{\nu})=G^s((\overline A_i)_{m_{\nu^s}\cap \overline A_i}/A_i)
$$
for $i\ge i_0$. Thus $A_i^h\subset (\tilde A_i)_{m_{\nu^s}\cap \tilde A_i}$ for all $i$ and $A_i^h=(\tilde A_i)_{m_{\nu_s}\cap \tilde A_i}$ for $i\gg 0$. The last assertion of the lemma now follows.
\end{proof}

Let $(K,\nu)$ be a valued field such that $\nu$ has rank 1. The completion $(\hat K,\hat\nu)$ (when $\nu$ has rank 1) is defined in Section 2 of \cite{E}. The completion $\hat K$ is defined to be the ring of $\nu$-Cauchy sequences in $K$ modulo the maximal ideal of $\nu$-null sequences ($\nu$-Cauchy sequences whose limit is $\infty$). The extension $\hat\nu$ of $\nu$ is defined by $\hat\nu(h)=\lim_{i\rightarrow \infty}\nu(h_i)$ if $(h_i)$ is a $\nu$-Cauchy sequence in $K$ which converges to $h$.  We have that $\hat K$ is a Henselian field (\cite[Lemma 16.7]{E}).
The following lemma is proven in \cite[Theorem 2.12 ]{E}.

\begin{Lemma}\label{LemmaB3} Suppose that $(K,\nu)$ is a rank 1 valued field  and 
 $(\hat K,\hat \nu)$ is a completion of $(K,\nu)$. 
Suppose that $f(z)\in K[z]$ is unitary, irreducible and separable, so that $f(z)$ is reduced in $\hat K[z]$. Let $f(z)=f_1(z)f_2(z)\cdots f_r(z)$ be the factorization of $f$ into irreducible unitary factors in $\hat K[z]$. 

Let $\hat\nu_i$ be the (unique) extension of $\hat\nu$ to $\hat K[z]/(f_i)$. Then the distinct extensions of $\nu$ to $K[z]/(f(z))$ are the $r$ restrictions $\nu_i$ of $\hat\nu_i$ to $K[z]/(f(z))$, under the natural inclusions $K[z]/(f(z))\rightarrow \hat K[z]/(f_i)$. 
\end{Lemma}

\begin{Lemma}\label{LemmaB6} Let notation be as in the statement of Lemma \ref{LemmaB3}. We then have a factorization $K\rightarrow K^h\rightarrow \hat K$ of valued fields.  Further, the factorizations of $f(z)$ into products of unitary irreducible polynomials in $K^h[z]$ and $\hat K[z]$ are the same.
\end{Lemma}

\begin{proof} We have a natural inclusion of $K^h$ into $\hat K$ since $\hat K$ is a Henselian field. The irreducible factors of $f(z)$ in $K^h(z)$ remain irreducible in $\hat K[z]$ since there is  a 1-1 correspondence of the irreducible factors of $f(z)$ in $\hat K[z]$ with the distinct extensions of $\nu$ to $L=K[z]/(f(z))$ by Lemma \ref{LemmaB3} and there is a 1-1 correspondence of the irreducible factors of $f(z)$ in $K^h[z]$ with the distinct extensions of $\nu$ to $L$ by Lemma \ref{Lemma2}.
\end{proof}

Some references on the defect of a finite field extension  are  \cite{K1}, \cite{E}, \cite{R} and \cite{V2}.

Suppose $(K,\nu)\rightarrow (L,\omega)$ is a finite separable extension of valued fields. Let $K^{\rm sep}$ be a separable closure of $K$ with an embedding of $L$ in $K^{\rm sep}$. Let $\nu^s$ be an extension of $\omega$ to a valuation of $K^{\rm sep}$.  As discussed above, we can use $\nu^s$ to define the Henselization $K^h$ of $(K,\nu)$, with valuation $\nu^h=\nu^s|K^h$, and then $L^h=L\cdot K^h$, the join of $L$ and $K^h$ in $K^{\rm sep}$,  is a Henselization of $(L,\omega)$ with valuation $\omega^h=\nu^s|L^h$ (\cite[Lemma 1.3]{V2}, \cite{K1}, \cite[(17.16)]{E}). The defect of $\omega$ over $\nu$ is defined as
\begin{equation}\label{eqefd}
\delta(\omega/\nu)=[L^h:K^h]/e(\omega^h/\nu^h)f(\omega^h/\nu^h)=[L^h:K^h]/e(\omega/\nu)f(\omega/\nu).
\end{equation}
The defect is a power of the residue characteristic $p$ of the valuation ring of $\nu$ by Ostrowski's lemma (\cite[Theorem 8.2]{K1}).

\section{Vaqui\'e's Algorithm}\label{SecV}

Suppose that $K$ is a field, $f(z)\in K[z]$ is unitary and irreducible,  $\nu$ is a valuation of $K$ and $\mu$ is a pseudo valuation of $K[z]$ which extends $\nu$  such that $I(\mu)_{\infty}=(f(z))$. Vaqui\'e shows in \cite[Theorem 2.5]{V1} that there exists a ``finite admissible family of valuations'' $\mathcal S$ which determines $\mu$. We will take the last element of $\mathcal S$ to be  the pseudo valuation $\mu$.  This result follows from \cite[Proposition 2.3]{V1}, which gives an algorithm for constructing such a  family.

We summarize the definition of an ``admissible family of valuations'' approximating $\mu$ (from \cite[Section 2.1]{V1}), which takes the following form since $I(\mu)_{\infty}=(f(z))\ne 0$.
A family $\mathcal S$ of iterated augmented valuations is called a ``simple admissible family'' if it is of the form $\mathcal S=(\mu_i)_{i\in I}$ where the set of indices $I$ is the disjoint union $I=B\coprod A$ with $B$ a finite set and $A$  a totally ordered set, where all elements of $A$ are larger than all elements of $B$ and $A$ does not have a largest element.

A family of valuations $\mathcal A=(\mu_i)_{i\in I}$ is called an ``admissible family'' for $\mu$ (defined on page 3473 of \cite{V1}) if it is a finite or countable union of simple admissible families $\mathcal S^{(t)}=(\mu_i^{(t)})_{i\in I^{(t)}}$. The first valuation of $\mathcal S^{(1)}$ is an inductive valuation of the form $\mu_1^{(1)}=[\mu_0;\mu_1^{(1)}(\phi_1^{(1)})=\gamma_1]$ where $\mu_0=\nu$ is the given valuation of $K$ and $\phi_1^{(1)}$ is a polynomial of degree 1. For $t\ge 2$, the first valuation $\mu_1^{(t)}$ of $\mathcal S^{(t)}$ is a ``limit augmented valuation'' for the family $(\mu_{\alpha^{(t-1)}})_{\alpha^{(t-1)}\in A^{(t-1)}}$. The construction of limit augmented valuations will be explained below.

Write $I^{(t)}=B^{(t)}\coprod A^{(t)}$ as above and write $B^{(t)}=\{1,\ldots,n^{(t)}\}$. Then for $i\ge 2$ in $B^{(t)}$, $\mu_i^{(t)}=[\mu_{i-1}^{(t)};\mu_i^{(t)}(\phi_i^{(t)})=\gamma_i^{(t)}]$ is an inductive valuation (Section \ref{Sec2}). For $\alpha\in A^{(t)}$, we have that $\mu_{\alpha}^{(t)}=[\mu_{n^{(t)}};\mu_{\alpha}^{(t)}(\phi_{\alpha}^{(t)}=\gamma_{\alpha}^{(t)}]$ is an inductive valuation, where $\deg_z\phi_{\alpha}^{(t)}=\deg_z\phi_{n^{(t)}}^{(t)}$.

Vaqui\'e requires that $\deg_z\phi_{i-1}^{(t)}<\deg_z\phi_i^{(t)}$ for $i\ge 2$ in $B^{(t)}$ but we do not assume this. By the definition of an inductive value, we do have that $\deg_z\phi_{i-1}^{(t)}\le \deg_z\phi_i^{(t)}$ for $i\ge 2$ in $B^{(t)}$. By the  construction of limit key polynomials, we have that $\deg_z\phi_{n^{(t)}}^{(t)}<\deg_z\phi_1^{(t+1)}$ for all $t$.

We require that for $g\in K[z]$ and $i<j\in I$,  
\begin{equation}\label{eqV2}
\mu_i(g)\le \mu_j(g)\le \mu(g).
\end{equation}
 Further, $\mu_i(\phi_i)=\mu(\phi_i)$ for all $i$. 
 
 We now discuss the construction of limit augmented valuations.

Suppose that $\mathcal A=(\mu_{\alpha})_{\alpha\in A}$ is an admissible family of valuations for $\mu$. Define 
(\cite[page 3473]{V1})
$$
\tilde\Sigma(\mathcal A)=\{g\in K[z]\mid \mu_{\alpha}(g)<\mu(g)\mbox{ for all }\mu_{i}\in\mathcal A\}.
$$
Define $d(\mathcal A)=\infty$ if $\tilde\Sigma(\mathcal A)=\emptyset$ and 
$$
d(\mathcal A)=\inf\{\deg_z\phi\mid \phi\in \tilde \Sigma(\mathcal A)\}
$$
if $\tilde\Sigma(\mathcal A)\ne\emptyset$.
Now define
$$
\Sigma(\mathcal A)=\{\phi\in \tilde\Sigma(\mathcal A)\mbox{ such that $\phi$ is unitary and $\deg_z\phi=d(\mathcal A)$}\}
$$
and
\begin{equation}\label{eqV1}
\Lambda(\mathcal A)=\{\mu(\phi)\mid \phi\in \Sigma(\mathcal A)\}.
\end{equation}
Suppose that   $\Lambda(\mathcal A)$ does not have a largest element.  We then define a totally ordered index set $C$, which does not have a largest element, so that
$$
\Lambda(\mathcal A)=\{\gamma_{\alpha}\mid \alpha\in C\},
$$
where $\alpha<\beta$ if and only if $\gamma_{\alpha}<\gamma_{\beta}$.

A ``limit key polynomial'' $\phi$ for $\mathcal A$ is defined on page 3465 of \cite{V2}.  It satisfies the three properties that $\phi$ is $\mathcal A$-minimal, $\phi$ is $\mathcal A$-irreducible and $\phi$ is unitary. The elements of $\Sigma(\mathcal A)$ are limit key polynomials for $\mathcal A$ by \cite[Proposition 1.21]{V1}. 
Choose $\phi_{\alpha}\in \Sigma(\mathcal A)$ for each $\alpha\in C$ so that $\mu(\phi_{\alpha})=\gamma_{\alpha}$. We then have a limit augmented valuation $\mu_{\alpha}=[\mathcal A;\mu_{\alpha}(\phi_{\alpha})=\gamma_{\alpha}]$ (\cite[Proposition 1.22]{V1}),  which is defined by 
\begin{equation}\label{eqV4}
\mu_{\alpha}(g)=\max_{j\in A}\{\min_i\{\mu_j(g_{i})+i\mu(\phi_{\alpha})\}\}
\end{equation}
for $g\in K[z]$, where 
$$
g=\sum g_{i}\phi_{\alpha}^i
$$
with $\deg_z g_{i}<\deg_z\phi_{\alpha}$.

The ``associated family of iterated augmented valuations''  to $\mathcal A$ is 
\begin{equation}\label{eqV3}
(\mu_{\alpha})_{\alpha\in C}.
\end{equation}

 We will explain here how the algorithm proceeds if we are given a  discrete simple admissible family $\mathcal S=\{\mu_1,\ldots,\mu_n\}$ such that $\Sigma(\mu_n)$ is nonempty.
We will produce an admissible family of valuations $\mathcal B$ such that $d(\mathcal B)>d(\mu_n)$.

All elements of $\Sigma(\mu_n)$ are key polynomials for $\mu_n$ by \cite[Theorem 8.1]{M1} or \cite[Theorem 1.15 page 3453]{V2}. 

First suppose that the set of values $\Lambda(\mu_n)$ has a largest element $\gamma'$ (which could be $\infty$). 
Then we can define
$\mu'=[\mu_n;\nu'(\phi')=\gamma']$ 
where $\phi'\in \Sigma(\mu_n)$ satisfies $\mu(\phi')=\gamma'$. 
We then have two cases, depending on if $\deg_z\phi'>\deg_z\phi_n$ or if $\deg_z\phi'=\deg_z\phi_n$.

Assume that $\deg_z\phi'>\deg_z\phi_n$. Set $\phi_{n+1}=\phi'$, $\gamma_{n+1}=\gamma'$ and 
$$
\mu_{n+1}=
\mu'=[\mu_n;\mu_{n+1}(\phi_{n+1})=\gamma_{n+1}].
$$
 Then define $\mathcal B=\{\mu_1,\ldots,\mu_n,\mu_{n+1}\}$ which is a discrete simple admissible family, with $d(\mathcal B)\ge \deg_z(\phi_{n+1})>\deg_z\phi_n$.

Now assume that $\deg_z\phi'=\deg_z\phi_n$. Then  define 
$\mathcal B=\{\mu_1,\ldots,\mu_{n},\mu'\}$ which is again a discrete simple admissible family with $d(\mathcal B)>d(\mu')$ (by \cite[Lemma 15.1]{M1} or \cite[Corollary, page 3448]{V1}).

The last case is when $\Lambda(\mu_n)$ does not have a largest element. Define the associated family of iterated augmented valuations $(\mu_{\alpha})_{\alpha\in C}$ of (\ref{eqV3}) for $\mu_n$. For all $\gamma_{\alpha}\in \Lambda(\mu_n)$, define  $\mu_{\alpha}=[\mu_n; \mu_{\alpha}(\phi_{\alpha})=\gamma_{\alpha}]$. Define $\mathcal S^{(1)}$ by adding to $\mathcal S$ the family
$\mathcal C=(\mu_{\alpha})_{\alpha\in C}$, so $\mathcal S^{(1)}$ is indexed by $I'=\{1,\ldots,n\}\coprod C$ (which does not have a largest element). We have that $\mathcal S^{(1)}$ is a simple admissible family. The family $\mathcal C$ is an ``exhaustive, continuous family of iterated augmented  valuations'' with the property that $\deg_z\phi_{\alpha}=d(\mu_n)$ for all $\alpha\in C$. We have that $f\not\in\Sigma(\mu_n)$ since $C$ does not have a largest element. Thus $\tilde\Sigma(\mathcal C)\ne\emptyset$. By \cite[Proposition 1.21]{V1}, all polynomials of $\Sigma(\mathcal C)$ are limit key polynomials for the family $\mathcal C$. We now choose a polynomial $\phi_1^{(2)}\in \Sigma(\mathcal C)$, 
and define the ``limit augmented valuation'' $\mu_1^{(2)}=[(\mu_{\alpha})_{\alpha\in C}; \nu_1^{(2)}(\phi_1^{(2)})=\mu(\phi_1^{(2)})]$ (by the definition on page 2465 of \cite{V1} and \cite[Proposition 1.22]{V1} and as explained in (\ref{eqV4})) and the discrete, simple admissible family $\mathcal S^{(2)}=\{\mu_1^{(2)}\}$. By \cite[Proposition 1.27]{V1}, $\deg_z\phi_1^{(2)}$ is greater than the degree of the polynomials in $\Sigma(\mu_n)$. Define the admissible family
$\mathcal B=\mathcal S^{(1)}\cup \mathcal S^{(2)}$, which is indexed by $I''=I'\coprod \{1^{(2)}\}$ (where $1^{(2)}$ is larger than every element of  $I'$).

\subsection{Comparison of the algorithms of Section \ref{Sec3} and Vaqui\'e}\label{SubSecMV}

Suppose that $W$ is a pseudo valuation of $K[z]$ which extends a valuation $V_0$ of $K$, such that $I(W)_{\infty}=(f(z))$ where $f$ is unitary and $f(z)\in R_{V_0}[z]$.
Let 
\begin{equation}\label{ch2}
V_1,\ldots,V_k,\ldots
\end{equation}
 be a sequence of approximants to $f$ over $V_0$ constructed by the algorithm of Section \ref{Sec3} which satisfy (\ref{eqV2}) (with  $\mu_j=V_j$ and $\mu=W$). 
 
 %If we have that $\deg_z\phi_{i+1}=\deg_z\phi_i$ for some $i$, then we have that $\phi_{i+1}$ is a key polynomial for $V_{i-1}$ and 
%$$
%[V_{i-1};V_{i+1}(\phi_{i+1})=W(\phi_{i+1})]=[V_i;V_{i+1}(\phi_{i+1})=W(\phi_{i+1})]
%$$
 %by Lemma 15.1 \cite{M1} or Corollary, page 3448 \cite{V1}. So we may modify the sequence of approximants so that either 
  
 %$$
 %\deg_z\phi_1\le\deg_z\phi_2\le\cdots\le\deg_z\phi_k
 %$$

 We then either have that  $\phi_k=f$ or $V_0,V_1,\ldots,V_k,\ldots$ is infinite with $\deg_z\phi_k=\deg_z\phi_{k_0}$ for $k\ge k_0$ 
  In the first case, we have that $\mathcal S=\{V_1,\ldots,V_k\}$ is a discrete simple admissible family of valuations which determines $W$. 
 
 Suppose that $V_1,\ldots,V_k,\ldots$ is infinite. Then $\phi_k\in \Sigma(V_{k_0})$ for $k> k_0$, and so $d(V_{k_0})=\deg_z\phi_{k_0}$.

 If $\Lambda(V_{k_0})$ has a maximal element $\gamma$, $\phi'\in\Sigma(V_{k_0})$ is a key polynomial with $W(\phi')=\gamma$ and corresponding valuation $\mu'=[V_{k_0};\mu'(\phi')=W(\phi')]$, then $\{V_1,\ldots,V_{k_0},\mu'\}$ is the first part of the discrete part of $S^{(1)}$ constructed by Vaqui\'e's algorithm. If $W(\phi')=\infty$, then $\mathcal S=\mathcal S^{(1)}=\{V_1,\ldots,V_{k_0},\mu'=W\}$ is an admissible family of valuations which determines $W$.
 
 Suppose that $\Lambda(V_{k_0})$ does not have a largest element. 
 Let $\mathcal C=(\mu_{\alpha})_{\alpha\in C}$ be the associated family of iterated augmented valuations associated to $V_{k_0}$ of (\ref{eqV3}).  Choose a limit key polynomial $\phi_1^{(2)}$ for $\mathcal C$. The next step in Vaqui\'e's algorithm is to construct $\mathcal S=\mathcal S^{(1)}\cup \mathcal S^{(2)}$ where $\mathcal S^{(1)}=\{V_1,\ldots,V_{k_0}\}\cup\mathcal C$ and $S^{(2)}=\{V_1^{(2)}=[\mathcal C;V_1^{(2)}(\phi_1^{(2)})=W(\phi_1^{(2)})]\}$.
 
 Looking again at the case where $\Lambda(V_{k_0})$ has a maximal element $\gamma$ and $\phi'\in\Sigma(V_{k_0})$ is the corresponding key polynomial, we have an expression $\phi'=\phi_{k_0}+h$ where $h\in K[z]$ has $\deg_zh<\deg_z\phi_{k_0}$. We further have that $h\in R_{V_0}[z]$ by Remark \ref{RemarkM17}. We have an expression (for some $r$)
 $$
 h=\sum_{j=1}^ra_j\phi_1^{\sigma_1(j)}\cdots\phi_{k_0-1}^{\sigma_{k_0-1}(j)}
 $$
 with $a_j\in R_{V_0}$, $0\le \sigma_i(j)<n_i=[G_{V_{i}}:G_{V_{i-1}}]$ for all $i$ and $j$ and
 $$
 W(a_i\phi_1^{\sigma_1(i)}\cdots\phi_{k_0-1}^{\sigma_{k_0-1}(i)})<W(a_j\phi_1^{\sigma_1(j)}\cdots\phi_{k_0-1}^{\sigma_{k_0-1}(j)})
 $$
 if $i<j$. Let
 \begin{equation}\label{eqMV1}
 \psi_i=\phi_{k_0}+a_1\phi_1^{\sigma_1(1)}\cdots \phi_{k_0-1}^{\sigma_{k_0-1}(1)}
 +\cdots+
 a_i\phi_1^{\sigma_1(i)}\cdots\phi_{k_0-1}^{\sigma_{k_0-1}(i)}
 \end{equation}
 for $1\le i\le r$.  We then have (for instance by the criterion of \cite[Proposition 1.9]{V1}) that 
 \begin{equation}\label{eqMV2}
 V_1,\ldots,V_{k_0},V_{k_0+1}',\ldots,V_{k_0+r}'
 \end{equation}
 is a $(k_0+r)$-th stage approximant to $f$ over $V_0$, where 
 $$
 V_{k_0+1}'=[V_{k_0};V_{k_0+1}'(\psi_1)=W(\psi_1)]\mbox{ and }
 V_{k_0+i}'=[V_{k_0+i-1}'; V_{k_0+i}'(\psi_i)=W(\psi_i)]\mbox{ for $2\le i\le r$.}
 $$
 Further, either $W(\phi')<\infty$ and 
 $$
 d(\{V_1,\ldots,V_{k_0},V_{k_0+1}',\ldots,V_{k_0+r}'\})>\deg_z\phi_{k_0} ,
 $$
 or $W(\phi')=\infty$, in which case $f=\phi'$ (since $f$ and $\phi'$ are unitary in $z$ of the same degree) and   $\psi_r=f$.
 
 We may now continue the algorithm of Section \ref{Sec3} to construct higher stage approximants, starting from $V_{k_0+r}'$. After a finite number of iterations of this procedure, we construct a sequence of approximants to $f$,
 \begin{equation}\label{ch1}
 V_1,\ldots,V_{k_1},\ldots
 \end{equation}
 so that $\deg_z\phi_i\le \deg_z\phi_{i+1}$ if $i<k_1$ and $\deg_z\phi_i=\deg_z\phi_i$ for $i\ge k_1$.
  which is either of finite length $k_1$, so that $V_{k_1}=W$,  or there is a jump ($t>1$) in the construction of the admissible family $\mathcal S=\mathcal S^{(1)}\cup\cdots\cup \mathcal S^{(t)}$ determining $W$.

 Suppose that (\ref{ch1}) is infinite and the  equivalent conditions of Lemma \ref{Lemma4} hold for (\ref{ch1}).  Let $\mathcal C=(\mu_{\alpha})_{\alpha\in C}$ be the associated family to $V_{k_1}$ of (\ref{eqV3}). Suppose $g\in K[z]$ and $W(g)<\infty$ and $k$ is so large that $W(\phi_k)>W(g)$.  Write $g=g_m\phi_k^m+\cdots+g_0$ with $\deg_zg_i<\deg_z\phi_k$ for all $i$. We have that 
 $$
 V_k(g)=V_k(g_0)=V_{k_0-1}(g_0)=W(g_0)=W(g).
 $$
 Thus $g\not\in \tilde \Sigma(\mathcal C)$ and so $\deg_zf$ is the smallest degree of an element of $\tilde\Sigma(\mathcal C)$. Thus $\mathcal S=\mathcal S^{(1)}\cup \mathcal S^{(2)}$ where $\mathcal S^{(1)}=\{V_1,\ldots,V_{k_1}\}\cup\mathcal C$ and $S^{(2)}=\{V_1^{(2)}\}$ where $V_1^{(2)}=[\mathcal C;V_1^{(2)}(f(z))=\infty]$.

 The following proposition follows from our analysis.
\begin{Proposition}\label{PropVO} Suppose that $V_0$ has finite rank. Then there exists a realization of the algorithm of Section \ref{Sec3} which produces the first simple admissible family $\mathcal S^{(1)}$ of an admissible family $\mathcal S=\mathcal S^{(1)}\cup\cdots\cup \mathcal S^{(t)}$ determining $W$, where all key polynomials are in $R_{\nu}[z]$.
\end{Proposition}

\subsection{Invariants of ramification and jumps}\label{Subsec7.2}

Suppose that $W$ is an extension of a valuation $V=V_0$ of $K$ to a pseudo valuation of $K[z]$ with $I(W)_{\infty}=(f(z))$ in $K[z]$ with $f$ unitary. Let $\omega$ be  the induced valuation on $L=K[z]/(f(z))$.

The jumps $s^{(j-1)}(\mathcal S)$ in a family $\mathcal S=\mathcal S^{(1)}\cup \cdots\cup \mathcal S^{(t)}$ realizing $W$ are defined by the equations
\begin{equation}\label{eqMV4}
\deg_z\phi_1^{(j)}=s^{(j-1)}(\mathcal S)\deg_z\phi_{\alpha}^{(j-1)}
\end{equation}
where $\phi_{\alpha}^{(j-1)}$ is a key polynomial of a member of  the continuous family $\mathcal C^{(j-1)}$ associated to $\mathcal S^{(j-1)}$.
The total jump of the family $\mathcal S$ is 
$$
s^{\rm tot}(\mathcal S)=\prod_{j=2}^ts^{(j-1)}(\mathcal S).
$$

We have by Lemma 2.11 and \cite[Corollary 2.10]{V2} that
\begin{equation}\label{eqMV3}
\deg_zf=[L:K]= e(\omega/V)f(\omega/V)s^{\rm tot}(\mathcal S).
\end{equation}
We have that $s^{\rm tot}(\mathcal S)=1$  if and only if there are no jumps in the construction of approximants. Here $e(\omega/V)=[G_{\omega}:G_V]$ where $G_{\omega}$ and $G_V$ are the respective value groups of $\omega$ and $V$, and $f(\omega/V)$ is the index of the respective residue fields of the valuation rings of $\omega$ and $V$.

In the case where $\omega$ is the unique extension of $V$ to a valuation of $L$,  we have by Ostrowski's lemma that 
\begin{equation}\label{eqMV5}
[L:K]=e(\omega/V)f(\omega/V)\delta(\omega/V)
\end{equation}
where the defect $\delta(\omega/V)$ is a power of the residue characteristic  $p$ of $V$. Comparing with (\ref{eqMV3}), we have that  $s^{\rm tot}(\mathcal S)=\delta(\omega/V)$ in this case. 
Thus (assuming $\omega$ is the unique extension of $V$) there is no jump if and only if there is no defect and in this case,
\begin{equation}\label{eqM40}
[L:K]=e(\omega/V)f(\omega/V).
\end{equation}

In constrast to the good property of key polynomials of (\ref{eqM16}), we have examples of the following type for limit key polynomials. 

\begin{Example}\label{Example10} The jumps $s^{(i)}(\mathcal S)$ and total jump $s^{\rm tot}(\mathcal S)$  can be rational numbers which are not integers.
\end{Example}
We now construct such an example. Let $\mathbf k$ be an algebraically closed field and $K=\mathbf k(x)$ be a rational function field in one variable over $\mathbf k$. Let $\nu$ be the valuation of $K$ with valuation ring $R_{\nu}=\mathbf k[x]_{(x)}$ and such that $\nu(x)=1$. Let $L=K[z]/(z^3-z^2-x)\cong \mathbf k(z)$. Let $\omega$ be the extension of $\nu$ to $L$ with valuation ring $R_{\omega}=\mathbf k[z]_{(z)}$ and $\omega(z)=\frac{1}{2}$. Then $e(\omega/\nu)=2$ and $f(\omega/\nu)=1$. Thus by (\ref{eqMV3}),
$$
s^{\rm tot}(\mathcal S)=\frac{\deg_z(z^3-z^2-x)}{e(\omega/\nu)f(\omega/\nu)}=\frac{3}{2}.
$$

\section{Defectless extensions}\label{Sec6}

\begin{Lemma}\label{LemmaH} Suppose that $(K,\nu)$ is a valued field containing an algebraically closed field $\mathbf k$ such that $R_{\nu}/m_{\nu}\cong \mathbf k$ and $f(z)\in R_{\nu}[z]$ is unitary, irreducible and separable. Let $L=K[z]/(f(z))$ and let $\omega$ be an extension of $\nu$ to $L$. Let $W$ be the induced pseudo valuation on $K[z]$. Let $\overline f(z)\in K^h[z]$ be the irreducible factor of $f(z)$ which induces $\omega$ (by Lemma \ref{Lemma2}) and let $\omega^h$ be the (unique) extension of $\nu^h$ to $K^h[z]/(\overline f(z))$. Let $\overline W$ be the induced pseudo valuation on $K^h[z]$. Let $V_0=\nu$ and $W_0=\nu^h$. Then the following hold:
\begin{enumerate}
\item[1)] ${\rm gr}_{\nu^h}(R_{\nu^h})={\rm gr}_{\nu}(R_{\nu})$.
\item[2)] Set $\phi_1=z$, let $V_1=[V_0;V_1(\phi_1)=W(\phi_1)]$ and let $W_1=[\nu^h;W_1(\phi_1)=W(\phi_1)]$. Then ${\rm gr}_{W_1}(R_{\nu^h}[z])={\rm gr}_{V_1}(R_{\nu}[z])$.
\item[3)] Suppose that $V_i=[V_{i-1};V_i(\phi_i)=W(\phi_i)]$ for $1\le i\le k$ is a realization of the algorithm of Section \ref{Sec3} in $R_{\nu}[z]$ such that $W_i=[W_{i-1};W_i(\phi_i)=W(\phi_i)]$ for $1\le i\le k$ is a realization of the the algorithm of Section \ref{Sec3} in $R_{\nu^h}[z]$ and
\begin{equation}\label{eqB11}
{\rm gr}_{W_i}(R_{\nu^h}[z])={\rm gr}_{V_i}(R_{\nu}[z])\mbox{ for }1\le i\le k.
\end{equation}
Suppose that $\phi_{k+1}\in R_{\nu}[z]$ has an expression $\phi_{k+1}=\phi_k^n-c_k\phi_1^{j_1}\cdots \phi_{k-1}^{j_{k-1}}$ of the form of (\ref{eq21}), $\phi_{k+1}$ is a key polynomial for $W_k$ and 
$W_{k+1}=[W_k;W_{k+1}(\phi_{k+1})=W(\phi_{k+1})]$ is a $(k+1)$-st approximant of $\overline f$ over $W_0$.
Then $\phi_{k+1}$ is a key polynomial for $V_k$ and $V_{k+1}=[V_k;V_{k+1}(\phi_{k+1})=W(\phi_{k+1})]$ 
is a $(k+1)$-st approximant of $f$ over $V_0$. Further, ${\rm gr}_{W_{k+1}}(R_{\nu^h}[z])={\rm gr}_{V_{k+1}}(R_{\nu}[z])$.
\end{enumerate}
\end{Lemma}

\begin{proof}
Statement 1) follows since $G_{\nu^h}=G_{\nu}$ and $R_{\nu^h}/m_{\nu^h}=R_{\nu}/m_{\nu}$ by \cite[Theorem 17.19]{E}. Statement 2) follows since 
$$
{\rm gr}_{W_1}(R_{\nu^h}[z])={\rm gr}_{\nu^h}(R_{\nu^h})[{\rm in}_{W_1}(\phi_1)]={\rm gr}_{\nu}(R_{\nu})[{\rm in}_{V_1}(\phi_1)]= {\rm gr}_{V_1}(R_{\nu}[z])   .
$$

Now we will prove statement 3).  To show that $\phi_{k+1}$ is a key polynomial over $V_k$, we must verify that 1) - 6) of the definition of a key polynomial, given after (\ref{eqM2}) hold for $\phi_{k+1}$ over $V_k$. This follows since these conditions hold for  $\phi_{k+1}$ over $W_k$. The fact that $W_{k+1}$ is a $(k+1)$-st  approximant to $\overline f$ over $W_k$ implies that $\phi_{k+1}$ equivalence divides $\overline f$ in $W_{k}$. Thus ${\rm in }_{W_{k}}(\phi_{k+1})$ divides ${\rm in}_{W_{k}}(\overline f)$ in ${\rm gr}_{W_k}(R_{\nu^h}[z])$. Now ${\rm in}_{W_{k}}(\overline f)$ divides ${\rm in}_{W_{k}}(f(z))$ in ${\rm gr}_{W_k}(R_{\nu^h}[z])$. So ${\rm in}_{W_{k}}(\phi_{k+1})$ divides ${\rm in}_{W_{k}}(f(z))$ in ${\rm gr}_{W_k}(R_{\nu^h}[z])={\rm gr}_{V_k}(R_{\nu}[z])$. Thus $\phi_{k+1}$ equivalence divides $f(z)$ in $V_k$ and so $V_{k+1}$ is a $(k+1)$-st approximant to $f(z)$ over $V_k$. 
We have that $n=[G_{W_k}:G_{W_{k-1}}]=[G_{V_k}:G_{V_{k-1}}]$ as $G_{W_{k-1}}=G_{V_{k-1}}$. Finally, we have that ${\rm gr}_{V_{k+1}}(R_{\nu}[z])={\rm gr}_{W_{k+1}}(R_{\nu^h}[z])$ by Lemma \ref{Lemma1}.
\end{proof}

%\begin{Proposition}\label{Prop7} Suppose that $V_0$ is a rank 1 valuation on a field $K$ and $W$ is an extension of $V_0$ to a pseudo valuation of $K[z]$ with $I(W)_{\infty}=(f(z))$, where $f$ is a monic,  irreducible and separable polynomial in $R_{V_0}[z]$ and $\omega$ is defectless over $V_0$ (that is, $\delta(\omega/V_0)=1$) where $\omega$ is the induced valuation on $K[z]/(f(z))$. Then $W$ is a finite stage approximant $V_k$ to $f$ over $V_0$ or is a limit valuation $V_{\infty}=\lim_{k\rightarrow \infty}V_k$ of an infinite sequence of approximants to $f$ over $V_0$.

%If $\omega$ is the unique extension of $V_0$ to $K[z]/(f(z))$ then $W$ is a finite stage approximant to $f$.

%\end{Proposition}

\begin{Theorem}\label{Theorem4*}   
Suppose that $A$ is a Nagata local domain which contains an algebraically closed field $\mathbf k$ such that $A/m_A\cong \mathbf k$. Let $K$ be the quotient field of $A$ and suppose that $V_0=\nu$ is a rank 1 valuation of $K$ which dominates $A$ and such that the residue field of the valuation ring of $V_0$ is $\mathbf k$. Suppose that 
$f(z)\in A[z]$ is unitary, irreducible and separable and $W$ is  a pseudo valuation  of $K[z]$ such that $I(W)_{\infty}=(f(z))$ in $K[z]$ which extends $V_0$. Let $\omega$ be the induced valuation of $L=K[z]/(f(z))$. Then $\omega$ is defectless over $\nu$ ($\delta(\omega/\nu)=1$) if and only if there exists a normal birational extension $A_1$ of $A$ which is dominated by $\nu$ such that there exists a realization
$$
V_1,\ldots,V_k,\ldots
$$
of the algorithm of Section \ref{Sec3} in $A_1[z]$, satisfying (\ref{eq1}) and (\ref{eq21}) for all $k$ with $c_k\in A_1$ for all $k\ge 1$, such that $W=V_k$ for some finite $k$ or $W=\lim_{k\rightarrow\infty}V_k$.

If these equivalent conditions hold, then 
there exists a positive integer $k$  such that
$$
{\rm gr}_{\omega}(A_1[z]/(f(z)))\cong {\rm gr}_{\nu}(A_1)[\overline \phi_1,\ldots,\overline \phi_k]/
I
$$
where
$$
I=(\overline \phi_1^{n_1}-\overline c_1,\overline \phi_2^{n_2}-\overline c_2\overline\phi_1^{j_1(2)},\ldots,\overline \phi_k^{n_k}-\overline c_k\overline\phi_1^{j_1(k)}\overline\phi_2^{j_2(k)}\cdots\overline \phi_{k-1}^{j_{k-1}(k)})
$$
is a finitely generated and presented ${\rm gr}_{\nu}(A_1)$-module.
\end{Theorem}

An example showing that the conclusions of Theorem \ref{Theorem4*} may not hold if $\nu$ has rank larger than one will be given in Section \ref{SecNEX}. In Example \ref{Ex201}, it will be shown that the conclusions of Theorem \ref{Theorem4*} may not hold if $f(z)$ is not separable over $K$.

\begin{proof} First suppose that  $\delta(\omega/\nu)=1$. 
Let notation be as in Section \ref{Sec5}. By Lemma \ref{Lemma2}, there exists an extension $\overline W$ of $\nu^h$ to a pseudo valuation of $K^h[z]$, such that $I(\overline W)_{\infty}=(\overline f)$ where $\overline f(z)$ is an irreducible factor of $f(z)$ in $K^h[z]$, and
 $\overline W$ is an extension of $W$. 
 
 We will construct a special sequence of approximants $W_1,\ldots,W_{k_0}$ 
       to $\overline f$ over $\nu^h$  such that $\overline W=W_{k_0}$.
       In particular,
 $$
 W_{k_0}=[W_{k_0-1};W_{k_0}(\phi_{k_0})=\infty]
 $$
  where $\phi_{k_0}=\overline f$.

 Set $\phi_1=z$ and let $W_1=[\nu^h;W_1(\phi_1)=W(\phi_1)]$.  
 Suppose by induction on $k$ that we have constructed a sequence of approximants to $\overline f$ over $\nu^h$, 
 $$
 W_1,\ldots,W_k
 $$
 giving a realization of the algorithm of Section \ref{Sec3}, such that expressions
 $$
 \phi_i=\phi_{i-1}^{n_{i-1}}-c_{i-1}\phi_1^{j_1(i-1)}\cdots \phi_{i-2}^{j_{i-2}(i-1)}
 $$
 of the form of  \ref{eq21}) hold for $i\le k$ with $c_i\in R_{\nu}$ for $i\le k-1$. After replacing $A$ with a birational extension  $A_1$ of $A$, we may suppose that $c_i\in A$ for $i\le k-1$.
 
 If $\Lambda(W_{k})$ does not have a largest element, then we have a jump $s^{(1)}>1$ by (\ref{eqMV4}) and the analysis of this case in Subsection \ref{SubSecMV}. But by (\ref{eqMV3}) and (\ref{eqMV5}), there cannot be a jump, and we have a contradiction, showing that $\Lambda(W_{k})$ has a largest element.

Suppose we are in the case where $\Lambda(W_{k})$ has a maximal element $\gamma\ne\infty$ and $\phi'\in \Sigma(W_{k})$ is a corresponding key polynomial.  
 We will modify the resulting sequence (\ref{eqMV2}) of the  analysis in Subsection \ref{SubSecMV}, which we will write as
\begin{equation}\label{eqND100}
W_1,\ldots,W_{k},W_{k+1}',\ldots,W_{k+r}'
\end{equation}
by modifying the $\psi_i$ of (\ref{eqMV1}), replacing the $a_i$ with suitable $b_i\in R_{\nu}$   for $1\le i\le r$. With the notation of Lemma \ref{Lemma3}, since $a_1,\ldots,a_r\in R_{\nu^h}$, there exists $A_l$ such that $a_i\in A_l^h$ for $1\le i\le r$ and $\phi_1,\ldots,\phi_{k_0}\in A_l[z]$. Thus, since $\overline W$ induces a rank 1 valuation on $K^h[z]/(\overline f(z))$, there exists $n\in \ZZ_+$ such that $nV_0(m_{A_l^h})>\overline W(\phi')=\gamma$.
Now $A_l\rightarrow A_l^h$ is unramified with no residue field extension, so there exists $b_i\in A_l$  such that $a_i-b_i\in m_{A_l^h}^n$ for $1\le i\le r$. Thus $V_0(b_i)-\nu^h(a_i)>\overline\omega(\phi')$ for $1\le i\le r$ and we can replace $\psi_i$ with $\psi_{i-1}+b_i\phi_1^{j_1(i)}\cdots\phi_{k-1}^{j_{k-1}(i)}$ in (\ref{eqMV1})  for $1\le i\le r$, to produce a sequence (\ref{eqND100}) with $\psi_i\in R_{\nu}[z]$ for all $i$.
We then have a corresponding sequence to (\ref{eqND100}),
$$
V_1,\ldots,V_{k},V_{k+1}',\ldots,V_{k+r}'
$$
of approximants to $f$ over $V_0$ by Lemma \ref{LemmaH}.

 Now we can continue, using  the algorithm of Section \ref{Sec3},
applying the above argument as necessary until we reach  $W_k$ such that  the maximal element of $\Lambda(W_{k})$ is $\infty$, so that $\overline f\in \Sigma(W_{k})$.

With this assumption, there exists $l$ (with the notation of Lemma \ref{Lemma3}) such that the coefficients of $\overline f$ are in $A_l^h$ and the coefficients of $\phi_1,\ldots,\phi_{k}$ are in $A_l$. We have $\overline f=\phi_{k}+h$ where $h\in A_l^h[z]$ and $\deg_zh<\deg_z\phi_{k}$. Set $\psi_0=\phi_{k}$. By induction, we may construct a sequence $\psi_i\in A_l[z]$  of monic poynomials with $\deg_z\psi_i=\deg_z\phi_{k}$, such that for all $i$,
$\overline f=\psi_i+h_i$ with $h_i\in (A_l)^h[z]$ a polynomial of degree $<\deg_z\phi_{k}$ and 
$$
\psi_{i+1}=\psi_i+b_i\phi_1^{\sigma_1(i)}\cdots\phi_{k_0-1}^{\sigma_{k_0-1}(i)}
$$
with $b_i\in A_l$ and $0\le \sigma_j(i)<n_j$ for $1\le j\le k_0-1$ such that $\overline W(\psi_{i+1})>\overline W(\psi_i)$ for all $i$.  Since $A_l$ is Noetherian, and $\overline W$ induces a  rank 1 valuation on $K^h[z]/(\overline f(z))$, we have that $\overline W$ takes on $A_l[z]$ only a finite number of values which are less than or equal to a given finite upper bound.  Thus
we either obtain that $\psi_i=\overline f(z)$ for some $i$, or that 
$$
\lim_{i\rightarrow\infty} W(\psi_i)=\lim_{i\rightarrow\infty}\overline W(\psi_i)=\infty.
$$
 
 By Lemma \ref{LemmaH}, inductively defining $V_i=[V_{i-1};V_i(\phi_i)=W(\phi_i)]$ for $1\le i\le k$ and $V_{i+k}=[V_{i+k-1};V_i(\psi_i)=W(\psi_i)]$ for $k<i$, we construct a sequence
 $$
 V_1,\ldots,V_k,\ldots
 $$
 of approximants to $f(z)$ over $V_0$ such that $\lim_{i\rightarrow \infty}V_i(\phi_i)=\infty$, so that $W=\lim_{i\rightarrow \infty}V_i$ by Lemma \ref{Lemma4}.

 Now suppose there exists a normal birational extension $A_1$ of $A$ and a realization $V_1,\ldots, V_k,\ldots$ of the algorithm of Section \ref{Sec3} as  in the statement of the theorem. We will show that the defect $\delta(\omega/\nu)=1$.
 
 First suppose that the sequence is of finite length, terminating with $V_k=W$, so that the last key polynomial is $\phi_k=f$ (with $V_k(\phi_k)=\infty$). We have that $\deg_z\phi_1=1$ and $\deg_z\phi_i=n_{i-1}\deg_z\phi_{i-1}$ for $i\ge 2$. Thus
 $$
[G_{\omega}:G_{V_0}]\delta(\omega/V_0)\le \deg_zf=n_1n_2\cdots n_{k-1}=[G_{\omega}:G_{V_0}]
 $$
 so that $\delta(\omega/\nu)=1$.
 
 Now suppose that $V_1,\ldots,V_k,\ldots$ is of infinite length. We have (by Lemma \ref{LemmaB6}) natural  extensions of valued fields
 $$
 (K,\nu)\rightarrow (K^h,\nu^h)\rightarrow (\hat K,\hat\nu).
 $$
 Let $\overline f(z)$ be the irreducible factor of $f(z)$ in $K^h[z]$ which induces $\omega$ (from Lemma \ref{Lemma2}). Then $\overline f(z)$ is irreducible in $\hat K[z]$  (by Lemma \ref{LemmaB6}) and so is the irreducible factor of $f(z)$ in $\hat K[z]$ which induces $\omega$ (by Lemma \ref{LemmaB3}). 
 Thus the pseudo valuation $W$ extends to a pseudo valuation $W^h$ of $K^h[z]$ and to a pseudo valuation $\hat W$ of $\hat K[z]$  such that $I(W^h)_{\infty}=(\overline f(z))$ in $K^h[z]$ and $I(\hat W)_{\infty}=(\overline f(z))$ in $\hat K[z]$. By (\ref{eqefd}), 
 \begin{equation}\label{B1}
 \delta(\omega/\nu)=[L^h:K^h]/[G_{\omega}:G_{\nu}]=\deg_z\overline f/[G_{\omega}:G_{\nu}].
 \end{equation}

There exists $k_0$ such that $\deg_z\phi_k=\deg_z\phi_{k_0}$ for $k\ge k_0$. 
There exist $a_i\in A_1$ and $j_1(i),\ldots,j_{k_0-1}(i)$ with $0\le j_l(i)<n_l$ for $1\le l\le k_0-1$ such that
$$
\phi_{k_0+i+1}=\phi_{k_0+i}-a_i\phi_1^{j_1(i)}\cdots\phi_{k_0-1}^{j_{k_0-1}(i)}
$$
for $i\ge 0$. 
Now 
$$
W(\phi_{k_0+i})=W(a_i\phi_1^{j_1(i)}\cdots\phi_{k_0-1}^{j_{k_0-1}(i)})
$$
for $i>0$ and 
\begin{equation}\label{B2}
W(\phi_{k_0+i})\mapsto \infty\mbox{ as }i\mapsto \infty
\end{equation}
by Lemma \ref{Lemma4}. Thus $\nu(a_i)\mapsto \infty$ as $i\mapsto \infty$.

For fixed $(b_1,\ldots,b_{k_0-1})$ such that $0\le b_l< n_l$ for $1\le l\le k_0-1$, define
$$
c_l(b_1,\ldots,b_{k_0-1})=\sum a_i,
$$
 where the sum is over $i<l$ such that 
$(j_1(i),\ldots,j_{k_0-1}(i))=(b_1,\ldots,b_{k_0-1})$.
Let 
$$
\tau_i=\sum _{b_1,\ldots,b_{k_0-1}}c_i(b_1,\ldots,b_{k_0-1})\phi_1^{b_1}\cdots\phi_{k_0-1}^{b_{k_0-1}}
$$
where the sum is over $b_1,\ldots, b_{k_0-1}$ such that $0\le b_j<n_j$ for $1\le j\le k_0-1$.
We have that $\phi_{k_0+i}=\phi_{k_0}-\tau_i$. Thus
$$
W(\tau_j-\tau_i)=W(\phi_{k_0+i}-\phi_{k_0+j})\ge\min\{W(\phi_{k_0+i}),W(\phi_{k_0+j})\}
$$
so $W(\tau_j-\tau_i)\mapsto \infty$ as $j\ge i\mapsto \infty$. We have that 
$$
W(\tau_j-\tau_i)=\min\{\nu(c_i(b_1,\ldots,b_{k_0-1})-c_j(b_1,\ldots,b_{k_0-1})) + W(\phi_1^{b_1}\cdots\phi_{k_0-1}^{b_{k_0-1}})\}
$$
where the minimum is over $b_1,\ldots,b_{k_0-1}$ with $0\le b_j<n_j$ for $1\le j\le k_0-1$.
So for all $b_1,\ldots,b_{k_0-1}$, $\nu(c_i(b_1,\ldots,b_{k_0-1})-c_j(b_1,\ldots,b_{k_0-1})) \mapsto \infty$ 
as $j\ge i\mapsto \infty$. Thus for each $b_1,\ldots,b_{k_0-1}$, $(c_i(b_1,\ldots,b_{k_0-1}))$ is a $\nu$-Cauchy sequence.  

%We must have that at least one of these sequences is not a null sequence. Otherwise, we find a contradiction, since from $\phi_{k_0+i}=\phi_{k_0}+\sum c_i(b_1,\ldots,b_{k_0-1})\phi_1^{b_1}\cdots\phi_{k_0-1}^{b_{k_0-1}}$ we would have that $W(\phi_{k_0+i})=W(\phi_{k_0})$ for $i\gg 0$.

Thus these sequences have limits in $\hat K$, and so $(\phi_{k_0+i})$ is a $\nu$-Cauchy sequence in $K[z]$ which has a non zero limit $\phi_{\infty}$ in $\hat K[z]$ ($\phi_{\infty}$ is necessarily unitary of degree equal to $\deg_z\phi_{k_0}$). Now $\phi_{\infty}\in I(\hat W)_{\infty}=(\overline f)$ by  (\ref{B2}). Thus $\deg_z\phi_{\infty}\ge \deg_z\overline f$.
Now
$$
\deg_z\overline f\le \deg_z\phi_{\infty}=\deg_z\phi_{k_0}=[G_{V_{k_0}}:{G_{V_0}}]=[G_{\omega}:G_{\nu}].
$$
Thus  $\deg_z\overline f=[G_{\omega}:G_{\nu}]$ and $\delta(\omega/\nu)=1$ by (\ref{B1}).
\end{proof}

\begin{Example}\label{Ex201}
The conclusions of Theorem \ref{Theorem4*} may fail if $f(z)$ is not separable over $K$.
\end{Example}

An example of F.K. Schmidt of a discrete valuation ring (with value group $\ZZ$) and an inseparable extension of  its quotient field which has defect is explained in \cite[Example 3.1]{K4}. The example is as follows. Let $\mathbf k$ be an algebraically closed field of characteristic $p>0$, $A=\mathbf k[x,y]_{(x,y)}$ be the localization of a polynomial ring in two variables  and $K$ be the quotient field of $A$. Let $\mathbf k[[t]]$ be a power series ring and let $s\in \mathbf k[[t]]$ be transcendental over $\mathbf k(t)$ and such that ${\rm ord}_t(s)>0$. The $\mathbf k$-algebra embedding $K\rightarrow \mathbf k[[t]]$ defined by $x\mapsto t$ and $y\mapsto s^p$ induces a valuation $\nu$ on $K$ which dominates $A$ by $\nu(g(x,y))={\rm ord}_t(g(t,s^p))$. We have that $G_{\nu}=\ZZ$ and $R_{\nu}/m_{\nu}=\mathbf k$. Let $f(z)=z^p-y\in K[z]$. There is a unique extension of $\nu$ to a valuation $\omega$ of $L=K[z]/(f(z))$ (since $L$ is purely inseparable over $K$) which is an  immediate extension of  $\nu$ ($G_{\omega}=G_{\nu}$ and $R_{\omega}/m_{\omega}=R_{\omega}/m_{\omega}$). Thus the defect $\delta(\omega/\nu)=\deg_zf=p$ by Ostrowski's lemma (\ref{eqN300}).  Since $\nu$ is a rank 1 discrete valuation, by MacLane's theorem (Section \ref{Sec2}), $\omega$ is a limit valuation which is realized by his algorithm. We will give an explicit construction. 

 Let $W$ be the pseudo valuation induced by $\omega$ on $K[z]$, and let $V_0=\nu$. We will construct a sequence of approximants $V_1,\ldots,V_i,\ldots$ to $f$ over $V_0$ which realize $W$.

 Expand
 $s=\sum_{i=1}^{\infty} a_it^i$ with $a_i\in \mathbf k$. We have that $s^p=\sum_{i=1}^{\infty}a_i^pt^{ip}$.
Define 
$$
\sigma(1)={\rm ord}_t(s)=\min\{i\mid a_i\ne 0\}
$$
 and
 for $j>1$, 
 $$
 \sigma(j)=\min\{i\mid \sigma(j-1)<i\mbox{ and }a_i\ne 0\}.
 $$
 
 The first approximant is $V_1=[V_0;V_1(\phi_1)=\sigma(1)]$ where $\phi_1=z$. For $i\ge 1$, $V_{i+1}$ is defined by 
 $V_{i+1}=[V_i;V_{i+1}(\phi_{i+1})=\sigma(i+1)]$, where $\phi_{i+1}=\phi_i-a_{\sigma(i+1)}x^{\sigma(i+1)}$.  Then 
 $\lim_{i\rightarrow \infty}V_i(\phi_i)=\infty$ and so $W$ is the limit valuation $W=\lim_{i\rightarrow\infty}V_i$ by Lemma \ref{Lemma4}.

\section{A Rank 1 Separable Example with Defect}\label{Sec8}

We consider an example from \cite[Theorem 7.38]{CP}, with regard to the algorithm of Section \ref{Sec3}. Let $\mathbf k$ be an algebraically closed field of characteristic $p>0$. Let $K=\mathbf k(u,v)$ be a two dimensional rational function field over $\mathbf k$, and, using the method of \cite{Sp} and \cite{CV1}, define a valuation $\nu$ of $K$ by the following generating sequence:
$$
P_0=u, P_1=v, P_2=v^{p^2}-u\mbox{ and }P_{i+1}=P_i^{p^2}-u^{p^{2i-2}}P_{i-1}\mbox{ for }i\ge 2.
$$
We normalize the valuation $\nu$ so that $\nu(u)=1$. We have  the defining relations $v^{p^2}\sim u$ in $\nu$ and $P_i^{p^2}\sim u^{p^{2i-2}}P_{i-1}$ in $\nu$ for $i \ge 2$.
As shown in \cite{CP}, the value group $G_{\nu}=\frac{1}{p^{\infty}}\ZZ=\bigcup_{i\ge 1}\frac{1}{p^i}\ZZ$. Let $f=x^p+ux^{p-1}-u\in K[x]$. By \cite[Theorem 7.38]{CP}, $\nu$ has a unique extension to a valuation $\omega$ of $L=K[x]/(f(z))$. Further, $\omega$ is an immediate extension, so it is a defect extension with $[L:K]=\delta(\omega/\nu)=[L:K]=p$.

Let $W$ be the pseudo valuation induced by $\omega$ on $K[x]$. We will construct a realization of the algorithm of Section \ref{Sec3}, giving an infinite sequence of approximants to $f$ over $V_0=\nu$,
$$
V_1,\ldots,V_k,\ldots
$$
satisfying (\ref{eq1}) and (\ref{eq21}) with $c_i\in R_{\nu}$ for all $i$. 

Setting $\phi_1=x$, we have that $N(V_0,\phi_1)$ has a single segment, which has the slope $\frac{V_0(u)}{p}=\frac{1}{p}$. Thus the first approximant to $f$ over $V_0$ is $V_1=[V_0;V_1(\phi_1)=\frac{1}{p}]$.

We will make use of the following observation when constructing our sequence of approximants. Suppose we have constructed the sequence $V_1,\ldots,V_k$ of approximants, where $\deg_x\phi_i=1$ for all $i$. Then for $i\le k$, $\phi_i=\phi_{i-1}+a_{i-1}$ with $a_{i-1}\in R_{\nu}$ and
$$
W(\phi_i)>W(\phi_{i-1})=\nu(a_{i-1}).
$$
Expanding
$$
ux^{p-1}=g_{p-1}\phi_k^{p-1}+g_{p-2}\phi_k^{p-2}+\cdots+g_0
$$
with $g_i\in R_{\nu}$, we have that
\begin{equation}\label{eqE1*}
V_0(g_0)\ge \min\{W(g_i\phi_k^i)\}=V_k(ux^{p-1})=W(ux^{p-1})=1+\frac{p-1}{p}>\frac{p^4}{p^4-1}.
\end{equation}
Now $f\sim \phi_1^p-u$ in $V_1$ and $u=v^{p^2}-P_2\sim v^{p^2}$ in $V_0$. Thus $f\sim (\phi_1-v^p)^p$ in $V_1$, and we take our second key polynomial to be $\phi_2=\phi_1-v^p=\phi_1-P_1^p$. We thus have that the second approximant is $V_2=[V_1;V_2(\phi_2)=W(\phi_2)]$.
Expanding
$$f=\phi_2^p+f_{p-1}\phi_2^{p-1}+\cdots+f_1\phi_2+f_0
$$
with $f_i\in R_{\nu}$, by (\ref{eqE1*}) with $k=2$, we have that 
$$
f_0=v^{p^2}-u+\mbox{ terms of value }>\frac{p^4}{p^4-1}
$$
and $v^{p_2}-u=P_2$. Now since $\omega$ is the unique extension of $\nu$, we have that the principal part of $N(V_1;\phi_2)$ is $N(V_1,\phi_2)$ and $N(V_1,\phi_2)$ has a single segment, which has slope 
$\frac{V_0(f_0)}{p}=\frac{V_0(P_2)}{p}=\frac{1}{p}+\frac{1}{p^5}$, which is less than $\frac{p^4}{p^4-1}$.
The second approximant is $V_2=[V_1;V_2(\phi_2)=\frac{1}{p}+\frac{1}{p^5}]$.

By (\ref{eqE1*}), we have that $f\sim \phi_2^p+P_2$ in $V_2$. Now $P_2\sim \frac{P_3^{p^2}}{u^{p^4}}$ in $V_0$, so $$
f\sim \left(\phi_2+\frac{P_3^p}{u^{p^3}}\right)^p\mbox{ in }V_2,
$$
and so $\phi_3=\phi_2+\frac{P_3^p}{u^{p^3}}$ is a key polynomial for $V_2$. We thus have that the third approximant is $V_3=[V_2;V_3(\phi_3)=W(\phi_3)]$. We expand
$$
f=\phi_3^p+f_{p-1}\phi_3^{p-1}+\cdots+f_1\phi_3+f_0
$$
with $f_i\in R_{\nu}$. By (\ref{eqE1*}), we have that 
$$
f_0=-\frac{P_3^{p^2}}{u^{p^4}}+P_2+\mbox{ terms of value }>\frac{p^4}{p^4-1}.
$$
Also, 
$$
-\frac{P_3^{p^2}}{u^{p^4}}+P_2=-\frac{P_4}{u^{p^4}}.
$$
Now since $\omega$ is the unique extension of $\nu$, we have that the principal part of $N(V_2,\phi_3)$ is $N(V_2,\phi_3)$ and $N(V_2,\phi_3)$ has a single segment, which has slope
$$
\frac{V_0(f_0)}{p}=\frac{1}{p}V_0\left(-\frac{P_4}{u^{p^4}}\right)=\frac{1}{p}+\frac{1}{p^5}+\frac{1}{p^9}<\frac{p^4}{p^4-1}.
$$
The third approximant is $V_3=[V_2;V_3(\phi_3)=\frac{1}{p}+\frac{1}{p^5}+\frac{1}{p^9}]$.

Continuing in this way, we construct an infinite sequence of approximants $V_1,\ldots,V_k,\ldots$ to $f$ over $V_0$ with
$V_k=[V_{k-1};V_k(\phi_k)=\mu_k]$ and for $k\ge 3$,
$$
\phi_k=\phi_{k-1}+(-1)^{k-1}\frac{P_{2(k-2)+1}^p}{u^{p^3+p^7+\cdots+p^{4(k-3)+3}}}
$$
with
$$
\mu_k=\frac{1}{p}+\frac{1}{p^5}+\frac{1}{p^9}+\cdots+\frac{1}{p^{4(k-1)+1}}.
$$
We have that 
\begin{equation}\label{eqEx10}
\lim_{k\rightarrow \infty}\mu_k=\frac{p^4}{p(p^4-1)}.
\end{equation}
 In particular, we have by Lemma \ref{Lemma4}, that the limit valuation $V_{\infty}=\lim_{k\rightarrow\infty}V_k$ is a valuation, and thus is not equal to $W$.

 We observe that there does not exist a birational extension $A_1$ of $\mathbf k[u,v]_{(u,v)}$ which is dominated by $V_0$ such that $\phi_i\in A_1[x]$ for all $i$, as there can only be finitely many values of elements in a  Noetherian local ring which is dominated by a rank 1 valuation that are less than a fixed finite bound.  

We now analyze the extension $W$ of $\nu$ in the context of Vaqui\'e's algorithm. We will construct an admissible family of valuations $\mathcal S$ which determines $W$.

In the above realization of the  algorithm of Section \ref{Sec3}, we started by defining    $\phi_1=x$, and $V_1=[V_0;V_1(\phi_1)=\frac{1}{p}]$.
With the notation of Section \ref{SecV}, we have 
$$
\Sigma(V_1)=\{x-g\mid g\in K\mbox{ and }W(x-g)>W(x)\}
$$
and 
$$
\Lambda(V_1)=\{W(\phi)\mid \phi\in \Sigma(V_1)\}.
$$
Let $(\mu_{\alpha})_{\alpha\in C}$ be the associated family of iterated augmented valuations to $\mathcal A=\{V_1\}$ of (\ref{eqV3}).   

The concept of distance of an element of $L$ from $K$ and the concepts of dependent and independent Artin-Schreier  extensions are introduced in \cite{K3}. In \cite{EG}, our extension $\omega$ of $\nu$ is analyzed, and it is shown that it is a dependent Artin-Schreier extension. We will  make use of a calculation in their proof, to determine $\limsup\{\Lambda(V_1)\}$. Suppose that $g\in \Sigma(V_1)$. Then $W(x-g)>W(x)$ so $\nu(g)=W(x)=\frac{1}{p}$. Thus $g\in R_{\nu}$, and by 2) of \cite[Theorem 4.4]{EG}, we have that
$$
W(g^p-x^p) \le 1+\frac{1}{p^4}+\cdots+\frac{1}{p^{4(k+1)}}
$$
for some $k\ge 0$. Thus 
$$
W(x-g)=\frac{1}{p}W(g^p-x^p)\le \frac{1}{p}+\frac{1}{p^5}+\cdots+\frac{1}{p^{4(k+1)+1}}<\frac{p^4}{p(p^4-1)}.
$$
By (\ref{eqEx10}), we have that $\limsup\{\Lambda(V_1)\}=\frac{p^4}{p(p^4-1)}$ and $\frac{p^4}{p(p^4-1)}\not\in \Lambda(V_1)$.
In particular, $\Lambda(V_1)$ does not have a largest element. Thus the first simple admissible family associated to $W$ is
$$
\mathcal S^{(1)}=\{V_1\}\cup \{(\mu_{\alpha})_{\alpha\in C}\}
$$
and $\mathcal S$ is the union of $t>1$ simple admissible families.  Since $\omega$ is an immediate extension of $\nu$, we have by (\ref{eqMV3}) that
\begin{equation}\label{eqLK1}
p=\deg_xf=s^{\rm tot}(\mathcal S)=\prod_{j=2}^ts^{(j-1)}(\mathcal S).
\end{equation}

Let $\psi_{\alpha}=x-\phi_{\alpha}\in K$ for $\alpha>1$. We have
$$
\nu(\psi_{\sigma}-\psi_{\rho})=W(\phi_{\sigma}-\phi_{\rho})=W(\phi_{\rho})=\mu_{\rho}<\mu_{\sigma}=W(\phi_{\sigma})=\nu(\psi_{\tau}-\psi_{\sigma})
$$
for $\rho<\sigma<\tau$. Thus $\{\psi_{\alpha}\}$ is a pseudo-convergent set in $K$ in the sense of Kaplansky \cite{Ka}. 

Let $g(x)$ be a limit key polynomial for $\{\phi_{\alpha}\}$ (defined in Section \ref{SecV}). 
As explained in \cite[Section 3]{NS}, $g(x)$ is a polynomial of smallest degree such that 
$g(\psi_{\alpha})<g(\psi_{\beta})$ for $\alpha<\beta$. By \cite[Lemma 10]{Ka}, the degree of $g$ is a power of $p$. By (\ref{eqLK1}), $g$ has degree $p$, and so $f$ is a limit key polynomial for $\{\phi_{\alpha}\}$. 
Thus $\phi_1^{(2)}=f$ and so $\mathcal S^{(2)}=\{\mu_1^{(2)}\}$ where $\mu_1^{(2)}$ is the limit augmented value 
$$
\mu_1^{(2)}=[(\mu_{\alpha})_{\alpha\in B};\mu_1^{(2)}(f)=\infty].
$$
In summary, our admissible family of valuations $\mathcal S$ which determine $W$ is
$$
\mathcal S=\mathcal S^{(1)}\cup \mathcal S^{(2)}
$$
where $\mathcal S^{(1)}$ and $\mathcal S^{(2)}$ are as described above. 

We now consider the key polynomials $\phi_i$ and valuations $V_i$ constructed in our realization of MacLane's algorithm. Since
$$
\limsup W(\phi_i)=\frac{p^4}{p(p^4-1)}=\limsup \Lambda(V_1)
$$
and $\phi_i\in \Sigma(V_1)$ for $i>1$, we have by Proposition 1.9 \cite{V1} that the limit valuations
$V_{\infty}=\lim_{i\rightarrow\infty}V_i$ and $\lim_{\alpha\in B}\mu_{\alpha}$ are equal. Thus the pseudo valuation $W$ satisfies
$$
W(g)=\left\{
\begin{array}{ll} \infty&\mbox{ if }f|g\mbox{ in }K[x]\\
V_{\infty}(g)&\mbox{ if }f\not| g\mbox{ in }K[x].
\end{array}
\right.
$$

%$$
%p=\deg_xf=s^{\rm tot}(\mathcal S)=\prod_{j=2}^ts^{(j-1)}(\mathcal S).
%$$
%Since the $s^{(j-1)}(\mathcal S)$ are positive rational mumbers with $s^{(j-1)}(\mathcal S)>1$ for $2\le j\le t$, we have that $t\ge 2$. 

%By (\ref{eqMV4}), the first key polynomial $\phi_1^{(2)}$ of $\mathcal S^{(2)}$ satisfies 
%$$
%\deg_x\phi_1^{(2)}=s^{(1)}(\mathcal S)\deg_x\phi_1=s^{(1)}(\mathcal S)=p.
%$$
%Thus $\phi_1^{(2)}=f$ and so $\mathcal S^{(2)}=\{\mu_1^{(2)}\}$ whre $\mu_1^{(2)}$ is the limit augmented valuation 
%$$
%\mu_1^{(2)}=[(\mu_{\alpha})_{\alpha\in B};\mu_1^{(2)}(f(z))=\infty].
%$$

%Thus an  admissible family of valuations $\mathcal S$ which determine $W$ starts with $\mathcal S^{(1)}$ and has at least one jump.

%We now consider the key polynomials $\phi_i$ and valuations $V_i$ constructed in our realization of the algorithm of Section \ref{Sec3}. Since
%$$
%\limsup W(\phi_i)=\frac{p^4}{p(p^4-1)}=\limsup \Lambda(V_1)
%$$
%and $\phi_i\in \Sigma(V_1)$ for $i>1$, we have by Proposition 1.9 \cite{V1} that the limit valuations
%$V_{\infty}=\lim_{i\rightarrow\infty}V_i$ and $\lim_{\alpha\in B}\mu_{\alpha}$ are equal. Thus the pseudo valuation $W$ satisfies
%$$
%W(g)=\left\{
%\begin{array}{ll} \infty&\mbox{ if }f|g\mbox{ in }K[x]\\
%V_{\infty}(g)&\mbox{ if }f\not| g\mbox{ in }K[x].
%\end{array}
%\right.
%$$

\section{A defectless extension of a rank two valuation with many jumps}\label{SecNEX}

In this section we construct the following example, which shows that the conclusions of Theorem \ref{Theorem4*} may not hold if $\nu$ has rank larger than one.

\begin{Example}\label{NEX} Let $\mathbf k$ be an algebraically closed field of characteristic not equal to 2, and let $\mathbf k[x,y]$ be a polynomial ring in two variables over $\mathbf k$. Let $K=\mathbf k(x,y)$ and let $\nu$ be the rank two valuation on $K$ defined by $\nu(x)=(0,1), \nu(y)=(1,0)\in (\ZZ_2)_{\rm lex}$ and $\nu |(\mathbf k\setminus 0)=0$. Let 
\begin{equation}\label{eqNEX10}
f=((z^2-x^2-x^3)^2-y^2(x^2+2x^3))^2-(y^6+y^7)
\end{equation}
and let $\omega$ be an extension of $\nu$ to $K[z]/(f(z))$. Let $W$ be the induced pseudo valuation of $K[z]$. 
Then  a family  $\mathcal S=\mathcal S^{(1)}\cup \cdots\cup S^{(t)}$  (with notation of Section \ref{SecV}) realizing $W$ has at least three jumps; that is, $t\ge 3$.
\end{Example}

We first establish that $f$ is irreducible in $K[z]$.  Setting $x=0$ in $f$, we obtain the reduction $\tilde f=z^8-(y^6+y^7)\in \mathbf k(y)[z]$. We have that 
$$
\tilde f=\prod_{j=0}^7(z-\tau^jy^{\frac{3}{4}}(1+y)^{\frac{1}{8}})
$$
over an algebraic closure of $\mathbf k(y)$, where $\tau$ is a primitive 8-th root of unity in $\mathbf k$. A unitary factor of $\tilde f$ of degree $r$ must have the constant term $\tau^s(y^{\frac{3}{4}}(1+y)^{\frac{1}{8}})^r$ for some $s\in \NN$. But $(y^{\frac{3}{4}}(1+y)^{\frac{1}{8}})^r\in \mathbf k(y)$ only if $r=8$, so $\tilde f$ is irreducible in $\mathbf k(y)[z]$, and thus $f$ is irreducible in $K[z]$.

 Henselization is discussed in Section \ref{Sec5}.

\begin{Lemma}\label{LemmaNEX1} 
The polynomial  $f$ factors into a product of linear unitary polynomials in $K^h[z]$, where $(K^h,\nu^h)$ is a Henselization of $(K,\nu)$.
\end{Lemma}

\begin{proof}
We will solve the equation $f(z)=0$ in $R_{\nu^h}$.
Let 
$$
Q=z^2-(x^2+x^3), U=Q^2-y^2(x^2+2x^3).
$$
 With these substitutions, the equation $f(z)=0$ becomes  $U^2=(y^6+y^7)$. Let $(1+y)^{\frac{1}{2}}$ be a square root of $1+y$  in the Henselization $A^h$ of $A=\mathbf k[x,y]_{(x,y)}$. Then $U=y^3(1+y)^{\frac{1}{2}}$ in $A^h$.  Thus we have that
$$
Q^2=y^3(1+y)^{\frac{1}{2}}+y^2(x^2+2x^3)
= y^2(x^2+2x^3+y(1+y)^{\frac{1}{2}}).
$$
Set $x=x_1$ and $y=x_1^2y_1$. We have that $x_1,y_1\in R_{\nu}$. Then
$$
Q^2=x_1^4y_1^2(x_1^2+2x_1^3+x_1^2y_1(1+x_1^2y_1)^{\frac{1}{2}})
= x_1^6y_1^2(1+2x_1+y_1(1+x_1^2y_1)^{\frac{1}{2}}).
$$
Let $(1+2x_1+y_1(1+x_1^2y_1)^{\frac{1}{2}})^{\frac{1}{2}}$ be a square root of $1+2x_1+y_1(1+x_1^2y_1)^{\frac{1}{2}}$ in the Henselization $A_1^h$ of $A_1=\mathbf k[x_1,y_1]_{(x_1,y_1)}$. Then
$$
Q=x_1^3y_1(1+2x_1+y_1(1+x_1^2y_1)^{\frac{1}{2}})^{\frac{1}{2}}
$$
in $A_1^h$. We now have that
$$
\begin{array}{lll}
z^2&=& Q+x^2+x^3\\
&=& x_1^3y_1(1+2x_1+y_1(1+x_1^2y_1)^{\frac{1}{2}})^{\frac{1}{2}}+x_1^2+x_1^3\\
&=& x_1^2(1+x_1+x_1y_1(1+2x_1+y_1(1+x_1^2y_1)^{\frac{1}{2}})^{\frac{1}{2}}).
\end{array}
$$

Let $(1+x_1+x_1y_1(1+2x_1+y_1(1+x_1^2y_1)^{\frac{1}{2}})^{\frac{1}{2}})^{\frac{1}{2}}$ be a square root of $$
1+x_1+x_1y_1(1+2x_1+y_1(1+x_1^2y_1)^{\frac{1}{2}})^{\frac{1}{2}}
$$
 in  $A_1^h$. Then
$$
z=x_1(1+x_1+x_1y_1(1+2x_1+y_1(1+x_1^2y_1)^{\frac{1}{2}})^{\frac{1}{2}})^{\frac{1}{2}}
\in A_1^h\subset R_{\nu^h}
$$
by Lemma \ref{Lemma3}.

Since all eight roots of $f(z)$ can be found this way, by making different choices of square roots, we have the desired factorization of $f(z)$ in $K^h[z]$ into a product of linear polynomials. 
\end{proof}

By Lemma \ref{Lemma2}, $\omega$ is the restriction to $K[z]/(f(z))$ of the extension of $\nu^h$ to a valuation $\omega^h$ of $K^h[z]/(\overline f)$ for some factor $\overline f$ of $f$ in $K^h[z]$.  Since $\overline f$ is a linear polynomial by Lemma \ref{LemmaNEX1}, we have that 
\begin{equation}\label{eqNEX20}
[G_{\omega}:G_{\nu}]=[G_{\omega_h}:G_{\nu^h}]=\deg_z\overline f=1
\end{equation}
 by (\ref{eqMV3}). 
 
We will require the following remark.

\begin{Remark}\label{RemarkNEX} An element $g\in \mathbf k(z)$ is a square of an element of $\mathbf k(z)$ if and only if 
all zeros and poles of $g(z)$ in $\AA^1_{\mathbf k}$ have even order.
\end{Remark}

The remark follows since every element $g(z)$ of $\mathbf k(z)$  has a unique factorization
$$
g(z)=c(z-a_1)^{n_1}\cdots (z-a_t)^{n_t}
$$
with $c\in \mathbf k$, $a_1,\ldots,a_t$ distinct elements of $\mathbf k$  and $n_1,\ldots,n_t$ nonzero integers.

We now turn to the construction of the family $\mathcal S$. We will use the notation of Section \ref{SecV}.
To begin with, we observe that the total jump $s^{\rm tot}(\mathcal S)$ of $\mathcal S$ satisfies
\begin{equation}\label{eqNEX21}
s^{\rm tot}(\mathcal S)=\deg_zf(z)=8
\end{equation}
by  (\ref{eqMV3}) and (\ref{eqNEX20}).

Let $V_0=\nu$. Since $W(f(z))=\infty$, we have that $W(z)=(0,1)$ and so the first approximant is $V_1=[V_0;V_1(z)=(0,1)]$. 
As above, let  $Q=z^2-(x^2+x^3)$. Since $W(f(z))=\infty$, we have that $W(Q)=(1,1)$. Let $\sum _{i=1}^{\infty}\alpha_ix^i$ with $\alpha_i\in k$ be a square root of $x^2+x^3=x^2(1+x)$ in $k[[x]]$. Let $\overline z=z-(\alpha_1x+\cdots+\alpha_nx^n)$ for some $n\in \ZZ_+$. Then
$$
\begin{array}{lll}
Q&=& (\overline z +\alpha_1x+\cdots+\alpha_nx^n)^2-(x^2+x^3)\\
&=& \overline z^2+2(\alpha_1x+\cdots+\alpha_nx^n)\overline z+(\alpha_1x+\cdots+\alpha_nx^n)^2-(x^2+x^3)
\end{array}
$$
so that $W(\overline z(\overline z+2(\alpha_1x+\cdots+\alpha_nx^n))> (0,n)$.
Thus 
\begin{equation}\label{eqNEX4}
W(z-\alpha_1x-\cdots-\a_nx^n)> (0,\frac{n}{2})\mbox{ or }
W(z+\alpha_1x+\cdots+\a_nx^n)> (0,\frac{n}{2}).
\end{equation}
Thus $d(V_1)=1$ and so
$$
\Sigma(V_1)=\{\phi=z+h\mid h\in K\mbox{ and }V_1(\phi)<W(\phi)\}.
$$
We will show that
\begin{equation}\label{eqNEX1}
\Lambda(V_1)=\{W(\phi)\mid \phi\in \Sigma(V_1)\}\subset \{0\}\times \ZZ_+.
\end{equation}

We now prove equation (\ref{eqNEX1}). Suppose there exists $h\in K$ such that setting $\phi=z+h$, we have that $W(\phi)\ge (1,0)$.  Then
\begin{equation}\label{eqNEX2}
W(h)=(0,1).
\end{equation}
Substituting into $Q$,  we have that
$Q=\phi^2-2h\phi + h^2-(x^2+x^3)$.
Now $W(Q)=(1,1)$ implies
\begin{equation}\label{eqNEX3}
W(h^2-(x^2+x^3))\ge (1,0).
\end{equation}
By (\ref{eqNEX2}), we have an expression
$$
h=\frac{\alpha_0(x)+y\Omega_1}{\beta_0(x)+y\Omega_2}
$$
with $\alpha(x),\beta(x)\in \mathbf k[x]$ nonzero and $\Omega_1,\Omega_2\in \mathbf k[x,y]$. Now substituting into (\ref{eqNEX3}), we have that
$$
W((\alpha_0(x)+y\Omega_1)^2+(x^2+x^3)(\beta_0(x)+y\Omega_2)^2)\ge (1,0)
$$
which implies
$$
W\left(\left(\frac{\alpha_0(x)}{\beta_0(x)}\right)^2-(x^2+x^3)\right)\ge (1,0)
$$
so that
$$
\left(\frac{\alpha_0(x)}{\beta_0(x)}\right)^2=x^2+x^3,
$$ 
a contradiction by Remark \ref{RemarkNEX}. Thus (\ref{eqNEX1}) holds. 

Let
$$
\mathcal A=\{\mu_{\alpha}=[V_1;\mu_{\alpha}(\phi_{\alpha})=W(\phi_{\alpha})\mid \phi_{\alpha}\in \Sigma(V_1)\}.
$$

By (\ref{eqNEX4}) and (\ref{eqNEX1}), we have that $(1,0)$ is the least upper bound of $\Lambda(\mathcal A)$ in $(\ZZ^2)_{\rm lex}$ but $(1,0)\not\in \Lambda(\mathcal A)$.  Thus $\mathcal A$ does not have a maximal element.

Suppose that $\mu_{\alpha}\in \mathcal A$.  Then $\mu_{\alpha}=[V_1;\mu_{\alpha}(\phi_{\alpha})=W(\phi_{\alpha}]$ with $\phi_{\alpha}=z+h$ for some $h\in K$. Expand
$$
Q=\phi_{\alpha}^2-2h\phi_{\alpha}+(h^2-(x^2+x^3)),
$$
so that
$$
\mu_{\alpha}(Q)\le 2\mu_{\alpha}(\phi_{\alpha}),
$$
and $\mu_{\alpha}(Q)<(1,0)$ by (\ref{eqNEX1}). Thus $Q\in \tilde\Sigma(\mathcal A)$, and since $Q$ has the smallest possible degree that a polynomial  in $\tilde\Sigma(\mathcal A)$ can have (it must have degree greater than $1=d(V_1)$) we have that $d(\mathcal A)=2$ and $Q\in \Sigma(\mathcal A)$, and so  $Q$ is a limit key polynomial for $\mathcal A$. Let $V_2=[\mathcal A;V_2(Q)=(1,1)]$. Then the first simple admissible family in $\mathcal S$ is $\mathcal S^{(1)}=\{V_1\}\cup \{\mathcal A\}$, and the second admissible family $\mathcal S^{(2)}$ begins with $V_2$. Thus the first jump in $\mathcal S$ is 
$$
s^{(1)}(\mathcal S)=\frac{\deg_zQ}{\deg_zz}=2.
$$
We have that $f=(Q^2-y^2(x^2+2x^3))^2-(y^6+y^7)$. Let 
$$
U=Q^2-y^2(x^2+2x^3)
$$
as above.
 We have that $W(U)=(3,0)$  since $W(f(z))=\infty$. Let $\sum_{i=1}^{\infty}\beta_ix^i$ with $\beta_i\in \mathbf k$ be a square root of $x^2+2x^3=x^2(1+x)$ in $\mathbf k[[x]]$. For $n\in \ZZ_+$, let  $\overline Q=Q-y(\beta_1 x+\cdots+\beta_nx^n)$. Then 
$$
U=\overline Q^2+2y(\beta_1x+\cdots+\beta_nx^n)\overline Q+y^2(\beta_1x+\cdots+\beta_nx^n)^2-y^2(x^2+2x^3),
$$
so that
$$
W(\overline Q(\overline Q+2y(\beta_1x+\cdots+\beta_nx^n)))> (2,n).
$$
Thus
\begin{equation}\label{eqNEX5}
W(Q-y(\beta_1x+\cdots+\beta_nx^n))> (1,\frac{n}{2})\mbox{ or }
W(Q+y(\beta_1x+\cdots+\beta_nx^n))> (1,\frac{n}{2}).
\end{equation}
Thus $d(V_2)=2$ and so
$$
\Sigma(V_2)=\{\phi=Q+Az+B\mid A,B\in K\mbox{ and }V_2(\phi)<W(\phi)\}.
$$
We will  show that 
\begin{equation}\label{eqNEX6}
\Lambda(V_2)=\{W(\phi)\mid \phi\in \Sigma(V_2)\}\subset \{1\}\times\ZZ_+.
\end{equation}
We now prove equation (\ref{eqNEX6}). Suppose there exist $A.B\in K$ such that setting 
$$
\phi=Q+Az+B,
$$
 we have that $W(\phi)\ge (2,0)$. We have that $W(Q)=W(Az+B)$. Expand
$$
U=\phi^2-2(Az+B)\phi+(Az+B)^2-y^2(x^2+2x^3).
$$
Now $W(\phi^2)\ge (4,0)$ and $W((Az+B)\phi)>(3,0)$. Since $W(U)=(3,0)$, we have that
$$
W((Az+B)^2-y^2(x^2+2x^3))\ge (3,0).
$$
Thus
$$
(1,1)=W((Az+B))=\min\{W(A)+(0,1),W(B)\}.
$$
We can thus write
$$
A=y\left(\frac{\alpha_0(x)+y\Omega_1}{\gamma_0(x)+y\Omega_3}\right),
B=y\left(\frac{\beta_0(x)+y\Omega_2}{\gamma_0(x)+y\Omega_3}\right)
$$
with $\Omega_1,\Omega_2,\Omega_3\in \mathbf k[x,y]$, $\gamma_0(x)\ne 0$ and at least one of $\alpha_0(x),\beta_0(x)\ne 0$. Thus
$$
W([(\alpha_0(x)+y\Omega_1)z+(\beta_0(x)+y\Omega_2)]^2-(\gamma_0(x)+y\Omega_3)^2(x^2+2x^3))
\ge (1,0),
$$
and so 
$$
\begin{array}{lll}
(1,0)&\le& W((\alpha_0(x)z+\beta_0(x))^2-\gamma_0(x)^2(x^2+2x^3))\\
&=&W(\alpha_0(x)^2z^2+2\alpha_0(x)\beta_0(x)z+\beta_0(x)^2-\gamma_0(x)^2(x^2+2x^3))\\
&=& W(\alpha_0(x)^2Q+2\alpha_0(x)\beta_0(x)z+(\alpha_0(x)^2(x^2+x^3)+\beta_0(x)^2-\gamma_0(x)^2(x^2+2x^3))).
\end{array}
$$
Thus 
$$
W(2\alpha_0(x)\beta_0(x)z+(\alpha_0(x)^2(x^2+x^3)+\beta_0(x)^2-\gamma_0(x)^2(x^2+2x^3)))\ge (1,0).
$$
But this implies that
\begin{equation}\label{eqNEX7}
\alpha_0(x)\beta_0(x)=0
\end{equation}
by (\ref{eqNEX1}) and thus
\begin{equation}\label{eqNEX8}
\alpha_0(x)^2(x^2+x^3)+\beta_0(x)^2-\gamma_0(x)^2(x^2+2x^3)=0.
\end{equation}
We have that $\alpha_0(x)=0$ or $\beta_0(x)=0$ by (\ref{eqNEX7}).  If $\alpha_0(x)=0$, then (\ref{eqNEX8}) becomes
$$
\left(\frac{\beta_0(x)}{\gamma_0(x)}\right)^2=x^2+2x^3
$$
which is not a square in $\mathbf k(x)$ by Remark \ref{RemarkNEX}, giving a contradiction.
If $\beta_0(x)=0$, then (\ref{eqNEX8}) becomes 
$$
\left(\frac{\alpha_0(x)}{\gamma_0(x)}\right)^2=\frac{x+2}{x+1},
$$
again giving a contradiction by Remark \ref{RemarkNEX}. Thus (\ref{eqNEX6}) holds.

Set 
$$
\mathcal B=\{\nu_{\beta}=[V_2;\nu_{\beta}(\phi_{\beta})=W(\phi_{\beta})]\mid \phi_{\beta}\in \Sigma(V_2)\}.
$$
Suppose $\nu_{\beta}\in \mathcal B$. Then $\nu_{\beta}=[V_2;\nu_{\beta}(\phi_{\beta})=W(\phi_{\beta})]$ with $\phi_{\beta}=Q+Az+B$ for some $A,B\in K$. Expand
$$
U=\phi_{\beta}^2-2(Az+B)\phi_{\beta}+(Az+B)^2-y^2(x^2+2x^3)
$$
to see that $\nu_{\beta}(U)\le 2\nu_{\beta}(\phi_{\beta})$, and thus 
$\nu_{\beta}(U)<(3,0)$ by (\ref{eqNEX6}). Thus $U\in \tilde\Sigma(\mathcal B)$. We thus have that $d(\mathcal B)=4$ or $d(\mathcal B)=3$.

Let $\psi\in \Sigma(\mathcal B)$, and define $V_3=[\mathcal B;V_3(\psi)=W(\psi)]$.  Then the second admissible family in $\mathcal S$ begins with $V_3$. Thus the second jump is
$$
s^{(2)}(\mathcal S)=\frac{\deg_z\psi}{\deg_zQ}=\frac{3}{2}\mbox{ or }2.
$$
Thus 
$$
s^{(1)}(\mathcal S)s^{(2)}(\mathcal S)\le 4<8=s^{\rm tot}(\mathcal S)
$$
so there must be at least one more jump in the construction of $\mathcal S$ so that $t\ge 3$.

\section{Extensions of associated graded rings and semigroups}\label{Sec9} 

We will consider in this section the conditions of finite generation of extensions of associated graded rings along a valuation and relative finite generation of extensions of valuation semigroups.

In this section, we will have the following assumptions. 
 Suppose that $A$ is a Noetherian local domain which contains an algebraically closed field $\mathbf k$ such that $A/m_A\cong \mathbf k$. Let $K$ be the quotient field of $A$ and suppose that $\nu$ is a rank 1 valuation of $K$ which dominates $A$, such that the residue field of the valuation ring of $\nu$ is $\mathbf k$.
 
 Suppose that $S$ is a sub semigroup of a semigroup $T$. We say that $T$ is a finitely generated module over $S$ if there exists a finite number of elements $t_1,\ldots,t_r$ of $T$ such that 
 $$
 T=(t_1+S)\cup \cdots \cup (t_r+S).
 $$
 With our assumptions, ${\rm gr}_{\nu}(A)$ is isomorphic to the semigroup algebra $\mathbf k[t^{S^A(\nu)}]$. Thus if $A\rightarrow B$ is an inclusion of domains and $\omega$ is an extension of $\nu$ to the quotient field of $B$ which is nonnegative on  $B$ such that the residue field of $\omega$ is $\mathbf k$, then ${\rm gr}_{\omega}(B)$ is a finitely generated ${\rm gr}_{\nu}(A)$-module if and only if $S^B(\omega)$ is a finitely generated module over $S^A(\nu)$.

 We have the following  immediate corollary  of Theorem \ref{Theorem1}.

\begin{Corollary}\label{Cor1}   Suppose that 
$f(z)\in A[z]$ is unitary  and irreducible and there is a unique extension of $\nu$ to a valuation $\omega$ of $K[z]/(f(z))$ and  the characteristic $p$ of $\mathbf k$ does not divide $\deg_zf(z)$.  Then
${\rm gr}_{\omega}(A[z]/(f(z)))$ is a finitely generated ${\rm gr}_{\nu}(A)$-module and $S^{A[z]/(f(z))}(\omega)$ is a finitely generated module over the semigroup $S^{A}(\nu)$.
 \end{Corollary}
 
 The following corollary addresses the case when the extension of valuations is not unique. It is an immediate corollary of Theorem \ref{Theorem4*}.
 
 \begin{Corollary}\label{Cor2}   
Further suppose that $A$ is a Nagata ring.  Suppose that 
$f(z)\in A[z]$ is unitary, irreducible and separable and $\omega$ is  a valuation  of $K[z]/(f(z))$  which extends 
$\nu$ and  there is no  defect in the extension ($\delta(\omega/\nu)=1$).
Then  there exists a birational extension $A_1$ of $A$ which is dominated by $\nu$ such that
 ${\rm gr}_{\omega}(A_1[z]/(f(z)))$ is a finitely generated ${\rm gr}_{\nu}(A_1)$-module and $S^{A_1[z]/(f(z))}(\omega)$ is a finitely generated module over the semigroup $S^{A_1}(\nu)$.
\end{Corollary} 

If we remove any of the assumptions of Corollary \ref{Cor1},  then the conclusions of the corollary are false, as is shown in the following three examples.
We consider finite extensions $A\rightarrow B$ where $A$ and $B$ are excellent, $B$ is a domain with quotient field $L$ and $\omega$ is an extension of $\nu$ to $L$ which dominates $B$.

\begin{Example}\label{IEx1} There exists a finite extension $A\rightarrow B$ such that $\omega$ is the unique extension of $\nu$ to $L={\rm QF}(B)$, $p$ does not divide $[L:K]$ but
${\rm gr}_{\omega}(B)$ is not a finitely generated ${\rm gr}_{\nu}(A)$-module and $S^{B}(\omega)$ is  not a finitely generated module over the semigroup $S^{A}(\nu)$.\end{Example}
 
 In particular, the representation of $B$ as a ``hypersurface singularity'' over $A$ is essential to the conclusions of Theorem \ref{Theorem1} and Corollary \ref{Cor1}.

\begin{Example}\label{IEx2} There exists an extension $A\rightarrow B=A[z]/(f(z))$ where $f(z)$ is unitary and irreducible, such that $p$ does not divide $\deg_zf(z)$ but the extension $\omega$ of $\nu$ to a valuation of $L={\rm QF}(B)$ is  not unique such that 
${\rm gr}_{\omega}(B)$ is not a finitely generated ${\rm gr}_{\nu}(A)$-module and $S^B(\omega)$ is a not a finitely generated module over the semigroup $S^A(\nu)$.
\end{Example}

Example \ref{IEx2} shows that the condition that $\omega$ is the unique extension of $\nu$ is necessary in Theorem \ref{Theorem1} and Corollary \ref{Cor1}, and that the birational extension $A\rightarrow A_1$ in the conclusions of  Corollary \ref{Cor2} is necessary.

\begin{Example}\label{IEx3} There exists an extension $A\rightarrow B=A[z]/(f(z))$ where $f(z)$ is unitary and irreducible, such that  the extension $\omega$ of $\nu$ to a valuation of $L={\rm QF}(B)$ is  unique but $p$ divides $\deg_zf(z)$ such that ${\rm gr}_{\omega}(B)$ is not a finitely generated ${\rm gr}_{\nu}(A)$-module and $S^B(\omega)$ is not a finitely generated module over $S^A(\nu)$.  In the example, $\delta(\omega/\nu)=1$.
\end{Example}  

Example \ref{IEx3} shows that the condition that $p\not|\deg_zf(z)$ is necessary in  Corollary \ref{Cor1}.

In the remainder of this section, we will construct these three examples.

Examples \ref{IEx1} and \ref{IEx2} will be obtained from Example 9.3 of \cite{CV1}.
In \cite[Example 9.3]{CV1}, $\mathbf k$ is an arbitrary field. We will make the further restriction that $\mathbf k$ is an algebraically closed field of characteristic $p>2$.  Let $T=\mathbf k[x,y]_{(x,y)}$, a localization of a polynomial ring in two variables, and $R$ be the subring $R=\mathbf k[x^2,xy,y^2]_{(x^2,xy,y^2)}$. Let $\omega$ be the rational rank 1 valuation dominating $T$ which is determined by the generating sequence
$$
P_0=x,P_1=y,P_2=y^3-x^5
$$
and
$$
P_{i+1}=P_i^3-x^{a_i}P_{i-1}\mbox{ for $i\ge 2$}
$$
where $a_i$ is even, and chosen so that $S^{T}(\omega)$ is not a finitely generated module over $S^{R}(\nu)$,
where $\nu$ is the restriction of $\omega$ to the quotient field $M$ of $R$. Let $N$ be the quotient field of $T$. 

Since the characteristic of $\mathbf k$ is not equal to 2, $N$ is Galois over $M$, and the Galois group is generated by the involution $\sigma$ defined by $\sigma(x)=-x$ and $\sigma(y)=-y$. Given $0\ne g\in T$, we expand
$$
g=\sum \alpha_{i_0,i_1,\ldots,i_r}P_0^{i_0}P_1^{i_1}\cdots P_r^{i_r}
$$
with $\alpha_{i_0,i_1,\ldots,i_r}\in \mathbf k$, $i_0\in \NN$ and $0\le i_j<3$ for $1\le j$, so that 
$$
\omega(g)=\min\{i_0\omega(P_0)+i_1\omega(P_1)+\cdots+i_r\omega(P_r)\mid \alpha_{i_0,i_1,\ldots,i_r}\ne 0\}.
$$
Then 
$$
\sigma(g)=\sum \alpha_{i_0,i_1,\ldots,i_r}(-1)^{i_0+i_1+\cdots+i_r}P_0^{i_0}P_1^{i_1}\cdots P_r^{i_r}
$$
and thus $\omega(\sigma(g))=\omega(g)$.  Since the extensions of a valuation in a finite Galois extension are conjugate (\cite[Corollary 3 to Theorem 12, page 66]{ZS2}), we have that $\omega$ is the unique extension of $\nu$ to $N$.

We now give a direct verification  that $T$ is not isomorphic to $R[z]/(f(z))$ for some $f(z)\in R[z]$. This follows since for a maximal ideal $m$ in $R[z]/(f(z))$, we have that 
$$
\dim_{\mathbf k}m/m^2\ge 3>2=\dim_{\mathbf k}m_T/m_T^2.
$$
We thus have that $R\rightarrow T$ gives Example \ref{IEx1}.

In \cite[Example 9.4]{CV1}, it is shown that in the natural extension $S\rightarrow T$, where $S=\mathbf k[u,v]_{(u,v)}$
and $u=x^2, v=y^2$, with valuation $\mu$ obtained by restricting $\omega$ to the quotient field of $S$, that $S^T(\omega)$ is not a finitely generated $S^{S}(\mu)$-module.  Now we have a factorization of our extension
$S\rightarrow U\rightarrow T$ where  $U=\mathbf k[x,v]_{(x,v)}$. Now $U\cong S[z]/(z^2-u)$ and $T\cong U[z]/(z^2-v)$. 
Let $\tau$ be the restriction of $\omega$ to the quotient field $L$ of $U$.

Now we must have that $S^{U}(\tau)$ is not a finitely generated $S^{S}(\mu)$-module or $S^{T}(\omega)$ is not a finitely generated $S^{U}(\tau)$-module since $S^{T}(\omega)$ is not a finitely generated $S^{S}(\mu)$-module.

We necessarily have by Corollary \ref{Cor1} that either $\tau$ is not the unique extension of $\mu$ to $L$ or $\omega$ is not the unique extension of $\tau$ to $N$, giving Example \ref{IEx2}.

In \cite{D}, a general theory of eigenfunctions for a valuation is developed for two dimensional quotient singularities, and a complete characterization is given of when the resulting extension of associated graded rings along the valuation is finite. 

We now construct Example \ref{IEx3}. Let $A=\mathbf k[u,v]_{(u,v)}$ with quotient field $K$ and let $\nu$ be the valuation of $K$  which dominates $A$ constructed in \cite[Theorem 7.38]{CP} and analyzed in Section \ref{Sec8}. Let $f(x)=x^p+ux^{p-1}-u$. It is shown in Theorem 7.38 \cite{CP} that there is a unique extension of $\nu$ to a valuation $\omega$ of $L=K[x]/(f(x))$. The extension is immediate, with defect $\delta(\omega/\nu)=p$. 
Let $B=A[x]/(f(x))$.  

We see from the generating sequence $P_0,\ldots,P_i,\ldots$ recalled in the beginning of Section \ref{Sec8} that
${\rm gr}_{\nu}(A)\cong \mathbf k[\overline P_0,\overline P_1,\ldots]/I$ where
$$
I=(\overline P_1^{p^2}-\overline P_0,\overline P_{i}^{p^2}-\overline P_0^{p^{2i-2}}\overline P_{i-1}\mbox{ for }i\ge 2).
$$
It is shown in formulas (35) and (36) of \cite{C2} that
$$
U_0=x,U_1=v,U_2=v^p-x
$$
and for $j\ge 2$,
$$
U_{j+1}=U_j^p-x^{p^{2j-2}}U_{j-1}\mbox{ if $j$ is odd},
$$
$$
U_{j+1}=U_j^{p^3}-x^{p^{2j-1}}U_{j-1}\mbox{ if $j$ is even}
$$
is a generating sequence for $\omega$ in $B$.  Thus 
${\rm gr}_{\nu}(B)\cong \mathbf k[\overline U_0,\overline U_1,\ldots]/J$ where
$$
J=(\overline U_1^{p}-\overline U_0,\overline U_i^p-\overline U_0^{p^{2i-2}}\overline U_{i-1}\mbox{ for $i\ge 2$ odd},
\overline U_{i}^{p^3}-\overline U_0^{p^{2i-1}}\overline U_{i-1}\mbox{ for $i\ge 2$ even}).
$$
Thus
$\overline U_n^p=\overline P_n$ if $n$ is even and $\overline U_n=\overline P_n$ if $n$ is odd, and so 
${\rm gr}_{\omega}(B)$ is not a finitely generated ${\rm gr}_{\nu}(A)$-module and $S^{\omega}(B)$ is not a finitely generated  $S^{\nu}(A)$-module. 

%\textcolor{blue}{It seems unlikely that the conclusions of Corollary \ref{Cor2} are true if the defect $\delta(\omega/\nu)>1$. Some results indicating this are in \cite{CP}, \cite{C2} and \cite{C3}.}

\end{document}